\documentclass[12pt,final]{amsart}
\usepackage[margin=1in,includeheadfoot]{geometry}
\usepackage{amsmath}

\usepackage{graphicx}
\usepackage{stmaryrd}
\usepackage{color}
\usepackage[dvipsnames]{xcolor}

\usepackage[colorlinks=true,
            linkcolor=violet,
            urlcolor=Aquamarine,
            citecolor=YellowOrange]{hyperref}

\newtheorem{thm}{Theorem}[section]
\newtheorem{cor}[thm]{Corollary}
\newtheorem{lem}[thm]{Lemma}
\newtheorem{prop}[thm]{Proposition}
\newtheorem{defn}[thm]{Definition}
\newtheorem{rem}[thm]{Remark}
\newtheorem{que}[thm]{Question}

\definecolor{darkgreen}{rgb}{0,0.5,0}
\definecolor{darkred}{rgb}{0.7,0,0}

\def\ba{\begin{array}}
\def\ea{\end{array}}
\def\be{\begin{equation}}
\def\ee{\end{equation}}
\def\bee{\begin{eqnarray}}
\def\beee{\begin{eqnarray*}}
\def\eee{\end{eqnarray}}
\def\eeee{\end{eqnarray*}}

\title[ Quantization for biharmonic maps ]{
\bf Quantization for biharmonic maps from  non-collapsed degenerating Einstein 4-manifolds
}

\author{Youmin Chen}
\address{School of Mathematical Sciences, Shanghai Jiao Tong University\\ 800 Dongchuan Road \\ Shanghai, 200240 \\P. R. China}%
\email{chyoumin19@sjtu.edu.cn}%
\author{Miaomiao Zhu}
\address{School of Mathematical Sciences, Shanghai Jiao Tong University\\ 800 Dongchuan Road \\ Shanghai, 200240 \\P. R. China}%
\email{mizhu@sjtu.edu.cn}%
\date{\today}

\begin{document}
\maketitle

\begin{abstract}
For a sequence of extrinsic or intrinsic biharmonic maps $u_j: M_j\rightarrow N$ from a sequence of non-collapsed degenerating closed Einstein 4-manifolds $(M_j,g_j)$ with bounded Einstein constants, bounded diameters and bounded $L^2$ curvature energy into a compact Riemannian manifold $(N,h)$ with uniformly bounded biharmonic energy, we establish a compactness theory modular finitely many bubbles, which are finite energy biharmonic maps from $\mathbb{R}^4$, or from $\mathbb{R}^4 / \Gamma$ for some nontrivial finite group $\Gamma \subset SO(4)$, or from some complete, noncompact, Ricci flat, non-flat ALE 4-manifold (orbifold). To achieve this, we develop a sophisticated asymptotic analysis for solutions over degenerating neck regions.


\end{abstract}

\

\section{Introduction}

There are many geometric variational problems which are critical in 4 dimensions and the moduli space of  solutions of the Euler-Lagrange equations (in particular, the minimizers of the action functional) possess rich structures when the underlying domain manifold is of dimension 4.  Careful investigation of the moduli space usually leads to profound geometric and topological applications, for instance, the theories of Yang-Mills fields (in particular, ASD connections) and Seiberg-Witten equations, etc.

In the present paper, we shall develop a new scheme in geometric analysis, to lay the ground for investigating the compactness of moduli space of solutions of certain geometric PDEs defined over 4-manifolds equipped with varying geometric structures. We shall consider a 4th order nonlinear geometric system that is critical in dimension 4, namely, the biharmonic map system, which is analytically complicated and subtle to handle. For the underlying domains, we shall consider non-collapsed closed Einstein 4-manifolds with varying metrics. The scheme developed in this paper can be applied to many other systems arising from geometry and physics, for instance, some lower order systems like Yang-Mills fields (including ASD connections) and Seiberg-Witten equations, which will be investigated in future works.

The theory of biharmonic functions dates back to 1862, when Maxwell and Airy studied it to describe a mathematical model of elasticity. In 3 dimensions, biharmonic functions can also be used to represent the velocity and pressure fields in the 3 dimensional Stokes flow (also called creeping flow), which is a type of fluid flow where advective inertial forces are small compared with viscous forces (namely, the Reynolds number is low), see e.g. \cite{HB1965}. Biharmonic maps between Riemannian manifolds are natural generalizations of biharmonic functions and harmonic maps, and it is expected that exploration of the solution space of biharmonic maps should lead to some geometric and topological applications. Here we recall some basic notations and some geometric analysis aspects for biharmonic maps.

Let $(M,g)$ be a smooth Riemannian manifold and let $(N,h)$ be a compact smooth Riemannian manifold which is embedded in some Euclidean space $\mathbb{R}^ K$. Consider the following functionals for a map $u$ from $(M,g)$ to $(N,h)$,
\begin{equation*}
H(u)=\frac{1}{4}\int_{M}|\triangle_g u|^2 dV_g
\end{equation*}
and
\begin{equation*}
T(u)=\frac{1}{4}\int_{M}|\tau(u)|^2 dV_g.
\end{equation*}
Here $\tau(u)$ is the tension field of $u$, or equivalently, the tangential part of $\triangle_g u$.
Critical points of $H(u)$ are called \emph{extrinsic biharmonic maps}, because the functional $H(u)$ depends on the specific embedding of $(N,h)$ into the Euclidean space. Critical points of $T(u)$ are called \emph{intrinsic biharmonic maps}. The corresponding Euler-Lagrange equations are 4th order nonlinear elliptic systems which are critical in dimension 4 (see Section \ref{Preliminaries111}). When the target is an Euclidean space, then biharmonic maps reduce to biharmonic functions.  In the sequel, if not specified, we shall mainly consider extrinsic biharmonic maps and omit the word extrinsic. The case of intrinsic biharmonic maps will be discussed in Section \ref{secintrinsic}.

When the domain is of dimension 4, the above two functionals and the corresponding Euler-Lagrange equations are scaling invariant, therefore, in general, strong compactness for a sequence of solutions with uniformly bounded energy fails, due to the possible energy concentration at isolated points in the domain. To investigate the compactness issue, one needs to apply blow-up analysis and get some bubbling solutions, i.e., finite energy nontrivial biharmonic maps from $\mathbb{R}^4$ to $N$, by rescaling at the energy concentration points, called blow-up points. This is a well known phenomenon in geometric variational problems
that are invariant under a non-compact group of local symmetries, dating back to the work \cite{sacks1981existence}.

In analogy to two dimensional harmonic maps which is conformally invariant, when the domain is Euclidean and of dimension 4, the blow-up theory for biharmonic maps have been developed via various methods e.g. \cite{wang2004remarks, hornung2012energy, wang2012energysphere, wang2012energy, laurain2013energy, liu2015finite,liu2016neck}. More precisely, a bubble tree convergence (including energy identity and no neck property) hold for a sequence of such maps with uniformly bounded energy. This gives compactness modular bubbles, a phenomenon which appears also in many other geometric variational problems, for instance, 2 dimensional and 1st order problems like J-holomorphic curves e.g. \cite{Gromov1985},
2 dimensional and 2nd order problems like harmonic maps from surfaces e.g. \cite{sacks1981existence,Jost, ding1995energy,parker1996bubble,qing1997bubbling, LinWang1998},
 4 dimensional problems like Yang-Mills fields e.g. \cite{Donaldson1983, DK1990, FU1991, InterpolaenergyquantYM}.

If we allow the domain to vary, then in general, bubble tree convergence may not hold. For instance, in two dimensions, when the domain Riemann surface varies in a noncompact region of the muduli space and hence degenerates to surfaces with nodes, bubble tree convergence still holds for some 1st order problems like J-holomorphic curves e.g. \cite{Gromov1985, Wolfson1988, PW1993, Ye1994}, however, for 2nd order problems like harmonic maps, in general, energy identity fails as energy may get lost on some degenerating necks, see the example in \cite{parker1996bubble}. In fact, the bubble tree convergence holds true only when we impose some extra constraints on the Poho\v{z}aev constant for harmonic maps over degenerating collor regions \cite{zhu2010harmonic}. 

More precisely, we shall investigate the compactness problem for the moduli space of biharmonic maps from varying non-collapsed Einstein 4-manifolds to a fixed compact Riemannian manifold. Thanks to the well developed and beautiful compactness theory for non-collapsed Einstein 4-manifolds
e.g. \cite{anderson1989ricci, nakajima1988hausdorff, BKN, bando1990bubbling, nakajima1994convergence, Tian1990, Anderson1992}, we may consider a sequence of biharmonic maps $u_j: M_j\rightarrow N$ from a sequence of non-collapsed closed Einstein 4-manifolds $(M_j,g_j)$ with bounded Einstein constants, bounded diameters and bounded $L^2$ curvature energy to a compact Riemannian manifold $(N,h)$ with uniformly bounded biharmonic energy.

It is well known that, on one hand, biharmonic maps can blow-up at energy concentration points in the domain and on the other hand, the underlying Einstein 4-manifolds can also blow-up at curvature concentration points and hence degenerate to a union of a Einstein 4-orbifold and some ALE  (asymptotically locally Euclidean) bubble manifolds (orbifolds). To study the compactness problem, there are several difficulties we need to overcome. The first difficulty is that the bubble-neck decomposition for biharmonic maps from varying domains is somewhat subtle, since there is a bubble tree convergence of the underlying Einstein 4-manifolds. In other words, the blow-up of biharmonic maps and the blow-up of the domain manifolds are entangled with each other. The second difficulty is that the degeneration of Einstein 4-manifolds is accompanied by the blow-up of Riemannian curvatures, while previous blow-up theory and neck analysis for biharmonic maps are developed for the cases of Euclidean domains in $\mathbb{R}^4$ or standard flat cylinder domains $[a,b] \times\mathbb{S}^3$ and hence can not be applied to the case of degenerating neck regions, where the geometric data are not explicitly known. More precisely, in order to analysis the behavior of biharmonic maps over the neck regions appeared in the bubble-neck decomposition for the degeneration of Einstein 4-manifolds, we need firstly to construct good global coordinates for those degenerate neck regions and take care of the behaviors of the varying metrics in these coordinate. Fortunately, inspired by the methods of constructing coordinates at the infinity in \cite{BKN, bando1990bubbling}, we succeed in constructing two types of good global coordinates for the degenerate neck region, which allow us to perform subtle neck analysis for biharmonic maps over the degenerating neck regions. In the following, we will sketch the bubble-neck decomposition procedure for the domain manifolds and briefly describe the geometry of the degenerating neck regions which look like portions of $\mathbb{R}^4/\Gamma$, and leave the details to the main text.



To investigate our problem, we shall firstly extend the energy identity and no neck property for biharmonic maps from a domain in $\mathbb{R}^4$ to the case that the domain is a fixed general Riemannian 4-manifold. Differently from harmonic maps from Riemann surfaces which is conformally invariant and hence the uniformization theorem for Riemann surfaces can be applied to reduce the local blow-up analysis problem to the case of Euclidean domains in $\mathbb{R}^2$ and standard flat cylinder domains $[a,b] \times\mathbb{S}^1$, the neck analysis for biharmonic maps from general curved 4-manifolds is more subtle. In \cite{liu2015finite}, the authors attempted to deal with neck analysis for biharmonic maps from a curved space and carried out detailed analysis for the special case of $\mathbb{S}^4$ with round metric. However, the case of general curved domain manifolds remains open.

 For a general Riemannian 4-manifold $(M,g)$, we shall always take the normal coordinates $(x)$ at some point $p=0$, so that the scaling $u(\lambda x)$ is well defined for small $\lambda>0$. For simplicity of notations, it suffices to consider the case that there is only one blow-up point in the domain and there is only one bubble map, as the general case of multiple bubbles follows by induction.

\begin{thm} \label{noneckthmintr}
 Let $ u _i$ be a sequence of extrinsic (or intrinsic) biharmonic maps from a geodesic ball $B_1\subset (M,g)$  to $N$ with uniformly bounded biharmonic energy
\begin{equation*}
E(u_i,B_1)\equiv\int_{B_1}|\nabla^2_{g}u_i|^2+|\nabla_{g} u_i|^4 dV_g\leq \Lambda,
\end{equation*}
for some $\Lambda > 0$. Assume that there is a sequence of positive numbers $\lambda_i \rightarrow 0$ such that
\begin{equation*}
u _i(\lambda_i x)\rightarrow \omega,
\end{equation*}
on any compact set $ K \subset \mathbb{R}^4$, $u_i$ converges weakly in $W^{ 2,2 }$ to some $u_\infty: B_1\subset (M,g) \rightarrow N $, and moreover, the nontrivial biharmonic map $\omega: \mathbb{R}^4 \rightarrow N $ is the only bubble. Then, we have

\

\noindent{\textbf{Energy identity}:}
\begin{equation*}
\lim_{\delta\rightarrow 0} \lim_{R\rightarrow \infty}\lim_{i \rightarrow \infty}\int_{B_\delta \setminus B_{\lambda_i R}}|\nabla^2_{g}u_i|^2+|\nabla_{g} u_i|^4 dV_g=0
\end{equation*}

\noindent{\textbf{No neck property}:}
\begin{equation*}
\lim_{\delta\rightarrow 0} \lim_{R\rightarrow \infty}\lim_{i \rightarrow \infty}osc_{B_\delta \setminus B_{\lambda_i R}}  \ u_i=0.
\end{equation*}
\end{thm}

\

To prove the above theorem, we can not use the idea of rewriting the biharmonic map equations with respect to the  Riemannian metric $g$ as some approximate biharmonic map type equations on a Euclidean domain. This is because if we view $u$ as approximate biharmonic maps from some Euclidean domain and put the effect of the domain metric $g$ into some error terms which involve fourth derivatives of $u$, then we can not get uniform $L^p(B_1)$ (for some $p>4/3$) control of these error terms when we deal with a blow-up sequence of biharmonic maps from $(M,g)$. So we can not apply the classical Poho\v{z}aev type arguments for approximate biharmonic maps from Euclidean domains and with error terms uniformly bounded in $L^p(B_1)$ for some $p>4/3$, which are developed in the previous works \cite{wang2012energy, liu2015finite}. Therefore we need to make new observations and develop new methods.

A key step for the proof of the above theorem is to establish the following Poho\v{z}aev identity for biharmonic maps from Riemannian manifolds (see Subsection \ref{Pohogenermetric}):
\begin{eqnarray}\label{pohoidentgenrlintro}
\int_{\partial B_r(p)}(r\partial_ru)(\partial_r\triangle_{g} u)+r\frac{|\triangle_g u|^2}{2} -\partial_r(r\partial_ru)\triangle _g u d\sigma_g=P(g,r)
\end{eqnarray}
where
\begin{eqnarray*}
P(g,r)&=&-\int_{B_r(p)}(\triangle_g x^k)e_ku (\triangle _g u )-2\int_{B_r(p)}(\nabla_gx^k)\nabla_g(e_k u)\triangle_g u-|\triangle_g u|^2dV_g\\
& &+\int_{B_r(p)}(\text{div}_g(r\partial_r)-4)\frac{|\triangle_g u|^2}{2} dV_g
- \int_{B_r(p)}Ric_g(\nabla_g u, x^ke_k)\triangle_g udV_g,\nonumber
\end{eqnarray*}
 $(x)$ are the normal coordinates, $\{e_i(q)\}, i=1,\cdots 4,$ are the orthonormal vector fields obtained by parallel transport along the geodesic curve from $p$ to $q$, and $r\partial_r=\sum_{i=1}^{4}x^ie_i$.
It is easy to see that the error term $P(g,r)$ in the right hand side of the above Poho\v{z}aev identity is caused by the non-flatness of the metric and it involves at most second derivatives of $u$. Later we will see that it is of order $O(r^2)$. This decay order is sufficient to carry out the neck analysis for the case of a fixed Riemannian domain in Theorem \ref{noneckthmintr}.

From the identity (\ref{pohoidentgenrlintro}), we get
\begin{equation*}\label{pohozaevmodnondege}
\int_{\partial B_r(p)} (\square_1+\square_2 )d\sigma_g=\frac{P(g,r)}{r},
\end{equation*}
where
\begin{eqnarray*}
\square_1=-\frac{1}{2}(\partial^{2}_r u)^2+\left(\frac{\tilde{m}^2}{2r^2}\right)(\partial_r u)^2+\frac{1}{2r^4}(\tilde{\triangle}u)^2+\left(\frac{\tilde{m}}{r^3}\right)(\partial_r u)(\tilde{\triangle}u),
\end{eqnarray*}
and
\begin{eqnarray*}
\square_2& =&(\partial_r u)(\partial^{3}_r u)+\frac{\tilde{m}-1}{r}(\partial_r u)(\partial^{2}_r u)+\left(\frac{\partial_r\tilde{m}}{r}-\frac{2\tilde{m}}{r^2}\right)(\partial_r u)^2\\
& &+\frac{1}{r^2}(\partial_r u)(\partial_r\tilde{\triangle} u)-\frac{\tilde{m}}{r^3}(\partial_r u)(\tilde{\triangle} u).\nonumber
\end{eqnarray*}
Here $\tilde{\triangle}$ is the associated Laplace operator of $ \widetilde{\partial B_r(p)} $, $ \widetilde{\partial B_r(p)} $ is the sphere $\partial B_r(p)$ with the metric $\frac{g_r}{r^2}$, and $g_r$ is the metric of $\partial B_r(p)$ induced from $(M,g)$. The above identity together with some further observations then reveals the relation between the tangential part energy and the radial part energy.

\

Now, we turn to the most challenging part of the present paper, namely, the case of domain manifolds with varying geometric structures. It turns out that when domain manifolds are equipped with degenerating Einstein metrics, the situation becomes much more complicated and difficult to handle. In fact, we need to take extra care of the decay orders of both the maps and the metrics over degenerating neck regions involved in each step, which requires a sophisticated asymptotic analysis.

Let $(M_i ,g_i)$  be a sequence of closed $4$-dimensional Einstein manifolds with uniformly bounded Einstein constants $\{\mu_i\}$ and satisfying
\begin{equation*}
diam(M_i ,g_i)\leq D, \,\, vol(M_i ,g_i)\geq V \,\, \text{and}\,\, \int_{M_i} |R_{g_i}|^{2}dV_{g_i}\leq R
\end{equation*}
for some positive constants $D>0, V>0$ and $R>0$, where we denote the curvature tensor of a metric $g$ by $R_g$. Let $u_i$ be a sequence of extrinsic (or intrinsic) biharmonic maps from $(M_ i ,g_ i )$ to $ N$ and satisfies the following uniform biharmonic energy bound condition:
\begin{equation*}
E(u_i,M_i)\equiv\int_{M_i}|\nabla^2_{g_i}u_i|^2+|\nabla_{g_i} u_i|^4 dV_{g_i}\leq \Lambda,
\end{equation*}
for some $\Lambda > 0$.

By the compactness theory of non-collapsed Einstein 4-manifold, we can get bubble tree convergence for the underlying Einstein 4-manifolds (up to a subsequence). For simplicity of the statement, we assume at the moment that there is only one curvature blow-up point in the domain (i.e., the curvatures of $M_i$ blow-up at only one point as $i\rightarrow \infty$), and that there is only one bubble manifold, which is a Ricci flat, asymptotically locally Euclidean (ALE) manifold. The case of multiple bubble manifolds (orbifolds) will be discussed at the end of Section \ref{treesigleale}.

Suppose that $M_j$ blow up at $x_{a,j}$, and for any $\delta >0$,
$$(B(x_{a,j},\delta),g_j)\subset M_j\rightarrow (B(x_{a},\delta),g_\infty)\subset M_\infty$$
in the sense of Gromov-Hausdorff convergence. Here $M_\infty$ is an orbifold with isolated singularity $x_a$.
Then there is a sequence of $r_j\rightarrow 0$ such that $(( M_j , r_{j}^{-2}g_j),x_{a,j})$ converge to $(( M_a , h_a),x_{a,\infty})$ in the pointed Gromov-Hausdorff distance, where $( M_a , h_a)$ is a complete, noncompact, Ricci flat, non-flat ALE $4$-manifold (4-orbifold). And the degenerate neck region $B(x_{a,j},\delta)\setminus B(x_{a,j},r_jR) $ with $\delta$ being very small and $R$ being very large is close to a potion of a flat cone $\mathbb{R}^4/\Gamma$ for some discrete group $\Gamma \subset SO(4)$, for large $j$.


On the thick part of the domain manifold, which is away from the curvature blow-up point $x_{a,j}$, we have nice convergence of the metrics (up to a subsequence) and hence the energy identity and the no neck property can be proved by slightly modifying the argument for the proof of Theorem \ref{noneckthmintr}, see Subsection \ref{blow-thick}.

At the curvature blow-up point $x_{a,j}$, the situation becomes much more complicated. The blow-up of the maps $u_j$ can occur inside the bubble manifold region $B(x_{a,j},r_jR)$, or inside the degenerate neck region $B(x_{a,j},\delta)\setminus B(x_{a,j},r_jR) $. In the former case, the blow-up scale can be some sequence $\lambda_j^{A} \ll r_j$ (i.e., $\lim_{j\rightarrow}\frac{\lambda_j^{A}}{r_j}=0$) or $r_j$. For the case of $\lambda_j^{A} \ll r_j$ (assume that there is only one bubble map in this case, i.e. $u_j$ blow up at only one scale), we have a bubble map
$$\omega^A: \mathbb{R}^4\rightarrow N.$$
For the case of $\{r_j\}$, we have a bubble map
$$\omega^B: M_a\rightarrow N.$$
In the degenerate neck region $B(x_{a,j},\delta)\setminus B(x_{a,j},r_jR) $, multiple bubble maps of different scales can possibly occur. Without loss of generality, suppose that there is only one bubble map in this case, then it takes the form
$$\omega^C: \mathbb{R}^4/\Gamma\rightarrow N.$$

Thus, to understand the compactness problem, we are left with two questions concerning the limiting behavior of the maps $u_j$ over the following region
$$(B(x_{a,j},\delta)\setminus B(x_{a,j},r_jR))\setminus B(x_{a,j}^{C},\lambda_j^{C}R) ,$$
where $B(x_{a,j}^{C},\lambda_j^{C}R)$ is the bubble map domain corresponding to $\omega^C$. The first question is whether there is any energy of the maps lost over this region. The second question is whether the image of the maps over this region is converging (up to a subsequence) to a point and hence the weak limit of the maps and the bubbles are all connected.

In this paper, we shall give affirmative answers to these two questions and hence establish a compactness theory modular finitely many bubbles for biharmonic maps from non-collapsed Einstein 4-manifolds with varying metrics. Now we state our main result in this paper.

\begin{thm}\label{degeymenergyintr} Let $(M_j,g_j)$, $N$ and $ u_j$ be as above.
Without loss of generality, we assume that there is only one ALE bubble manifold and there is at most one bubble map in each case in the blow-up scheme discussed above, then we have (up to a subsequence)

\

\noindent{\textbf{Energy identity}:}

\begin{eqnarray*}
\int_{B(x_{a,j},\delta_0)}|\nabla^2_{g_j}u_j|^2+|\nabla_{g_j} u_j|^4 dV_{g_j}=\int_{B(x_{a},\delta_0)}|\nabla^2_{g_\infty}u_\infty|^2+|\nabla_{g_\infty} u_\infty|^4 dV_{g_\infty}+\sum_{a=A,B,C}E(\omega^a),
\end{eqnarray*}
where $E(\cdot)$ is the biharmonic energy.

\

\noindent{\textbf{No neck property}:}
The image of $ u_\infty$ and the images of all bubbles are connected, namely, both the distance between the images of any two connected bubbles and the distance between the image of $u_\infty$ and the image of any bubble which is connected to it in $N$ are zero.
\end{thm}


The general case of multiple bubble manifolds and multiple bubble maps follows from induction arguments. This gives bubble tree convergence (including possibly finitely many ghost bubbles) for biharmonic maps from non-collapsed degenerating Einstein 4-manifolds. We remark that the biharmonic maps from Ricci flat ALE manifolds (orbifolfds) may be trivial, since the number of nontrivial Ricci flat ALE bubble manifolds (orbifolds) is finite, so the induction process must be terminated in finite number of steps.

When the target $N$ has trivial fourth fundamental group, namely, $\pi_4(N)=\{e\}$, then it is easy to see that the two types of bubble maps $\omega^A: \mathbb{R}^4\rightarrow N$ and $\omega^C: \mathbb{R}^4/\Gamma\rightarrow N$ have to be trivial, therefore, the maps $u_j$ blow up only if the underlying Einstein 4-manifolds $(M_j,g_j)$ blow up. Note that the existence of biharmonic maps in each free homotopy class in $[M,N]$ is ensured by the method of heat flow in \cite{wang2007heat,gastel2006} when $\pi_4(N)=\{e\}$. In this case, it is interesting to know whether there are finite energy  nontrivial biharmonic maps from a nontrivial Ricci flat ALE manifold (orbifold) $V$ to $N$.

To prove Theorem \ref{degeymenergyintr}, we shall develop three circle type method and establish new Poho\v{z}aev type identity for biharmonic maps from non-collapsed degenerating Einstein 4-manifolds. Recall that three circle type method is an important and powerful technique in geometric analysis since the work \cite{Simon1983}. For biharmonic maps from 4 dimensional Euclidean domains, such kind of method was developed in \cite{liu2015finite, liu2016neck} to get the decay estimates for the tangential part energy. Poho\v{z}aev type identities (or arguments) for biharmonic maps from a fixed domain in $\mathbb{R}^4$ were derived in \cite{wang2004remarks, hornung2012energy, wang2012energysphere, wang2012energy, liu2015finite,liu2016neck,laurain2013energy}. One can control the radial part energy by the tangential part energy in the neck analysis by detailed analysis of the relation between the radial part energy by the tangential part energy embodied in these identities. We remark that this relation is not as obvious as in the harmonic maps case.

However, when domains are degenerating Einstein 4-manifolds where the geometric data of the degenerating neck regions are not explicitly known, the situation becomes much more complicated and subtle, and there are several new difficulties we need to overcome. The first one is that we have to construct \emph{good global coordinates} on the whole degenerating neck region $B(x_{a,j},\delta)\setminus B(x_{a,j},r_jR)$ in which the metric $g_j$ is sufficiently regular for the neck analysis of biharmonic maps in order to prove Theorem \ref{degeymenergyintr}. The second one is that the non-flat metrics can effect the arguments throughout the proof, so we need to take care of the error terms appearing in each step.

To carry out delicate analysis on the degenerate neck region $B(x_{a,j},\delta)\setminus B(x_{a,j},r_jR)$, we succeed in constructing two types of good global coordinates on this region: $(x)$ and $(y)$, see Section \ref{constrcoordinates}. In the coordinates $(x)$, $g_j$ is $C^0$ close to the flat metric. While, in the coordinates $(y)$, $g_j$ is $C^4$ close to the flat metric in some weighted function space. To develop the three circle type method for biharmonic maps over the degenerating neck region, we have to use the coordinate $(y)$. While for the Poho\v{z}aev type argument, we have to use the $(x)$ coordinates which share a lot of similarities with the usual normal coordinates at a point. Therefore, to complete the proofs, we need to take advantage of both two coordinate systems.

To prove the degenerating case in Theorem \ref{degeymenergyintr}, we develop a new Poho\v{z}aev type argument which is more general than the case of a fixed Riemannian domain give in the proof of Theorem \ref{noneckthmintr} and the cases of Euclidean domains derived in previous works. More precisely, we derive a new Poho\v{z}aev type argument in a more geometric way and it depends much less on the coordinate functions than the previous ones developed for a fixed domain with local geometric data being explicitly given and regular enough. In fact, to deal with the degenerating case, we have to treat it in this way, since on one hand the coordinate functions $x^k$ are not regular enough for the analysis, and on the other hand the geodesic spheres of radius $r$ are not equal to the coordinates spheres given by $|y|=r$. We believe that our methods are also of wider interest in the field.

Here, we briefly describe our argument, see Subsection \ref{pohodege} for more details. The Poho\v{z}aev identity used in the degenerating neck region takes the following form:
\begin{eqnarray*}\label{pohoidentgenrl2}
\int_{\partial A_{j;t,ct}}(r\partial_ru_j)(\partial_r\triangle_{g_j} u_j)+r\frac{|\triangle_{g_j} u_j|^2}{2} -\partial_r(r\partial_ru_j)\triangle _{g_j}u_j d\sigma_{g_j}
=P(g_j, t),
\end{eqnarray*}
where $r=|x|=d_{g_j}(\cdot,x_{a,j})$, $P(g_j, t)$ is an error term depending on the metric $g_j$, and the bi-harmonic energy of $u_j$ over a portion of the degenerating neck region, denoted by
$ A_{j;t,ct}\equiv B(x_{a,j},ct)\setminus B(x_{a,j},t).$ Then it follows that
\begin{eqnarray*}\label{pohoidentgenrl2intro}
c\int_{\partial B(x_{a,j},ct)}\square_1+\square_2d\sigma_{g_j}-\int_{\partial B(x_{a,j},t)}\square_1+\square_2 d\sigma_{g_j}
=\frac{P(g_j,t)}{t},
\end{eqnarray*}
where $\square_1,\square_2$ have the same form as above. Later in Section \ref{degeneneckanalys}, we will see that the following form
$$cA(ct)-A(t)=\frac{P(g_j,t)}{t}, \quad (c>1) $$
is crucial to  prove Theorem \ref{degeymenergyintr} even though we do not have an identity of the form (\ref{pohoidentgenrlintro}).

A crucial and subtle fact in our argument is to show that the error term $P(g_j, t)$ has the same  order as the closeness of $g_j$ to the flat metric, see Lemma \ref{errorringhtj11}. Another difficulty in our Poho\v{z}aev type argument used in the proof of Theorem \ref{degeymenergyintr} is that we can only work on some sub-annular regions with nice geometric pictures, but not on the whole degenerating neck region $B(x_{a,j},\delta)\setminus B(x_{a,j},r_jR)$. This makes the whole argument more technical and also more interesting.

We remark that the geometry of the degenerating neck domains plays a vital role in the paper, especially the property
\begin{eqnarray*}
|g_{j,kl}-\delta_{kl}|<\eta_j(r)
\end{eqnarray*}
on $ A_{j;r_jR,\delta}\equiv B(x_{a,j},\delta)\setminus B(x_{a,j},r_jR)$ in the coordinates $(x)$ and $(y)$,
 where
 $$\eta_j(r)=O\left(r^{-\varepsilon_5}_\infty r^{\varepsilon_5}+ (r_jR_0)^{\varepsilon_5}r^{-\varepsilon_5}\right),$$
  $r_\infty>0$, $\varepsilon_5>0$, $r_j>0$ are the same as in the curvature estimate in Proposition \ref{neckprop}, and $R_0>0$ is some fixed large number. The arguments in this proof reveal that Theorem \ref{degeymenergyintr} holds because of the important fact that there is no curvature concentration on the degenerating neck regions in the tree of ALE bubble manifolds (orbifolds) $\mathbf{Tr_{ALE}}$, which is a direct consequence of Proposition \ref{neckprop} (see Section \ref{Preliminaries111}).

The paper is organized as follows. In Section \ref{Preliminaries111}, we recall some background knowledge about the bubble tree convergence of non-collapsed degenerating Einstein manifolds and some key results for the blow-up analysis of biharmonic maps. In Section \ref{noneckprf}, we prove Theorem \ref{noneckthmintr} (for extrinsic biharmonic maps). In Section \ref{treesigleale}, we do the bubble-neck decomposition. In Section \ref{constrcoordinates}, we construct the two types of good coordinates. In Section \ref{degeneneckanalys}, we prove Theorem \ref{degeymenergyintr} (for extrinsic biharmonic maps). In Section \ref{secintrinsic}, we deal with the case of intrinsic biharmonic maps.

\

\section{ Preliminaries}\label{Preliminaries111}

\subsection{Bubble tree convergence of non-collapsed Einstein manifolds}
We recall the compactness theory of non-collapsed Einstein metrics developed in
\cite{anderson1989ricci, nakajima1988hausdorff, BKN, bando1990bubbling, Tian1990, nakajima1994convergence, Anderson1992} etc. The definitions of (pointed) Gromov-Hausdorff convergence, Riemannian orbifold, and ALE (asymptotically locally Euclidean) manifold can be found in the above mentioned works and references therein. For more details about G-H convergence, we refer to Chapter 7 of \cite{BBSmetricgeo}. Roughly speaking, a 4-dimensional orbifold locally looks like $\mathbb{R}^4/\Gamma_0$ for some finite group $\Gamma_0 \subset SO(4)$ at the isolated singularities, an ALE 4-manifold is a complete noncompact manifold which is close to $\mathbb{R}^4/\Gamma_\infty$ at the infinity for some finite group $\Gamma_\infty \subset SO(4)$. We remark here that the results collected in this subsection in fact hold for all dimensions $n\geq4$, as along as we replace the $L^2$ curvature condition given in Theorem \ref{mainconvgethm} by the following
$$\int_{M_i} |R_{g_i}|^{\frac{n}{2}}dV_{g_i}\leq R.$$

\begin{thm} \label{mainconvgethm}(\cite{anderson1989ricci,nakajima1988hausdorff,BKN})
Let $(M_i ,g_i)$  be a sequence of $4$-dimemsional smooth manifolds and the Einstein metrics on them with uniformly bounded Einstein constants $\{\mu_i\}$ satisfying
\begin{equation*}
diam(M_i ,g_i)\leq D, \,\, vol(M_i ,g_i)\geq V \,\, \text{and}\,\, \int_{M_i} |R_{g_i}|^{2}dV_{g_i}\leq R
\end{equation*}
for some positive constants $D, V$ and $R$, where the curvature tensor of a metric $g$ is denoted by $R_g$. Then there exist s subsequence $\{j\}\subset \{i\}$ and a compact Einstein orbifold $(M_\infty, G_\infty )$ with a finite set (possibly empty) of orbifold singular points $S=\{x_1,x_2,\cdots,x_s\}\subset M_\infty$  for which the following statements hold:

\begin{itemize}
\item[(1)] $(M_j ,g_j)$ converges to $(M_\infty, g_\infty )$ in the Gromov-Hausdorff distance.

\item[(2)] There exists an into diffeomorphism $F_j:M_\infty\rightarrow M_j$ for each $j$ such that $F^{*}_{j}g_j$ converges to
$g_\infty$ in the $C^\infty$ -topology on $M_\infty \setminus S$.

\item[(3)] For every $x_a\in S$ ($a=1,2,\cdots,s$) and $j$, there exists $x_{a,j}\in M_j$  and a positive number $r_j$ such that

\begin{itemize}
\item[(3.a)] $B(x_{a,j} ,\delta)$ converges to $B(x_a, \delta) $ in the Gromov-Hausdorff distance for all $\delta>0$.

\item[(3.b)] $\lim_{j\rightarrow \infty} r_j =0$

\item[(3.c)] $(( M_j , r_{j}^{-2}g_j),x_{a,j})$ converge to $(( M_a , h_a),x_{a,\infty})$ in the pointed Gromov-Hausdorff distance, where $( M_a , h_a)$ is a complete, noncompact, Ricci flat, non-flat $4$-manifold which is ALE of order $4$.

\item[(3.d)] There exists an into diffeomorphism $G_j : M_a \rightarrow M_j$ such that $G^{*}_{j}(r_{j}^{-2} g_j)$ converges to $h_a$ in the $C^\infty$-topology on $M_a$.

\end{itemize}

\item[(4)] It holds that
\begin{equation*}
\lim_{j\rightarrow \infty}\int_{M_j}|R_{g_j}|^2 dV_j\geq \int_{M_\infty}|R_{g_\infty}|^2 dV_\infty +\sum_a \int_{M_a}|R_{h_a}|^2 dV_{h_a}.
\end{equation*}

\end{itemize}

\end{thm}

\begin{rem}
Cheeger-Naber showed that non-collapsed 4-manifolds
with bounded Ricci curvature have a priori $L^2$ Riemannian curvature estimates,
see \cite{cheeger2015regularity}.
\end{rem}

\begin{rem}\label{remorbfsmoth}
By Theorem (5.1) of \cite{BKN}, the metric $g_\infty$ on $M_\infty$ extends smoothly across $x_a$ as an orbifold Einstein metric. That is to say there is a (covering) map
$\Pi_a:B_1\setminus\{0\}\rightarrow B(x_a, \delta)\setminus\{x_a\} $ such that
$\Pi^{*}_{a}g_\infty $ extends to a smooth Einstein metric on $B_1\subset \mathbb{R}^4$.
\end{rem}

Later, the bubble tree convergence was proved in \cite{bando1990bubbling}, showing that there is no curvature concentration in the neck regions (see also \cite{nakajima1994convergence}). Here for the sake of convenience, we state these results using notations in \cite{bando1990bubbling}. Now we take a positive constant $r_\infty$ sufficiently small so that for all
\begin{equation*}
\sup_{B(x_{a,j},r_\infty)} |R_{g_j}|^2=|R_{g_j}|^2(x_{a,j})\rightarrow \infty \quad \text{as} \quad j\rightarrow \infty
\end{equation*}
and
\begin{equation*}
\int_{B(x_a,r_\infty)} |R_{g_\infty}|^2\leq \frac{\varepsilon}{2}
\end{equation*}
with a small positive number $\varepsilon\leq \varepsilon_4/2$ determined by the argument in \cite{bando1990bubbling}, where $\varepsilon_4$ is the number in Proposition 1 of \cite{bando1990bubbling} (a small energy regularity theorem).
And for sufficiently large $j$ we can find a positive number $r_j$ so that
\begin{equation*}
\int_{B(x_{a,j},r_\infty)\setminus B(x_{a,j},r_j)} |R_{g_j}|^2 =\varepsilon.
\end{equation*}
It is easy to see that
\begin{equation*}
r_j\rightarrow 0 \quad \text{as} \quad j\rightarrow \infty.
\end{equation*}

\begin{prop}(\cite{bando1990bubbling})\label{covergtoorbifold}
There is a subsequence $\{k\}\subset\{j\}$ (denoted still by $\{j\}$), such that $  ( M_j,r_j^{-2}g_j),x_{a,j})$ converges to $((Y,h),y_\infty)$ (we call it an ALE bubble) in the pointed G-H distance, where $(Y,h)$ is a complete, non-compact, Ricci flat, non-flat ALE 4-manifold (orbifold) of order 4 with only finitely many isolated singular points. The convergence is actually smooth except at the singular points.
\end{prop}

\begin{rem}\label{aletree}
By repeating the same process, we can obtain a bubble tree, for more details we refer to Page 211 of \cite{bando1990bubbling}. In the following we shall call it ALE-tree, and denote it by $\mathbf{Tr_{ALE}}$. And each one of the Ricci flat ALE manifolds (orbifolds) in the tree has only one end by the splitting theorems \cite{Borzellino1994,Cheeger1971} and the arguments in Theorem 3.5 of \cite{anderson1989ricci}.
\end{rem}

\begin{prop} \label{neckprop}(Proposition 3 and Proposition 4 in \cite{bando1990bubbling})
There exist positive constants $C_7>0$ and $\varepsilon_5>0$ such that for $ 4r_j\leq r< 4r \leq r_\infty$ it holds that
\begin{equation*}\label{neckcurvturestmprop}
r^2|R_{g_j}|\leq C_7 \max \left\{(\frac{r_j}{r})^{\varepsilon_5},(\frac{r}{r_\infty})^{\varepsilon_5}  \right\}.
\end{equation*}
And if one takes $1<K_1<K_2$ sufficiently large, then the subset $B(x_{a,j},K_2^{-1}r_\infty)\setminus B(x_{a,j},K_1r_j) $ (called the degenerate neck region) is close to a potion of some flat cone $\mathbb{R}^4/\Gamma$ for large $j$.
\end{prop}

\begin{rem}\label{neckcoordinates}
It was noticed in \cite{bando1990bubbling} that there exist certain coordinates on the degenerate neck region. However, no details of construction of coordinates were given and little was known about the properties of such coordinates before. In order to carry out refined analysis for solutions of geometric PDEs defined over degenerating domains, we need to firstly construct good coordinates on these neck regions. By adapting the methods in \cite{BKN}, we succeed in constructing coordinates on the degenerate neck region (see Theorem \ref{neckmeytricflat})
$$ A_{K_1r_j,K_2^{-1}r_\infty}\equiv B(x_{a,j},K_2^{-1}r_\infty)\setminus B(x_{a,j},K_1r_j),$$
in which $g_j$ is close to the standard flat metric in certain sense.
It indicates that the finite group $\Gamma $ of the neck is equal to the fundamental group at the infinity of the ALE orbifold $Y$ and the local fundamental group of $B(x_a,r_\infty)$.
\end{rem}

 Notice that $\varepsilon_5>0$, a direct consequence of Proposition \ref{neckprop} is a curvature energy identity of the following type, one can find more specific statements in Theorem 2.5 of \cite{nakajima1994convergence}.

\begin{cor}\label{curenergyidentity}
Let $r_j,r_\infty,x_{a,j}, K_1,K_2$ be as above, then it holds
\begin{equation*}
\lim_{K_2\rightarrow \infty}\lim_{K_1\rightarrow \infty }\lim_{r_j\rightarrow 0}\int_{B(x_{a,j},K_2^{-1}r_\infty)\setminus B(x_{a,j},K_1r_j)}|R_{g_j}|^2 dV_j=0.
\end{equation*}
\end{cor}

\subsection{Biharmonic maps}
Recall that extrinsic and intrinsic biharmonic maps are the critical points of the following two functionals
\begin{equation*}
H(u)=\frac{1}{4}\int_{M}|\triangle_g u|^2 dV_g
\end{equation*}
and
\begin{equation*}
T(u)=\frac{1}{4}\int_{M}|\tau(u)|^2 dV_g.
\end{equation*}
 The Euler-Lagrange equation for $H(u)$ is
 \begin{equation}\label{extrequ}
 \triangle^{2}_{g}u=\triangle_g u(B(u)(\nabla _g u,\nabla_g u))+2\nabla_g\cdot\langle\triangle_g u,\nabla_g(P(u))\rangle
 -\langle\triangle_g (P(u)),\triangle_g u\rangle.
 \end{equation}
  Here $B $ is the second fundamental
form of $N\subset \mathbb{R}^K$ and $P(u)$ is the projection to the tangent space $T _u N$.
The Euler-Lagrange equation for $T(u)$ is more complicated, we refer to \cite{liu2016neck}. In this paper we shall mainly focus on extrinsic biharmonic maps, the case of intrinsic biharmonic maps will be handled in Section \ref{secintrinsic}.

The analytic aspects of biharmonic maps have been studied extensively in recent decades.
The regularity theory of biharmonic maps are developed in  \cite{chang1999regularity,Strzelecki2003On, wang2004biharmonic, wang2004stationary, wang2004remarks,  Moser2006Remarks, lamm2008conservation, struwe2008partial, scheven2008dimension} etc. The blow-up theory (including energy identity and the no neck property) of biharmonic maps from domains in $\mathbb{R}^4$ had been studied in \cite{wang2004remarks,hornung2012energy,laurain2013energy,wang2012energysphere,wang2012energy,
liu2016neck,liu2015finite} etc.

Now we recall some basic analytical tools for the blow-up analysis for biharmonic maps in dimension 4. In fact, these basic tools can be easily extended to the case of a Riemannian domain.

First of all, there is a small energy regularity theorem for biharmonic maps.

\begin{thm}($\varepsilon$-regularity)\label{smallenergythm}
 Let $(B_1,g)$ be a ball in $\mathbb{R}^4$ equipped with a Riemannian metric $g$, for any $p > 1$, there exists $\varepsilon_0 > 0$ and $C_p > 0 $ such that if
 $u \in W ^{2,2} (B_1, \mathbb{R}^K ) $ is an extrinsic (or intrinsic) biharmonic map into $N$ satisfying
 \begin{equation}\label{smallenergy1}
 \int_{B_1}|\nabla^2_{g}u|^2+|\nabla_{g} u|^4 dV_g\leq \varepsilon_0,
 \end{equation}
 then
 \begin{equation}\label{smallenergyestm}
 \|u-\bar{u}\|_{W^{4,p}(B_{1/2})}\leq C_p (\|\nabla^2u\|_{L^{2}(B_{1})}+\|\nabla u\|_{L^{2}(B_{1})}),
 \end{equation}
 where $\bar{u}$ is the mean value of $u$ over the unit ball.
\end{thm}

\begin{proof}
When the metric $g$ is Euclidean, the corresponding result has been proved in previous works, for example, \cite{wang2012energy, laurain2013energy}. It is easy to see that the arguments in \cite{laurain2013energy} remain true when we consider biharmonic maps on balls equipped with smooth Riemannian metrics. It follows that (\ref{smallenergyestm}) holds.
\end{proof}

\begin{rem}\label{regusmalmetrdeped}
 There is no need to distinguish the $L^p$ norm with respect to $g$ and the $L^p$ norm with respect to the Euclidean metric for our purpose, as they are  locally equivalent to each other. 
\end{rem}

\begin{cor} \label{smallenergycor}
Let $\varepsilon_0 > 0$ be as in Theorem \ref{smallenergythm}, if
\begin{equation}\label{smallenergy1}
 \int_{B_r}|\nabla^2_{g}u|^2+|\nabla_{g} u|^4 dV_g\leq \varepsilon_0,
 \end{equation}
 then for $l=1,2,\cdots,$
\begin{equation}\label{smallenergy1}
 \sup_{B_{r/2}}|\nabla^l_{g}u|\leq C\frac{(\varepsilon_0)^{1/2}}{r^{l}}.
 \end{equation}
\end{cor}

It is easy to see that the regularity theory for weakly biharmonic maps from domains in $\mathbb{R}^4$  developed in e.g. \cite{wang2004biharmonic, lamm2008conservation} can be extended to the case of a Riemannian domain $B_1 \subset (M^4,g)$. Moreover, one can apply an argument as Lemma 2.5 in \cite{hornung2012energy} to show that a biharmonic map which is smooth in $B_1 \setminus  \{0\} \subset (M^4,g)$ and with finite energy is in fact a weakly biharmonic map over the whole domain. Therefore, the removable singularity theorem of extrinsic (or intrinsic) biharmonic maps follows as a corollary.

\begin{thm}(removable singularity)\label{remsingu}
Let $ u$ be a smooth biharmonic map on $B_1 \setminus  \{0\} \subset (M^4,g)$,
if
\begin{equation*}
 \int_{B_1}|\nabla^2_{g}u|^2+|\nabla_{g} u|^4 dV_g< \infty,
 \end{equation*}
then $u$ can be extended to a smooth biharmonic map on the whole $B_1$.
\end{thm}

\begin{rem}
Usually, removable singularity theorems can be proved by applying Poho\v{z}aev type arguments as in \cite{sacks1981existence}. See for example the proof of removable singularity theorem for biharmonic maps in Section 6 of \cite{liu2016neck} for the case of Euclidean domains. In fact, one can modify the method in \cite{liu2016neck} to give a different proof of Theorem \ref{remsingu}. We will use similar arguments to show that the oscillation of a biharmonic map at the infinity is arbitrarily small (see Lemma \ref{lemrovsiguifty}) in Appendix \ref{prflem1}.

\end{rem}

\begin{thm}\label{energygap11}
(energy gap) Let $u $ be a biharmonic map from $\mathbb{R}^4 / \Gamma$ ($\Gamma $ is some finite group in $SO(4)$) to $N$. There exists $\varepsilon(N,\Gamma)> 0$ depending only on $N$ and $\Gamma$, such that if $u$ satisfies
\begin{equation}\label{smallenergy1}
 \int_{\mathbb{R}^4 / \Gamma}|\nabla^2_{g}u|^2+|\nabla_{g} u|^4 dV_g\leq \varepsilon(N,\Gamma),
 \end{equation}
then $u$ is a constant map.
\end{thm}
\begin{proof} For $\Gamma=\{e\}$, it is a direct consequence of Theorem \ref{smallenergycor}.
For a nontrivial group $\Gamma$, we can lift the map $u$ to be a map $\tilde{u}:\mathbb{R}^4\setminus \{0\}\rightarrow N $. And then by Theorem \ref{remsingu}, it can be extended to be a biharmonic map from $\mathbb{R}^4$ to $N$. We may take $\varepsilon(N,\Gamma)=\varepsilon_0/|\Gamma|$, where $|\Gamma|$ is the number of elements in $\Gamma$.
\end{proof}

With the above analytical tools in hand, now we are ready to analyse the blow-up procedure for biharmonic maps from Riemannian manifolds. We refer to \cite{laurain2013energy,wang2012energy} for the bubble-neck decomposition.

For a fixed Riemannian maniflod $(M,g)$, we always take the normal coordinates $x$, so that the scaling $u(\lambda x)$ is well defined for small $\lambda>0$. By Gauss Lemma, for small $r>0$ the geodesic ball $B _r$  is the same as the ball of radius $r$ in normal coordinates.

Now we state the energy identity and no neck property theorem for biharmonic maps defined on general Riemannian manifolds.
\begin{thm} (Energy identity and no neck property)\label{noneckthm}
 Let $ u _i$ be a sequence of extrinsic biharmonic maps from $B_1\subset (M,g)$  to $N$ satisfying
\begin{equation*}
E(u_i,B_1)\equiv\int_{B_1}|\nabla^2_{g}u_i|^2+|\nabla_{g} u_i|^4 dV_g\leq \Lambda,
\end{equation*}
for some $\Lambda > 0$. Assume that there is a sequence positive $\lambda_i \rightarrow 0$ such that
\begin{equation*}
u _i(\lambda_i x)\rightarrow \omega,
\end{equation*}
on any compact set $ K \subset \mathbb{R}^4$, $u_i$ converges weakly in $W^{ 2,2 }$ to $u_\infty$, and moreover $\omega$ is
the only bubble. Then,
\begin{equation*}
\lim_{\delta\rightarrow 0} \lim_{R\rightarrow \infty}\lim_{i \rightarrow \infty}\int_{B_\delta \setminus B_{\lambda_i R}}|\nabla^2_{g}u_i|^2+|\nabla_{g} u_i|^4 dV_g=0
\end{equation*}
and
\begin{equation*}
\lim_{\delta\rightarrow 0} \lim_{R\rightarrow \infty}\lim_{i \rightarrow \infty}osc_{B_\delta \setminus B_{\lambda_i R}} \ u_i=0.
\end{equation*}
\end{thm}

\begin{rem}
By arguments in Section \ref{secintrinsic}, we know that the same results hold for intrinsic biharmonic maps.
\end{rem}

A simple consequence of Theorem \ref{noneckthm} is that the following limit
$$\lim_{ |x|\rightarrow\infty} \omega(x)$$
exists (see the end of Section 2 of \cite{liu2016neck}). In this theorem, we assume that there is only one bubble. The general case of multiple bubble maps can be handled by induction.

In order to prove Theorem \ref{noneckthm}, we shall extend the methods developed for the case of  Euclidean domains in \cite{ laurain2013energy, liu2016neck, liu2015finite} to a more general setting. There are two main ingredients in \cite{liu2016neck, liu2015finite}, one is a three circle theorem for approximate biharmonic functions which is used to get the exponential decay for the tangential part energy on the neck regions, the other is a Poho\v{z}aev type argument. The main difficulty is to establish a Poho\v{z}aev type identity for biharmonic maps defined on geodesic balls in general Riemannian manifolds. Fortunately, the Poho\v{z}aev type identities on Euclidean domains derived in \cite{hornung2012energy, wang2012energy, laurain2013energy} provides the key clue. We shall put the proof in the next section. We remark that the proof have something in common with the proof for the uniqueness of tangent cone of (bi)harmonic maps \cite{Simon1983,simon1996theorems,chen2019uniqueness}. Both of them need to prove some decay estimates for the maps by subtle analysis of the structure of equations.

\

\section{Neck analysis on a fixed Riemannian manifold}\label{noneckprf}

In this section, we shall prove the energy identity and no neck property in Theorem \ref{noneckthm} by establishing a Poho\v{z}aev identity for biharmonic maps on general Riemannian manifolds.

\subsection{Decay of the tangential part energy}\label{tangdeacynodege}
For the sake of readers' convenience, we shall firstly recall some results in \cite{liu2015finite,liu2016neck}.
We start from the definition of $\eta$-approximate biharmonic functions.

\begin{defn}\label{defappbiharm} (Definition 1 of \cite{liu2015finite})
Let u be a smooth function defined on $B _{r_2}\setminus  B_{ r_1} $, $ u$ is called an
$\eta$-approximate biharmonic function, if it satisfies
\begin{eqnarray}\label{approbihar}
\triangle^2_{g}u(r,\theta)&=&a_1\nabla_{g}\triangle_g u+a_2\nabla_{g}^2u+a_3\nabla_g u+a_4 u \nonumber\\
&+&\frac{1}{|\partial B_r|}\int_{\partial B_r}b_1\nabla_{g}\triangle_g u+b_2\nabla_{g}^2u+b_3\nabla_g u+b_4 u d\sigma +h(x),
\end{eqnarray}
where $a_i$, $b_i$ and $h$ are smooth functions satisfying the followings: for any $\rho\in [r_1, r_2/2]$,

\begin{itemize}

\item[(a)] $||g_{ij}(\rho x)-\delta_{ij}||_{C^4(B_2\setminus B_1)}<\eta.$
Namely, the metric after scaling to $B_2\setminus B_1$ is close to the flat metric in $C^4$ norm.

\item[(b)]
\begin{equation*}\label{tensioncondit}
|||x|^{4(1-1/p)}h||_{L^p(B _{r_2}\setminus  B_{ r_1} )}\leq \eta
\end{equation*}

\item[(c)]
\begin{equation*}
\sum_{i=1}^{4}||a_i||_{\tilde{W}^{4-i,p}(B_{2\rho}\setminus B_\rho)}+||b_i||_{\tilde{W}^{4-i,p}(B_{2\rho}\setminus B_\rho)}\leq \eta.
\end{equation*}
Here the $\tilde{W}^{4-i,p} $ norm is defined by
\begin{equation*}\label{tensioncondit}
||w||_{\tilde{W}^{4-i,p}(B_{2\rho}\setminus B_\rho)}=||\rho ^i w(\rho x)||_{W^{4-i,p}(B_{2}\setminus B_1)}.
\end{equation*}

\end{itemize}

\end{defn}

One can check that if $u$ is an $\eta$-approximate biharmonic function
on $B _{r_2}\setminus  B_{ r_1}$, then $w(x) = u( x/\lambda )$ is another $\eta$-approximate biharmonic function on
$B _{\lambda r_2}\setminus  B_{ \lambda r_1}$.

The next is an interior $L^p$ estimate for approximate biharmonic functions.
\begin{lem}\label{appbiharmest} (Lemma 1 of \cite{liu2015finite})
 Suppose that $u : B _4 \setminus B _1 \rightarrow R $ is an $\eta$-approximate biharmonic function (for small $\eta>0$) with
 \begin{equation*}
\sum_{i=1}^{4}||a_i||_{{W}^{4-i,p}(B_{4}\setminus B_1)}+||b_i||_{{W}^{4-i,p}(B_{4}\setminus B_1)}\leq \eta.
\end{equation*}
and
\begin{equation*}\label{tensioncondit}
||h||_{L^p(B _{4}\setminus  B_{ 1} )}\leq C.
\end{equation*}
Then, for any $p > 1$, we have
\begin{equation*}
||u||_{W^{4,p}(B _{3}\setminus  B_{ 2})}\leq C(||u||_{L^{p}(B _{4}\setminus  B_{ 1})}+||h||_{L^p(B _{4}\setminus  B_{ 1} )}).
\end{equation*}
\end{lem}

Now we study the behaviour of the biharmonic map (with respect to the matric $g$) $u_i$ on the neck region $\Sigma=B_\delta \setminus B_{\lambda_i R}$ as in Section 4 of \cite{liu2016neck}. Assume without loss of generality that
\begin{equation*}
\Sigma=\bigcup_{l=l_0}^{l_i}A_l,
\end{equation*}
for $A_l=B_{e^{-(l-1)L}} \setminus B_{e^{-lL}}$ and $l_0<l_i$. Moreover we may assume as in \cite{liu2016neck} that for any $\varepsilon>0$ ,
\begin{equation}\label{energysmallnecki}
\int_{A_l}|\nabla_{g}^2u_i|^2+|\nabla_g u_i|^4dV_g<\varepsilon^2<\varepsilon_0
\end{equation}
holds for $l=l_0,\cdots,l_i$, when $i$ and $R$ are large enough and $\delta>0$ is small enough. In the above, $\log \delta=-l_0L$ and $\log\lambda_i R=-l_iL$.

Our aim in this subsection is to prove the following exponential decay estimates:
\begin{lem}\label{lemtangdecay}
After setting $ r=e^t$, we have for arbitrarily small $\varepsilon>0$, if $\delta>0$ is sufficiently small and $i$ and $R$ are large enough, then
\begin{eqnarray*}\label{tangentialdecay}
& &\int_{(-lL,-(l-1)L)\times S^3}\left(|\triangle_{S^3}u_i|^2+|\nabla_{S^3}u_i|^4+|\partial_t \nabla_{S^3}u_i|^2+|\partial_t \triangle_{S^3}u_i|^2\right)dt d\theta\nonumber\\
&\leq&C\varepsilon^2(e^{-(l-l_0)L}+e^{-(l_i-l)L}).
\end{eqnarray*}
Or equivalently,
\begin{eqnarray*}
& &\int_{B_{e^L r}\setminus B_r}\frac{1}{r^4} \left ((\triangle_{S^3}u_i)^2+|\nabla_{S^3}u_i|^4 \right)+\frac{1}{r^2}|\partial_r\nabla_{S^3}u_i|^2+\frac{1}{r^2}|\partial_r\triangle_{S^3}u_i|^2dx\\
&\leq&C\varepsilon^2 \left(\frac{r}{\delta}+\frac{\lambda_i R}{r} \right).
\end{eqnarray*}
\end{lem}

The main tool to prove the above lemma is the following three circle theorem for approximate biharmonic functions:

\begin{thm} \label{thrcircledegeneck} (Theorem 4 of \cite{liu2015finite})
There is some constant $\eta_0>0$ such that the following is true.
Assume that $v : A_{l-1} \bigcup A_l \bigcup A_{l+1} \rightarrow \mathbb{R}^N$ is an $\eta_0$-approximate biharmonic function
in the sense of Definition \ref{defappbiharm}. Suppose \begin{equation}\label{akeycondition1}
\max_{l-1,l,l+1}||r^{4(1-1/p)}h||^2_{L^p(A_l)}\leq \eta_0 F_l(v)
\end{equation}
and
\begin{equation}\label{akeycondition2}
\int_{\partial B_r(p)}vd\sigma_g=0,
\end{equation}
then we have the following alternatives:

\

\begin{itemize}
\item[(a)] if $F_{ l+1 }(v) \leq e^{ -L} F _l (v)$, then $F_ l (v) \leq e^{ -L} F _{l-1 }(v)$;\\
\item[(b)] if $F _{l-1} (v) \leq e^{ -L} F _l (v)$, then $F_ l (v) \leq e^{ -L} F_{ l+1 }(v)$;\\
\item[(c)] either $F_ l (v) \leq e ^{-L} F_{ l-1} (v)$, or $F_ l (v) \leq e ^{-L }F _{l+1}(v)$.\\
\end{itemize}
Here $r=d_{g}(x, p)$, $A_l=B_{ e^{ -(l-1)L}} \setminus B_{ e^{ -lL}}(p)\subset (M,g)$,
\begin{equation*}
F_ l (v) =\int_{A_l}\frac{v^2}{r^4 } dV_{g}.
\end{equation*}

\end{thm}

In this paper, we shall apply the above theorem only in the flat metric situation. Correspondingly,
\begin{equation*}
F_ l (v) :=\int_{A_l}\frac{v^2}{|x|^4 } dx.
\end{equation*}

\begin{cor}\label{threecirtotangdecay}
If $u_i-(u_i)^*\equiv v_i$ is an $\eta_0$-approximate biharmonic function on $\Sigma=B_\delta \setminus B_{\lambda_i R}$
in the sense of Definition \ref{defappbiharm} with respect to the Euclidean metric, and it satisfies (\ref{approbihar}) with the condition
\begin{equation}\label{akeycondition3}
||r^{4(1-1/p)}h_i||_{L^p(A_l)}\leq \eta(r)\varepsilon.
\end{equation}
Here $(u_i)^*(r)=\frac{1}{|\partial B_r|}\int_{\partial B_r(0)}u_id\sigma$  and $\eta(r)=O\left((\frac{\lambda_iR}{r})^{\kappa}+(\frac{r}{\delta})^{\kappa}\right)$, $\kappa>0$.
Then we have
\begin{eqnarray}\label{tangentialdecayzero}
F_l(v_i)\leq C\varepsilon^2(e^{-\hat{\kappa}(l-l_0)L}+e^{-\hat{\kappa}(l_i-l)L}),
\end{eqnarray}
and
\begin{eqnarray}\label{tangentialdecaycor}
& &\int_{(-lL,-(l-1)L)\times S^3}(|\triangle_{S^3}u_i|^2+|\nabla_{S^3}u_i|^4+|\partial_t \nabla_{S^3}u_i|^2+|\partial_t \triangle_{S^3}u_i|^2)dt d\theta\nonumber\\
&\leq&C\varepsilon^2(e^{-\hat{\kappa}(l-l_0)L}+e^{-\hat{\kappa}(l_i-l)L}),
\end{eqnarray}
where $\hat{\kappa}=\min(1,2\kappa)$.
In particular, if $\kappa \geq1/2$, then $u_i$ satisfies the estimate in
Lemma \ref{lemtangdecay}.

\end{cor}

\begin{rem}
Corollary \ref{threecirtotangdecay} is almost a direct consequence of Theorem \ref{thrcircledegeneck} and Lemma \ref{appbiharmest}. Firstly, it is easy to see that $v_i=u_i-(u_i)^*$ satisfies (\ref{akeycondition2}). Secondly, (\ref{energysmallnecki}) says that $F_l(v_i)\leq C\varepsilon^2$ for all $l_0<l<l_i$. Thirdly, when condition (\ref{akeycondition1}) in Theorem \ref{thrcircledegeneck} is not satisfied, (\ref{akeycondition3}) ensures that we can get the desired decay estimate in (\ref{tangentialdecayzero}). Notice that $\nabla_{S^3}u_i=\nabla_{S^3}v_i$, by using Lemma \ref{appbiharmest} and Sobolev embedding theorem we obtain the decay estimate for the tangential part energy of $u_i$ in (\ref{tangentialdecaycor}).
For more details, we refer to the proofs of Lemma 3 and Lemma 4 in \cite{liu2015finite}.
\end{rem}

By using Corollary \ref{threecirtotangdecay}, to prove Lemma \ref{lemtangdecay}, the key is to rewrite the biharmonic map equation for $u_i$ (w.r.t the metric $g$) to be some approximate biharmonic map equation (w.r.t the flat metric), and then apply similar arguments as in \cite{liu2015finite} to show that $u_i-(u_i)^*$ is an approximate biharmonic function.

\begin{proof}[\textbf{Proof of Lemma \ref{lemtangdecay}}]

Recall that the extrinsic biharmonic map equation (\ref{extrequ}) can be written as (see \cite{liu2016neck})
\begin{eqnarray}\label{extrequ3}
\triangle^{2}_g u &=&\alpha_1(u)\nabla_g \triangle _g u \sharp \nabla_g u+\alpha_2(u)\nabla^{2}_g u \sharp \nabla^{2}_g u\nonumber\\
&& + \alpha_3(u)\nabla^{2}_g u \sharp \nabla_g u\sharp \nabla_g u+\alpha_4(u)\nabla_g u \sharp \nabla_g u\sharp \nabla_g u\sharp \nabla_g u.
\end{eqnarray}
Here $\alpha_i$ is a smooth function of $u$ and $\sharp$ is the contraction of tensors (with respect to the metric $g$), for which we have, for example,
\begin{equation}\label{propertytensor}
|\nabla_g \triangle _g u \sharp \nabla_g u|\leq C |\nabla_g \triangle _g u ||\nabla_g u|.
\end{equation}

Let $\rho\in[e^{-(l-1)L},e^{-lL}]$, we can assume
\begin{equation}\label{almostflat1}
\|g_{ij}(\rho x)-\delta_{ij}\|_{C^{4}}<\eta(\rho),
\end{equation}
where $\eta(\rho)\rightarrow 0$ as $\rho \rightarrow 0$. Namely, the metric after scaling to $ B _2 \setminus B_ 1$ is
close to the flat metric in $C^ 4 $ norm.

\

We make the following:

\

\noindent{\bf Claim 0}: Actually we can take $\eta(\rho)=O(\rho^2)$.

\

In local coordinates on $ B _2 \setminus B_ 1$, we have
\begin{equation*}
(\nabla_gu)^{i}=g^{ij}\partial_j u=g\sharp \partial u
\end{equation*}
and
\begin{eqnarray*}
\triangle_g u&=&\frac{1}{\sqrt{g}}\partial_i(\sqrt{g}g^{ij} \partial_j u)\\
&=&\frac{1}{\sqrt{g}}(\partial \sqrt{g})\sharp g\sharp \partial u+\partial g\sharp \partial u+g\sharp \partial\partial u,
\end{eqnarray*}
where $g\equiv\det{[g_{ij}]}$ and $\partial g$ denotes the first derivatives of the coefficients of the metric tensor. Note that $\sharp$ in the last two formulas denote Einstein summation in local coordinates, we use the same notation here since both tensor contraction and Einstein summation satisfy estimates of type (\ref{propertytensor}).

Now by the above two formula we rewrite (\ref{extrequ3}) in local coordinates. We will not give the full details here, since it is just a direct computation which is long but not difficult.
We shall compute the term $\nabla_g \triangle _g u \sharp \nabla_g u$, the others follows in a similar way.
\begin{eqnarray*}
& &\nabla_g \triangle _g u \sharp \nabla_g u=g\sharp \partial(\triangle_g u)\sharp g\sharp \partial u\\
&=&g\sharp g\sharp \partial \partial \partial u \sharp g\sharp \partial u+(\frac{1}{\sqrt{g}}g\sharp \partial \sqrt{g}\sharp\partial g+2g\sharp \partial g)\sharp \partial \partial u \sharp g\sharp \partial u\\
&+&(g\sharp\partial\frac{1}{\sqrt{g}} \partial \sqrt{g}\sharp g+ \frac{1}{\sqrt{g}}g\sharp \partial \partial\sqrt{g}  \sharp g+\frac{1}{\sqrt{g}}g\sharp \partial\sqrt{g}\sharp \partial g+g\sharp \partial \partial g)\sharp  \partial u\sharp g\sharp \partial u
\end{eqnarray*}
So by using (\ref{almostflat1}), we have
\begin{equation*}
|\nabla_g \triangle _g u \sharp \nabla_g u-\nabla \triangle u \sharp \nabla u|\leq
C\eta |(\nabla^3 u+\nabla^2 u+\nabla u)\sharp \nabla u|
\end{equation*}
on $B _2 \setminus B_ 1$,
where $g_{kl}\equiv g_{kl}(\rho x)$ and $\nabla, \triangle $ are the gradient operator and Laplacian operator with respect to the flat metric on $ B _2 \setminus B_ 1$.

By repeated computation (with the use of (\ref{almostflat1})) as above, it is easy to see that
(\ref{extrequ3}) (for $u_i(\rho \tilde{x})$) in local coordinates $(\tilde{x})$ on $ B _2 \setminus B_ 1$ is equivalent to
\begin{eqnarray*}
\triangle^2 u&=&\alpha_1(u)\nabla \triangle  u \sharp \nabla u+\alpha_2(u)\nabla^{2} u \sharp \nabla^{2} u\nonumber\\
&+&\alpha_3(u)\nabla^{2} u \sharp \nabla u\sharp \nabla u+\alpha_4(u)\nabla u \sharp \nabla u\sharp \nabla u\sharp \nabla u+\tilde{\tau}(u),
\end{eqnarray*}
where $\tilde{\tau}(u)$ satisfies
\begin{eqnarray*}
|\tilde{\tau}(u)|&\leq &C\eta(|\nabla^4 u|+\cdots+|\nabla u|)+C\eta(|\nabla^3 u|+|\nabla^2 u|+|\nabla u|) |\nabla u |\nonumber\\
&+&C\eta(|\nabla^2 u|+|\nabla u|)(|\nabla^2 u|+|\nabla u|)+C\eta|\nabla u|^4.
\end{eqnarray*}

Therefore from (\ref{energysmallnecki}) and the $\varepsilon$-regularity theorem (Theorem \ref{smallenergythm}) together with the use of Sobolev embedding theorems and H\"{o}lder inequality, we can obtain that for any $p>1$,
\begin{eqnarray}\label{tauestm}
||\tilde{\tau}(u_i(\rho \tilde{x}))||_{L^p(B _2 \setminus B_ 1)}\leq C\varepsilon \eta(\rho).
\end{eqnarray}

After scaling back to the scale of $A_l=B_{e^{-(l-1)L}} \setminus B_{e^{-lL}}$,
the equation (\ref{extrequ3}) for $u_i(x)$  in local coordinates is equivalent to
\begin{eqnarray*}
\triangle^2 u&=&\alpha_1(u)\nabla \triangle  u \sharp \nabla u+\alpha_2(u)\nabla^{2} u \sharp \nabla^{2} u\nonumber\\
&+&\alpha_3(u)\nabla^{2} u \sharp \nabla u\sharp \nabla u+\alpha_4(u)\nabla u \sharp \nabla u\sharp \nabla u\sharp \nabla u+{\tau}(u).
\end{eqnarray*}
It is easy to see that ${\tau(u_i(x))}=\rho^{-4}\tilde{\tau}(u_i(\rho \tilde{x})). $
By using (\ref{tauestm}), we have that
\begin{eqnarray*}\label{tensionfieldest2}
|||x|^{4(1-1/p)}{\tau}(u_i(x))||_{L^p(A_l)}\leq C \varepsilon\eta(e^{-lL}).
\end{eqnarray*}

Next, by applying an argument as in the proof of Lemma 2 of \cite{liu2015finite}, we know that
$v_ i = u _ i - u^*_{i}$
is an $\eta_0$-approximate biharmonic function (w.r.t. the Euclidean metric) in the sense of Definition \ref{defappbiharm} defined on $B _\delta \setminus B_{ \lambda_i R}$,  if $i$ and $R$ are large,  $\delta$ and $\varepsilon$ are small. That is
\begin{eqnarray}\label{similar33}
\triangle^2v_i(r,\theta)&=&a_1\nabla\triangle v_i+a_2\nabla v_i+a_3\nabla v_i+a_4 v_i \nonumber\\
&+&\frac{1}{|\partial B_r|}\int_{\partial B_r}b_1\nabla\triangle v_i+b_2\nabla v_i+b_3\nabla v_i+b_4 v_i d\sigma +h_i(x),
\end{eqnarray}
where the coefficients satisfy the conditions in Definition \ref{defappbiharm} and
\begin{eqnarray*}
h_i=\tau(u_i)-\frac{1}{|\partial B_r|}\int_{\partial B_r(0)}\tau(u_i)d\sigma.
\end{eqnarray*}
 Here $\eta_0>0$ is the small constant in Theorem \ref{thrcircledegeneck}.

Now, we are able to show the decay on the neck for the tangential part energy. Thanks to {\bf Claim 0}, we have that
\begin{eqnarray*}
& &|||x|^{4(1-1/p)}h_i||_{L^p(A_l)}\\
&\leq&C||\tilde{\tau}(u_i) ||_{L^p(B_2\setminus B_1)} =O(\varepsilon e^{-2lL})=O(\varepsilon|x|^2).
\end{eqnarray*}
In particular, $h_i$ satisfies the condition (b) in Definition \ref{defappbiharm} when $\varepsilon>0$ is small enough. Therefore, by applying Corollary \ref{threecirtotangdecay}, we have for any $\varepsilon>0$, if $i$ and $R$ are large enough and $\delta>0$ is sufficiently small, then
\begin{eqnarray*}\label{tangentialdecay}
& &\int_{(-lL,-(l-1)L)\times S^3}(|\triangle_{S^3}u_i|^2+|\partial_t \nabla_{S^3}u_i|^2+|\partial_t \triangle_{S^3}u_i|^2)dt d\theta\nonumber\\
&\leq&C\varepsilon^2\left(e^{-(l-l_0)L}+e^{-(l_i-l)L}\right),
\end{eqnarray*}
where $\log \delta=-l_0L$ and $\log\lambda_i R=-l_iL$.
\end{proof}

To complete the above proof, now we prove \noindent{\bf Claim 0}.

\begin{proof}[\textbf{Proof of \noindent{\bf Claim 0}}]

Since the metric $g$ is fixed, we can assume that the metric on $B_{3\rho}$ has bounded geometry, namely, the norm of the Riemannian curvature tensor $|R_g|_g$ and the norms of its covariant derivatives $|\nabla_g R_g|_g$, $|\nabla^2_{g}R_g|_g$ are bounded. Then by the scaling property of these tensors, we have that the metric $\tilde{g}=\rho^2 g$ on $ B _2 $ (w.r.t $\tilde{g}$)
satisfies
\begin{eqnarray*}
|R_{\tilde{g}}|_{\tilde{g}}<C\rho^2, \quad
|\nabla_{\tilde{g}} R_{\tilde{g}}|_{\tilde{g}}<C\rho^3,\quad
|\nabla^2_{\tilde{g}}R_{\tilde{g}}|_{\tilde{g}}<C\rho^4.
\end{eqnarray*}
Then in some coordinates $\tilde{x}$ (see Chapter 5 of \cite{schoen1994lectures} or Section 5 of \cite{lee1987yamabe})
\begin{eqnarray*}
{\tilde{g}}_{pq}(\tilde{x})&=&\delta_{pq}+1/3R_{{\tilde{g}};pijq}\tilde{x}^i\tilde{x}^j+1/6R_{\tilde{g};pijq,k}\tilde{x}^i\tilde{x}^j\tilde{x}^k\\
& &+(1/20 R_{{\tilde{g}};pijq,kl}+2/45R_{{\tilde{g}};pijm}R_{{\tilde{g}};qklm})\tilde{x}^i\tilde{x}^j\tilde{x}^k\tilde{x}^l+O(\rho^5),
\end{eqnarray*}
Notice that if we set $x=\rho\tilde{x}$, then  ${\tilde{g}}_{ij}(\tilde{x})=g_{ij}(x)$, hence
\begin{eqnarray*}\label{almostflat2}
g_{ij}(\rho \tilde{x})-\delta_{ij}&<&C\rho^2,\nonumber\\
\partial_kg_{ij}(\rho \tilde{x}),\partial_k\partial_lg_{ij}(\rho \tilde{x})&<&C\rho^2,\\
\partial^3g_{ij}(\rho \tilde{x})<C\rho^3, \,\,\partial^4g_{ij}(\rho \tilde{x})&<&C\rho^4,\nonumber
\end{eqnarray*}
where $\partial_i u\equiv\frac{\partial}{\partial \tilde{x}^i}u$ and  $\partial^3$ represents all the third derivatives (w.r.t. $\tilde{x}$).
\end{proof}

\subsection{Poho\v{z}aev identity on Riemannian manifolds} \label{Pohogenermetric}
Recall that Poho\v{z}aev type estimates or identities for biharmonic maps on  $B_r(0)\subset \mathbb{R}^4$ were already derived in \cite{wang2004remarks, hornung2012energy, wang2012energysphere, wang2012energy, liu2015finite,liu2016neck,laurain2013energy}.

In this subsection, we shall establish the Poho\v{z}aev identity for extrinsic biharmonic maps defined on general Riemannian manifolds. Let $(B_1(p), g)$ be a geodesic ball in $(M,g)$. Let $(x)$ be the normal coordinates centered at the point $p$, and set $e_i=(\frac{\partial}{\partial x^i})_p$, $i=1,\cdots,4$. For any $q\in B_1(p)$, we can obtain an orthogonal frame $\{e_i(q)\}, i=1,\cdots 4,$ by parallel transport along the geodesic line from $p$ to $q$.
Set $r^2=\sum_{i=1}^{4}x_{i}^2$, then we have $r\partial_r=\sum_{i=1}^{4}x^ie_i$.

Firstly, we multiply the extrinsic biharmonic map equation by $x^k e_k(u)$ and then we do integration by part. Thanks to the fact that $\triangle^2_{g}u \bot T_u N$ almost everywhere for an extrinsic biharmonic map $u$, we have
\begin{eqnarray*}
0& =&\int_{B_r(p)}(x^k e_k(u))(\triangle^2_{g} u)dV_g\\
&=&-\int_{B_r(p)}\langle \nabla_g(x^k)e_k u,\nabla_g(\triangle_g u)\rangle dV_g-\int_{B_r(p)}\langle x^k \nabla_g e_ku, \nabla_g\triangle u\rangle\\
& & +\int_{\partial B_r(p)}(x^k e_k(u))(\partial_r\triangle_{g} u)d \sigma_g\\
&=&\int_{B_r(p)}(\triangle_g x^k)e_ku (\triangle _g u )d V_g+2\int_{B_r(p)}(\nabla_gx^k)\nabla_g(e_k u)\triangle_g udV_g\\
& &-\int_{\partial B_r(p)}\partial_r(x^k e_ku)\triangle _g u d\sigma_g
+\int_{B_r(p)}x^k \triangle_g(e_k u)\triangle_g u dV_g\\
& &
+\int_{\partial B_r(p)}(x^k e_k(u))(\partial_r\triangle_{g} u)d \sigma_g.
\end{eqnarray*}

Next, we show that $$\int_{B_r(p)}x^k \triangle_g(e_k u)\triangle_g u dV_g$$ is
 closely associated with the following quantity $$r\int_{\partial B_r(p)}\frac{|\triangle_g u|^2}{2}d\sigma_g-\int_{B_r(p)} \text{div}_g(r\partial_r)\frac{|\triangle_g u|^2}{2} dV_g. $$
Direct calculations show
\begin{eqnarray*}
& &r\int_{\partial B_r(p)}\frac{|\triangle_g u|^2}{2}d\sigma_g=\int_{B_r(p)}
\text{div}(\frac{|\triangle_g u|^2}{2}x^ke_k)dV_g\\
&=&\int_{B_r(p)}x^ke_k(\triangle_g u) \triangle_gu dV_g+\int_{B_r(p)}\text{div}_g(r\partial_r)\frac{|\triangle_g u|^2}{2} dV_g,
\end{eqnarray*}
By the Ricci identity,
\begin{equation*}
x^k \triangle_ge_k u=x^ke_k(\triangle_g u)+Ric_g(\nabla_g u, x^ke_k).
\end{equation*}
It follows that
\begin{eqnarray*}
& &\int_{B_r(p)}x^k (\triangle_ge_k u )\triangle_gu dV_g\\
&=&\int_{B_r(p)}[x^ke_k(\triangle_g u)+Ric_g(\nabla_g u, x^ke_k)]\triangle_gu dV_g\\
&=&r\int_{\partial B_r(p)}\frac{|\triangle_g u|^2}{2}d\sigma_g+\int_{B_r(p)}Ric_g(\nabla_g u, x^ke_k)\triangle_gu -\text{div}_g(r\partial_r)\frac{|\triangle_g u|^2}{2} dV_g.
\end{eqnarray*}

Therefore, by putting the above formulas together, we have the following Poho\v{z}aev identity for biharmonic maps over Riemannian geodesic balls:
\begin{eqnarray*}
0&=&\int_{B_r(p)}(\triangle_g x^k)e_ku (\triangle _g u )d V_g+2\int_{B_r(p)}(\nabla_gx^k)\nabla_g(e_k u)\triangle_g udV_g\\
& &
-\int_{B_r(p)}4\frac{|\triangle_g u|^2}{2} dV_g
-\int_{\partial B_r(p)}\partial_r(x^k e_ku)\triangle _g u d\sigma_g\\
& &
+\int_{\partial B_r(p)}(x^k e_k(u))(\partial_r\triangle_{g} u)d \sigma_g+r\int_{\partial B_r(p)}\frac{|\triangle_g u|^2}{2}d\sigma_g\\
& &+\int_{B_r(p)}Ric_g(\nabla_g u, x^ke_k)\triangle_g u dV_g-\int_{B_r(p)}(\text{div}_g(r\partial_r)-4)\frac{|\triangle_g u|^2}{2} dV_g.
\end{eqnarray*}

Finally, by putting the boundary terms to the left hand side, we have that
\begin{eqnarray}\label{pohoidentgenrl1}
& &\int_{\partial B_r(p)}(r\partial_ru)(\partial_r\triangle_{g} u)+r\frac{|\triangle_g u|^2}{2} -\partial_r(r\partial_ru)\triangle _g u d\sigma_g\nonumber\\
&=&-\int_{B_r(p)}(\triangle_g x^k)e_ku (\triangle _g u )-2\int_{B_r(p)}(\nabla_gx^k)\nabla_g(e_k u)\triangle_g u-|\triangle_g u|^2dV_g\\
& &+\int_{B_r(p)}(\text{div}_g(r\partial_r)-4)\frac{|\triangle_g u|^2}{2} dV_g
- \int_{B_r(p)}Ric_g(\nabla_g u, x^ke_k)\triangle_g udV_g.\nonumber
\end{eqnarray}

We shall denote the integrand in the left hand side of (\ref{pohoidentgenrl1}) by $\square$, and denote the whole right hand side integration term by $ P(g,r)$.

To make better use of (\ref{pohoidentgenrl1}), we need to decompose $\triangle_g u$ into the radial part and the tangential part.
Let $ \widetilde{\partial B_r(p)} $ be the sphere $\partial B_r(p)$ with the metric $\frac{g_r}{r^2}$, where $g_r$ is the metric of $\partial B_r(p)$ induced from $M$.
Then we have
\begin{equation*}
\triangle_g u=\partial^2_{r}u+ \frac{\tilde{m}}{r}\partial_ru+\frac{1}{r^2}\tilde{\triangle}u,
\end{equation*}
where $\tilde{m}$ is the mean curvature, $\tilde{\nabla}$ is the gradient operator, and $\tilde{\triangle}$ is the associated Laplace operator of $ \widetilde{\partial B_r(p)} $.
And direct computations yield
\begin{eqnarray*}
(\triangle_{g}u)^2&=&(\partial^{2}_ru)^2+(\frac{\tilde{m}}{r})^2(\partial_r u)^2+\frac{1}{r^4}(\tilde{\triangle}u)^2\\
 & &+(\frac{2\tilde{m}}{r})(\partial_r u)(\partial_r^{2} u)^2+\frac{2}{r^2}(\tilde{\triangle}u)(\partial_r^{2} u)+(\frac{2\tilde{m}}{r^3})(\partial_r u)(\tilde{\triangle}u).
\end{eqnarray*}

The integrand term in the left hand side of (\ref{pohoidentgenrl1}) $$\square=(r\partial_ru)(\partial_r\triangle_{g} u)+r\frac{|\triangle_g u|^2}{2} -\partial_r(r\partial_ru)\triangle _g u $$ consists of the following two terms denoted by $r\square_1$ and $r\square_2$:

\begin{eqnarray}\label{leftterm11}
\square_1&\equiv&(\frac{1}{2}(\triangle_{g}u)^2-(\partial^2_{r}u)\triangle _g u)\nonumber\\
&=&-1/2(\partial^{2}_ru)^2+
\left(\frac{\tilde{m}^2}{2r^2}\right)(\partial_r u)^2+\frac{1}{2r^4}(\tilde{\triangle}u)^2+\left(\frac{\tilde{m}}{r^3}\right)(\partial_r u)(\tilde{\triangle}u),
\end{eqnarray}
\begin{eqnarray}\label{leftterm22}
\square_2 &\equiv& (\partial_ru)(\partial_r\triangle_{g} u)-\frac{1}{r}(\partial_ru)(\triangle_{g} u)\nonumber\\
&=&(\partial_r u)(\partial^{3}_r u)+\frac{\tilde{m}-1}{r}(\partial_r u)(\partial^{2}_r u)+\left(\frac{\partial_r\tilde{m}}{r}-\frac{2\tilde{m}}{r^2}\right)(\partial_r u)^2\\
& &+\frac{1}{r^2}(\partial_r u)(\partial_r\tilde{\triangle} u)-\frac{\tilde{m}}{r^3}(\partial_r u)(\tilde{\triangle} u).\nonumber
\end{eqnarray}

So, the identity (\ref{pohoidentgenrl1}) can be written briefly as
\begin{equation}\label{pohozaevmodnondege}
\int_{\partial B_r(p)} r(\square_1+\square_2 )d\sigma_g=P(g,r).
\end{equation}

Next, we shall estimate $P(g,r)$ in (\ref{pohozaevmodnondege}). It turns out that this estimate is crucial in our proof of Theorem \ref{noneckthm}. A key observation is that the quantity $P(g,r)$ involves at most the second derivatives of $u$. We remark that the idea of putting the effect of the metric $g$ into the bi-tension field (or error term) does not work for the Poho\v{z}aev type argument.

For a fixed Riemannian manifold $(M,g)$, we can assume that (set $\tilde{x}=\rho^{-1}x$)
\begin{equation*}\label{almostflatmetricnondege}
\|g_{ij}(\rho \tilde{x})-\delta_{ij}\|_{C^{4}(B_2)}<C\rho^2.
\end{equation*}
 Namely, the metric after scaling is
close to the flat metric in $C^ 4 $ norm.

\begin{lem}\label{almostflatconseq}
Let $x^i$, $e_i$, $\tilde{\triangle}$ and $\tilde{m}$ be as above, then it is not hard to check that for all $q\in B_{2\rho}(p)$, the followings hold:

\begin{itemize}
\item[1)] $$|e_i(x^j)-\delta_{i}^j| <C\rho^2;$$

\item[2)] for $i=1, \cdots,4$, $$ |\triangle _g x^i| <C\rho;$$

\item[3)] on $ \widetilde{\partial B_\rho(p)} $, for any smooth function $f$,
\begin{eqnarray*}
  |\nabla_{S^3}f-\tilde{\nabla}f|& <&C\rho^2|\nabla_{S^3}f|,\\
  |\triangle_{S^3}f-\tilde{\triangle}f| &<&C\rho^2(|\triangle_{S^3}f|+|\nabla_{S^3}f|),
\end{eqnarray*}
where the norm is dependent on the standard metric on $S^3$ and it is equivalent to the norm defined by $\frac{g_r}{r^2}$,
and
$$|\partial_r\triangle_{S^3}f-\partial_r\tilde{\triangle}f| <C\rho^2(|\partial_r\triangle_{S^3}f|+|\partial_r\nabla_{S^3}f|)+C\rho(|\triangle_{S^3}f|+|\nabla_{S^3}f|);$$

\item[4)] on $ \widetilde{\partial B_\rho(p)} $, $$\triangle_g r=\frac{3+O(r^2)}{r} ,\,\, \tilde{m}=r\triangle_g r=3+O(r^2);$$

\item[5)] on $ \widetilde{\partial B_\rho(p)} $, $$ |\partial_r\tilde{m}|<C\rho;$$

\item[6)] $$\text{div}_g(r\partial_r)-4=r\triangle_g r+1-4=O(r^2).$$

\end{itemize}

\end{lem}

\

We put the proof of Lemma \ref{almostflatconseq} in Appendix \ref{pftwolem}.

Recall that
\begin{eqnarray*}\label{pohoidentgenrl}
P(g,r)
&=&-\int_{B_r(p)}(\triangle_g x^k)e_ku (\triangle _g u )-2\int_{B_r(p)}(\nabla_gx^k)\nabla_g(e_k u)\triangle_g u-|\triangle_g u|^2dV_g\\
& &+\int_{B_r(p)}(\text{div}_g(r\partial_r)-4)\frac{|\triangle_g u|^2}{2} dV_g
- \int_{B_r(p)}Ric_g(\nabla_g u, x^ke_k)\triangle_g udV_g.\nonumber
\end{eqnarray*}
With the help of Lemma \ref{almostflatconseq}, we can easily derive the following error estimate:
\begin{lem}\label{rigttermestlem11}
\begin{equation}\label{rigttermest}
  P(g,r)=O(r^2).
\end{equation}
\end{lem}
\begin{proof}
Indeed, by 1) of Lemma \ref{almostflatconseq},
\begin{eqnarray*}
  (\nabla_gx^k)\nabla_g(e_k u)&=&\sum_{l=1}^{4}e_lx^k\nabla_{e_l}(e_k u)=e_kx^k\nabla_{e_k}(e_k u)+\sum_{l\neq k}e_lx^k\nabla_{e_l}(e_k u)\\
  &=&O(r^2)|\nabla^{2}_g u|+\triangle_g u.
\end{eqnarray*}
So thanks to
$$ \int_{B_1}|\nabla^2_{g} u|^2+|\nabla_{g} u|^4dV_g<C,$$
it follows that
\begin{eqnarray*}
 & & \int_{B_r(p)}(\nabla_gx^k)\nabla_g(e_k u)\triangle_gu-{|\triangle_g u|^2}  dV_g\\
 &=&\int_{B_r(p)}  O(r^2)|\nabla^{2}_g u|^2  dV_g=O(r^2).
\end{eqnarray*}

By using H\"{o}lder inequality, we have
\begin{eqnarray*}
& &\int_{B_r(p)}|\triangle_g u||\nabla_g u|dV_g\\
&\leq&(\int_{B_r(p)}|\nabla^{2}_{g} u|^2dV_{g})^{1/2}(\int_{B_r(p)}|\nabla_{g} u|^4dV_{g})^{1/4}(\int_{B_r(p)}dV_{g})^{1/4}
\leq Cr.
\end{eqnarray*}
Notice that $Ric_g$ is bounded, $\triangle _g x^i=O(r )$ and  $\text{div}_g(r\partial_r)-4=O(r^2)$ by Lemma \ref{almostflatconseq}, then (\ref{rigttermest}) follows from simple computations.
\end{proof}

Now we shall reveal from (\ref{pohozaevmodnondege}) the relationship between the tangential part energy and the radial part energy, see (\ref{pohoradictangen3}) and (\ref{pohoradictangen4}) in the following arguments.

Integrating by part, we get
\begin{eqnarray}\label{coaera333}
& &\int_{B_R\setminus B_r}(\partial_r u)(\partial^{3}_r u)dV_g=\int_{r}^{R}ds\int_{\partial B_s}(\partial_s u)(\partial^{3}_s u)d\sigma_g\nonumber\\
&=&\int_{r}^{R}ds\int_{S^3}(s^3\sqrt{g(s,\theta)})(\partial_s u)(\partial^{3}_s u)d\theta\\
&=&\int_{r}^{R}ds\frac{d}{ds}\int_{S^3}(s^3\sqrt{g(s,\theta)})(\partial_s u)(\partial^{2}_s u)d\theta\nonumber\\
& &-\int_{r}^{R}ds\int_{S^3}(\frac{d}{ds}(s^3\sqrt{g(s,\theta)}))(\partial_s u)(\partial^{2}_s u)d\theta-\int_{r}^{R}ds\int_{S^3}(s^3\sqrt{g(s,\theta)})(\partial_s ^2 u)^2d\theta.\nonumber
\end{eqnarray}
Here $\sqrt{g}$ is the square root of the determinant of the metric $g$.
Set $A(r, \theta)=r^3\sqrt{g(r,\theta)}$, notice that
\begin{equation*}
\frac{\partial_r A(r,\theta)}{A(r,\theta)}=\triangle_gr=\frac{3}{r}+O(r),
\end{equation*}
we have
\begin{eqnarray}\label{coaera444}
\int_{B_R\setminus B_r}(\partial_r u)(\partial^{3}_r u)dV_g&=&\int_{\partial(B_R\setminus B_r)}(\partial_r u)(\partial^{2}_r u)\\
& &-\int_{B_R\setminus B_r}(\partial^{2}_r u)^2+\left(\frac{3}{r}+O(r)\right)(\partial_r u)(\partial^{2}_r u)dV_g.\nonumber
\end{eqnarray}

Recall that
\begin{eqnarray*}\label{leftterm2}
\square_2
&=&(\partial_r u)(\partial^{3}_r u)+\frac{\tilde{m}-1}{r}(\partial_r u)(\partial^{2}_r u)+\left(\frac{\partial_r\tilde{m}}{r}-\frac{2\tilde{m}}{r^2}\right)(\partial_r u)^2\\
& &+\frac{1}{r^2}(\partial_r u)(\partial_r\tilde{\triangle} u)-\frac{\tilde{m}}{r^3}(\partial_r u)(\tilde{\triangle} u).\nonumber
\end{eqnarray*}
Therefore we have
\begin{eqnarray*}
& &\int_{B_R\setminus B_r}\square_2dV_g\\
&=&\int_{\partial(B_R\setminus B_r)}(\partial_r u)(\partial^{2}_r u)d\sigma_g
-\int_{B_R\setminus B_r}(\partial^{2}_r u)^2+\left(\frac{3}{r}+O(r)\right)(\partial_r u)(\partial^{2}_r u) dV_g\\
& &+\int_{B_R\setminus B_r}\frac{\tilde{m}-1}{r}(\partial_r u)(\partial^{2}_r u)+\left(\frac{\partial_r\tilde{m}}{r}-\frac{2\tilde{m}}{r^2}\right)(\partial_r u)^2dV_g\\
& &+\int_{B_R\setminus B_r}\frac{1}{r^2}(\partial_r u)(\partial_r\tilde{\triangle} u)-\frac{\tilde{m}}{r^3}(\partial_r u)(\tilde{\triangle} u)dV_g.
\end{eqnarray*}
In the above $\int_{\partial(B_R\setminus B_r)}f d\sigma_g\equiv \int_{\partial B_R} f d\sigma_g -\int_{\partial B_r} f d\sigma_g.$

Then by direct calculations using 4) and 5) of Lemma \ref{almostflatconseq}, if we choose $R=cr$ for some $c>1$, we have
\begin{eqnarray*}
& &\int_{B_R\setminus B_r}\square_2dV_g\\
&=&\int_{\partial(B_R\setminus B_r)}(\partial_r u)(\partial^{2}_r u)d\sigma_g
+\int_{B_R\setminus B_r}\left(O(1)-\frac{6+O(r^2)}{r^2}\right)(\partial_r u)^2-(\partial^{2}_r u)^2dV_g\\
& &+\int_{B_R\setminus B_r}\frac{2+O(r^2)}{r}(\partial_r u)(\partial^{2}_r u)- \left(\frac{3}{r}+O(r)\right)(\partial_r u)(\partial^{2}_r u)dV_g\\
& &+\int_{B_R\setminus B_r}\frac{1}{r^2}(\partial_r u)(\partial_r\tilde{\triangle} u)-\frac{\tilde{m}}{r^3}(\partial_r u)(\tilde{\triangle} u)dV_g.
\end{eqnarray*}
Hence
\begin{eqnarray*}
& &\int_{B_R\setminus B_r}\square_2dV_g\\
&=&\int_{\partial(B_R\setminus B_r)}(\partial_r u)(\partial^{2}_r u)d\sigma_g
+\int_{B_R\setminus B_r}\left(-\frac{5}{r^2}+O(1)\right)(\partial_r u)^2-(\partial^{2}_r u)^2dV_g\\
& &+\int_{B_R\setminus B_r}\left(-\frac{1}{2r}+O(r)\right)\partial_r (\partial_r u)^2-\frac{1}{r^2}(\partial_r u)^2dV_g\\
& &+\int_{B_R\setminus B_r}\frac{1}{r^2}(\partial_r u)(\partial_r\tilde{\triangle} u)-\frac{\tilde{m}}{r^3}(\partial_r u)(\tilde{\triangle} u)dV_g.
\end{eqnarray*}
Next by the use of the formula $\frac{\partial_r A(r,\theta)}{A(r,\theta)}=\triangle_gr=\frac{3}{r}+O(r)$ again, we get
\begin{eqnarray}\label{coarea1}
\int_{\partial(B_R\setminus B_r)}\frac{1}{2s}(\partial_s u)^2d\sigma_g
&=&\int_{r}^{R}\frac{d}{ds}\int_{\partial B_s}\frac{1}{2s}(\partial_s u)^2d\sigma_gds\nonumber\\
&=&\int_{r}^{R}\frac{d}{ds}\int_{S^3}\frac{1}{2s}(\partial_s u)^2A(s, \theta)d\theta ds\nonumber\\
&=&\int_{r}^{R}\int_{\partial B_s}\frac{1}{2s}\partial_s (\partial_s u)^2-\frac{1}{2s^2}(\partial_s u)^2+\frac{1}{2s}(\triangle_g r)(\partial_s u)^2d\sigma_g ds\\
&=& \int_{r}^{R}\frac{1}{2s}\partial_s (\partial_s u)^2+\frac{1}{s^2}(\partial_s u)^2+O(1)(\partial_s u)^2d\sigma_g\nonumber\\
&=&\int_{B_R\setminus B_r}\frac{1}{2s}\partial_s (\partial_s u)^2+\frac{1}{s^2}(\partial_s u)^2+O(1)(\partial_s u)^2dV_g.\nonumber
\end{eqnarray}
Therefore
\begin{eqnarray*}
& &\int_{B_R\setminus B_r}\square_2dV_g\\
&=&\int_{\partial(B_R\setminus B_r)}(\partial_r u)(\partial^{2}_r u)d\sigma_g
+\int_{B_R\setminus B_r}\left(O(1)-\frac{5}{r^2}\right)(\partial_r u)^2-(\partial^{2}_r u)^2dV_g\\
& &-\int_{\partial(B_R\setminus B_r)}\frac{1}{2r}(\partial_r u)^2d\sigma_g+\int_{B_R\setminus B_r}O(r)\partial_r (\partial_r u)^2dV_g\\
& &+\int_{B_R\setminus B_r}\frac{1}{r^2}(\partial_r u)(\partial_r\tilde{\triangle} u)-\frac{\tilde{m}}{r^3}(\partial_r u)(\tilde{\triangle} u)dV_g.
\end{eqnarray*}

By putting the boundary terms together, we get
\begin{eqnarray}\label{leftterm3}
& &\int_{B_R\setminus B_r}\square_2dV_g\nonumber\\
&=&\int_{\partial(B_R\setminus B_r)}(\partial_r u)(\partial^{2}_r u)-\frac{1}{2r}(\partial_r u)^2d\sigma_g\nonumber\\
& &+\int_{B_R\setminus B_r}\left(O(1)-\frac{5}{r^2}\right)(\partial_r u)^2-(\partial^{2}_r u)^2+O(r)\partial_r (\partial_r u)^2dV_g\\
& &+\int_{B_R\setminus B_r}\frac{1}{r^2}(\partial_r u)(\partial_r\tilde{\triangle} u)-\frac{\tilde{m}}{r^3}(\partial_r u)(\tilde{\triangle} u)dV_g.\nonumber
\end{eqnarray}

By integrating (\ref{pohozaevmodnondege}) we get
\begin{eqnarray}\label{pohozaevmo22}
\int_{B_R\setminus B_r}(\square_1+\square_2)dV_{g}
=\int_r^{R}\int_{\partial B_s }(\square_1+\square_2)d\sigma_{g}ds
=\int_r^{R}\frac{1}{s}P(g,s)ds.
\end{eqnarray}
Recall
\begin{eqnarray*}\label{leftterm1}
\square_1
=-1/2(\partial^{2}_ru)^2+
(\frac{\tilde{m}^2}{2r^2})(\partial_r u)^2+\frac{1}{2r^4}(\tilde{\triangle}u)^2+(\frac{\tilde{m}}{r^3})(\partial_r u)(\tilde{\triangle}u).
\end{eqnarray*}
Then by combining (\ref{pohozaevmo22}), (\ref{leftterm3}) and (\ref{rigttermest}) (namely, $P(g,r)=O(r^2)$) together, we have
\begin{eqnarray}\label{pohoradictangen}
& &\int_{B_R\setminus B_r}\frac{3}{2}(\partial^{2}_r u)^2+\left(\frac{10-\tilde{m}^2}{2r^2}+O(1)\right)(\partial_r u)^2 dV_g\nonumber\\
&\leq&\int_{B_R\setminus B_r}\frac{1}{2r^4}(\tilde{\triangle}u)^2+\frac{1}{r^2}|(\partial_r u)(\partial_r\tilde{\triangle}u)|dV_g\\
& &+\int_{\partial(B_R\setminus B_r)}(\partial_r u)(\partial^{2}_r u)-\frac{1}{2r}(\partial_r u)^2d\sigma_g\nonumber\\
& &+\int_{B_R\setminus B_r}O(r)|\partial_r (\partial_r u)^2|dV_g+\int_{r}^{R}O(s)ds.\nonumber
\end{eqnarray}

Next we estimate $\int_{B_R\setminus B_r}O(r)|\partial_r (\partial_r u)^2|dV_g$. Since
$$ \int_{B_1}|\nabla^2_{g} u|^2+|\nabla_{g} u|^4dV_g$$
is bounded, by H\"{o}lder inequality, we have
\begin{equation*}
\int_{B_R\setminus B_r}|\partial_r (\partial_r u)^2|dV_g\leq Cr.
\end{equation*}
And Cauchy-Schwarz inequality gives us
\begin{eqnarray*}
& &\int_{B_{cr}\setminus B_r}\frac{1}{r^2}(\partial_r u)(\partial_r\tilde{\triangle}u)dV_g\\
&\leq&\int_{B_{cr}\setminus B_r}\frac{1}{r^2} \left(\frac{1}{32}(\partial_r u)^2+8(\partial_r\tilde{\triangle}u)^2\right)dV_g.
\end{eqnarray*}
So (\ref{pohoradictangen}) and 4) of Lemma \ref{almostflatconseq} imply that
\begin{eqnarray}\label{pohoradictangen3}
& &\int_{B_{cr}\setminus B_r}\frac{3}{2}(\partial^{2}_r u)^2+\frac{1}{3r^2}(\partial_r u)^2 dV_g\nonumber\\
&\leq&\int_{B_{cr}\setminus B_r}\frac{1}{2r^4}(\tilde{\triangle}u)^2+\frac{8}{r^2}(\partial_r\tilde{\triangle}u)^2dV_g\\
& &+\int_{\partial(B_{cr}\setminus B_r)}(\partial_r u)(\partial^{2}_r u)-\frac{1}{2r}(\partial_r u)^2d\sigma_g+Cr^2.\nonumber
\end{eqnarray}

Set $r_2(s)=e^{t_0+s}$ and $r_1(s)=e^{t_0-s}$, here $\log(\lambda_iR)<t_0< \log\delta$, and
$$0\leq s \leq \min \left(-t_0+\log\delta,-\log(\lambda_iR)+t_0 \right).$$
Since we can chose $c$ such that
$B_{r_2}\setminus B_{r_1}=\bigcup B_{c^l {r_1}}\setminus B_{r_1} $,
by adding up inequalities as (\ref{pohoradictangen3}), we have
\begin{eqnarray}\label{pohoradictangen4}
& &\int_{B_{r_2}\setminus B_{r_1}}\frac{3}{2}(\partial^{2}_r u)^2+\frac{1}{3r^2}(\partial_r u)^2 dV_g\nonumber\\
&\leq&\int_{B_{r_2}\setminus B_{r_1}}\frac{1}{2r^4}(\tilde{\triangle}u)^2+\frac{8}{r^2}(\partial_r\tilde{\triangle}u)^2dV_g\\
& &+\int_{\partial({B_{r_2}\setminus B_{r_1}})}(\partial_r u)(\partial^{2}_r u)-\frac{1}{2r}(\partial_r u)^2d\sigma_g+\int_{r_1}^{r_2}O(r)dr.\nonumber
\end{eqnarray}

Through the above efforts, now we are ready to prove the energy identity and the no neck identity by the relationship between the tangential part energy and the radial part energy reflected in (\ref{pohoradictangen3}) and (\ref{pohoradictangen4}).

\begin{proof}[\bf Proof of Theorem \ref{noneckthm}]
If we set $c=e^L$, $r=e^t$, then from the estimate
\begin{eqnarray*}
& &\int_{B_{e^L r}\setminus B_r}\frac{1}{r^4}\left((\triangle_{S^3}u_j)^2+|\nabla_{S^3}u_j|^4\right)
+\frac{1}{r^2}|\partial_r\nabla_{S^3}u_j|^2+\frac{1}{r^2}|\partial_r\triangle_{S^3}u_j|^2dx\\
&\leq&C\varepsilon^2 \left(\frac{r}{\delta}+\frac{\lambda_i R}{r}\right)
\end{eqnarray*}
in Lemma \ref{lemtangdecay} and 3)
in Lemma \ref{almostflatconseq} which indicates that the operators $\tilde{\nabla}$ and $\tilde{\triangle}$ are almost equal to $\nabla_{S^3}$ and $\triangle_{S^3}$, we get by Cauchy-Schwarz inequality that
\begin{eqnarray}\label{equaldecay11}
& &\int_{B_{cr}\setminus B_r}\frac{1}{2r^4}(\tilde{\triangle}u_i)^2+\frac{8}{r^2}|(\partial_r\tilde{\triangle}u_i)|^2dV_g\\
&\leq&C\varepsilon^2\left(\frac{r}{\delta}+\frac{\lambda_i R}{r}\right)  \nonumber.
\end{eqnarray}

Recall that $A_l=B_{e^{-(l-1)L}} \setminus B_{e^{-lL}}$, $\log \delta=-l_0L$ and $\log\lambda_i R=-l_iL$.
Then by (\ref{pohoradictangen3}) and (\ref{equaldecay11}) we have
\begin{eqnarray*}\label{pohoradictangen2}
& &\int_{B_{\delta}\setminus B_{\lambda_iR}}\frac{3}{2}(\partial^{2}_r u_i)^2+\frac{1}{3r^2}(\partial_r u_i)^2 dV_g\nonumber\\
&\leq&\sum_{l_0}^{l_i}\int_{A_l}\frac{1}{2r^4}(\tilde{\triangle}u_i)^2+\frac{8}{r^2}|\partial_r\tilde{\triangle}u_i|^2dV_g\\
& &+\int_{\hat{\partial}(B_{\delta}\setminus B_{\lambda_iR})}|(\partial_r u_i)(\partial^{2}_r u_i)|+\frac{1}{2r}(\partial_r u_i)^2d\sigma_g+C\sum_{l_0}^{l_i}e^{-2lL}\nonumber\\
&\leq&C\sum_{l_0}^{l_i}\varepsilon^2 \left(\frac{e^{-lL}}{\delta}+\frac{\lambda_i R}{e^{-lL}}\right)+C\sum_{l_0}^{l_i}e^{-2lL}   \\
& &+\int_{\hat{\partial}(B_{\delta}\setminus B_{\lambda_iR})}|(\partial_r u_i)(\partial^{2}_r u_i)|+\frac{1}{2r}(\partial_r u_i)^2d\sigma_g.
\end{eqnarray*}
In the above, $\hat{\partial}(B_{\delta}\setminus B_{\lambda_iR})=\partial B_{\delta}\cup \partial B_{\lambda_iR}.$
Next by (\ref{energysmallnecki}), we have from the $\varepsilon$-regularity Theorem \ref{smallenergythm} and the Sobolev embedding theorem that
\begin{equation*}
\int_{\hat{\partial}(B_{\delta}\setminus B_{\lambda_iR})}|(\partial_r u_i)(\partial^{2}_r u_i)|+\frac{1}{2r}(\partial_r u_i)^2d\sigma_g\leq C\varepsilon^2.
\end{equation*}
Therefore we have for any $\varepsilon$, there are constants $R_0$, $\delta_0$ and $i_0$ dependent only on $\varepsilon$ , if $i>i_0$, $R>R_0$ and $\delta<\delta_0$, then
\begin{equation}\label{radialenergyestimate}
\int_{B_{\delta}\setminus B_{\lambda_iR}}(\partial^{2}_r u_i)^2+\frac{1}{r^2}(\partial_r u_i)^2 dV_g\leq C\varepsilon^2.
\end{equation}
Combining the last inequality with (\ref{equaldecay11}), we have
\begin{equation*}\label{energyidentityformula}
\lim_{\delta\rightarrow 0} \lim_{R\rightarrow \infty}\lim_{i \rightarrow \infty}\int_{B_\delta \setminus B_{\lambda_i R}}|\nabla^2_{g}u_i|^2+|\nabla_{g} u_i|^4 dV_g=0.
\end{equation*}
So far we have proved the energy identity.

Next, we shall prove the no neck property. Recall that $r_2(s)=e^{t_0+s}$ and $r_1(s)=e^{t_0-s}$, we define
\begin{equation*}
F(s)=\int_{B_{r_2}\setminus B_{r_1}}\frac{3}{2}(\partial^{2}_r u_i)^2+\frac{1}{3r^2}(\partial_r u_i)^2 dV_g.
\end{equation*}
Hence
\begin{eqnarray*}
\partial_sF(s)&=&e^{t_0+s}\int_{\partial B_{r_2} }\frac{3}{2}(\partial^{2}_r u_i)^2+\frac{1}{3r^2}(\partial_r u_i)^2 d\sigma_g\\
& &+e^{t_0-s}\int_{\partial B_{r_1} }\frac{3}{2}(\partial^{2}_r u_i)^2+\frac{1}{3r^2}(\partial_r u_i)^2 d\sigma_g.
\end{eqnarray*}
By applying Cauchy-Schwarz inequality ($\sqrt{3}\leq 1/2(3+1)$), we get
\begin{eqnarray*}
& &\int_{\partial(B_{r_2}\setminus B_{r_1})}(\partial_r u_i)(\partial^{2}_r u_i)d\sigma_g \\
&\leq &\frac{1}{\sqrt{3}}e^{t_0+s}\int_{\partial B_{r_2} }\frac{3}{2}(\partial^{2}_r u_i)^2
+(\frac{1}{2r^2})(\partial_r u_i)^2  d\sigma_g\\
& &+ \frac{1}{\sqrt{3}}e^{t_0-s}\int_{\partial B_{r_1} }\frac{3}{2}(\partial^{2}_r u_i)^2+(\frac{1}{2r^2})(\partial_r u_i)^2 d\sigma_g.
\end{eqnarray*}
By (\ref{equaldecay11}), we get
\begin{eqnarray*}
\int_{B_{r_2}\setminus B_{r_1}}\frac{1}{2r^4}(\tilde{\triangle}u)^2+\frac{8}{r^2}(\partial_r\tilde{\triangle}u)^2dV_g
\leq C\varepsilon^2(e^{t_0-\log\delta}+e^{\log(\lambda_iR)-t_0})e^{s}.
\end{eqnarray*}
So from (\ref{pohoradictangen4}), we can deduce that ($1/\sqrt{3}<2/3$)
\begin{equation}\label{energyode}
F(s)\leq \frac{2}{3}\partial_s F(s)+Ce^{2(t_0+s)}+C\varepsilon^2(e^{t_0-\log\delta}+e^{\log(\lambda_iR)-t_0})e^{s}.
\end{equation}

Recall $\log(\lambda_iR)<t_0< \log\delta$, and $0\leq s \leq \min(-t_0+\log\delta,-\log(\lambda_iR)+t_0).$ Multiplying $e ^{-\frac{3}{2}s}$ to both sides of (\ref{energyode}), if we choose $\delta$ small such that $\delta<\varepsilon^2$, then we have
\begin{equation*}
\partial_s(e^{-\frac{3}{2}s}F)\geq -C\varepsilon^2(e^{t_0-\log\delta}+e^{\log(\lambda_iR)-t_0})e^{-\frac{1}{2}s},
\end{equation*}
since $ e^{2(t_0+s)}\leq e^{(t_0+s)}=\delta e^{(t_0-\log\delta)}e^s\leq \varepsilon^2e^{(t_0-\log\delta)}e^s.$
Integrating
the above inequality from $s = 1 $ to
$s= \min(-t_0+\log\delta,-\log(\lambda_iR)+t_0)$,
we have from the fact that $F(s)$ is uniformly bounded by $C\varepsilon^2$ (see (\ref{radialenergyestimate})), that
\begin{eqnarray*}
F(1)&=&\int_{B_{e^{t_0+1}}\setminus B_{e^{t_0-1}}}\frac{3}{2}(\partial^{2}_r u_i)^2+(\frac{1}{3r^2})(\partial_r u_i)^2 dV_g\\
&\leq &C\varepsilon^2(e^{t_0-\log\delta}+e^{\log(\lambda_iR)-t_0}).
\end{eqnarray*}

Now we can apply similar arguments as in Section 5 of \cite{liu2016neck} to show that the decay of the radial part energy and the tangential part energy (see (\ref{equaldecay11})) is sufficient to guarantee that the following no neck property hold
\begin{equation*}
\lim_{\delta\rightarrow 0} \lim_{R\rightarrow \infty}\lim_{i \rightarrow \infty}osc_{B_\delta \setminus B_{\lambda_i R}}u_i=0.
\end{equation*}
The idea is to derive the decay of first derivative of $u_i$ from the energy decay by the $\varepsilon$-regularity and the Sobolev embedding theorem. The proof of Theorem \ref{noneckthm}  is finished.
\end{proof}

\

\section{Blow-up analysis for maps from degenerating Einstein manifolds}\label{treesigleale}

Let $(M_i ,g_i)$  be a sequence of $4$-dimensional Einstein manifolds as in Theorem \ref{mainconvgethm}. And let $u_i$ be a sequence of biharmonic maps from  $(M_ i ,g_ i )$ (with uniformly bounded Einstein constants) to $ N$, and satisfy
\begin{equation*}
E(u_i,M_i)\equiv\int_{M_i}|\nabla^2_{g_i}u_i|^2+|\nabla_{g_i} u_i|^4 dV_{g_i}\leq \Lambda,
\end{equation*}
for some $\Lambda > 0$. Without loss of generality, we can assume by Theorem \ref{mainconvgethm} that there exists a subsequence  $\{j\}\subset \{i\}$ such that $(M_ j ,g_ j )$ converges to $(M_ \infty ,g_ \infty )$ which is an Einstein orbifold with only one singularity $x_a$.

\begin{rem}
To study the convergence of maps from a sequence of manifolds, we recall here some necessary background knowledge about the convergence of functions defined on a sequence of metric spaces, see e.g. \cite{Mikhael1981Groups, Grove1991Manifolds}.
Suppose $X=\lim_{G-H}X_i$. Extend the metrics on $X, X_i$ to a metric on $X\coprod X_i$ so that $X_i\rightarrow X$ inside $X\coprod X_i$ in the classical Hausdorff sense. We say that $f_i: X_i\rightarrow \mathbb{R}^N$ converges to $f: X\rightarrow \mathbb{R}^N$ provided $f_i(x_i)\rightarrow f(x)$ whenever $x_i\rightarrow x$.
\end{rem}

\subsection{Blow-up away from orbifold singularities}\label{blow-thick}
In this subsection, we shall consider the simper case that the maps $u_j$ blow up away from orbifold singularities.

Away from the orbifold  singularity $x_a$, $(M_ j ,g_ j )$ converges to $(M_ \infty ,g_ \infty )$ smoothly, and
\begin{equation*}
u_j:M_j \rightarrow N
\end{equation*}
(sub)converges weakly to a map
\begin{equation*}
u_\infty:M_\infty \rightarrow N
\end{equation*}
via some diffeomorphsim
\begin{equation*}
G_j: M_\infty\setminus \{x_a\}\rightarrow M_j.
\end{equation*}
That is
\begin{equation*}
u_j\circ G_j  \rightarrow u_\infty.
\end{equation*}

Now we shall firstly examine how the limit map $u_\infty:M_\infty \rightarrow N$ behave around the singularity $x_a$. By Remark \ref{remorbfsmoth}, $\Pi^{*}_{a}g_\infty $ extends to a smooth metric on $B_1$, and on account of the fact that $u_\infty$ is a smooth map on
$B(x_a, \delta)\setminus\{x_a\} $,
we can consider the map $u_\infty(\Pi_a)$ as a smooth map on $(B_1,\Pi^{*}_{a}g_\infty )\setminus\{0\}$. It is easy to check that $u_\infty(\Pi_a)$ satisfies the biharmonic map equation with respect to the metric $\Pi^{*}_{a}g_\infty $ on $B_1\setminus\{0\}$, and the energy of $u_\infty(\Pi_a)$ over  $(B_1,\Pi^{*}_{a}g_\infty )\setminus\{0\}$  is finite. By the removable singularity Theorem \ref{remsingu}, $u_\infty(\Pi_a)$ extends to a smooth map on $B_1$.
Therefore, $u_\infty$ is a continuous map on $B(x_a,\delta)\subset M_\infty$.
So we conclude the following:
\begin{lem}\label{limmapcontinus}
The limit map $u_\infty$ is a continuous map from $M_\infty$ to $N$.
\end{lem}

If $u_j$ blow-up at non-degenerating points $x_{b,j}$, namely, $( B(x_{b,j},r),g_ j )\rightarrow (B(x_b,r),g_\infty) $ smoothly in the Cheeger-Gromov sense, then we can pull the metric of $( B(x_{b,j},r),g_ j )$ back to $ M_ \infty $ via a smooth map
$F_j: B(x_b,r)\rightarrow M_j$. Therefore locally we can consider $u_j$ as a sequence of biharmonic maps defined on $B(x_b,r) $. Without loss of generality, we suppose that away from the orbifold singularity $x_a$, $u_j$ blow-up only at $x_b$, and there is only one bubble map $\omega^b$.

In local coordinates around the blow-up point $x_b$, It is easy to see that there is a sequence of positive numbers $\lambda_j\rightarrow 0$ such that $u_j(\lambda_j x)\rightarrow \omega^b$, where
\begin{equation*}
\omega^b:(\mathbb{R}^4,g_{\text{euc}})\rightarrow N
\end{equation*}
is a biharmonic map and $g_{\text{euc}}$ is the Euclidean metric.

Since $(M_ j ,g_ j )$ converges to $(M_ \infty ,g_ \infty )$ smoothly away from $x_a$, without loss of generality,  we can assume that in local coordinates around $x_b$, the metrics $g_j $ satisfy the following: for sufficiently small $\rho>0$,
\begin{equation*}\label{almostflat}
\|(g_j)_{kl}(\rho x)-\delta_{kl}\|_{C^4}<C\rho^2
\end{equation*}
on $B_2\setminus B_1$ uniformly for large enough $j$. Then by similar arguments as in the proof of Theorem \ref{noneckthm}, we have the following bubble tree convergence away from the orbifold singularity:
\begin{thm}\label{noneckthm1}
Let $ u _i$ be a sequence of biharmonic maps from $(B(x_b,1),g_i)$  to $N$ satisfying
\begin{equation*}
E(u_i,B_1)\equiv\int_{B_1}|\nabla^2_{g_i}u_i|^2+|\nabla_{g_i} u_i|^4 dV_{g_i}\leq \Lambda,
\end{equation*}
for some $\Lambda > 0$.
If $\{u_i\} $ as above blow-up at $x_b$ (away from the orbifold singularity $x_a$ such that $d_{g_\infty}(x_a,x_b)>1$ ) with only one (nontrivial) bubble $\omega^b: (\mathbb{R}^4,g_{\text{euc}})\rightarrow N$,
 then
 \begin{equation*}
\lim_{\delta\rightarrow 0} \lim_{R\rightarrow \infty}\lim_{i \rightarrow \infty}\int_{B_\delta \setminus B_{\lambda_i R}}|\nabla^2_{g_i}u_i|^2+|\nabla_{g_i} u_i|^4 dV_{g_i}=0
\end{equation*}
and
\begin{equation*}
\lim_{\delta\rightarrow 0} \lim_{R\rightarrow \infty}\lim_{i \rightarrow \infty} osc_{B_\delta \setminus B_{\lambda_i R}}u_i=0.
\end{equation*}
\end{thm}
In particular the above theorem indicate that the image in $N$ of the limit map $u_\infty$ is connected to the image of the bubble map $\omega^b$.

\

\subsection{Blow-up at the orbifold singularity
\uppercase\expandafter{\romannumeral1}: single ALE bubble case }\label{secblowseting}

In this subsection, we shall focus on the more subtle case that the maps $u_j$ blow up at the orbifold singularity $x_a \in M_\infty$. The analysis involved in this case becomes much more complicated.

For simplicity, we shall assume that there is only one ALE bubble orbifold $M_a$ at $x_a$. Note that $M_a$ is smooth in this case. Later, we will study the more general cases in which we may encounter ALE bubble orbifolds.

First of all, we need to study the domain decomposition of the manifolds $M_j$ near the orbifold singularity $x_a$. Note that we have assumed that $x_a$ is the only singularity and there is only one bubble manifold.

By (3) of Theorem \ref{mainconvgethm}, for every $j$, there exists $x_{a,j}\in M_j$  and a positive number $r_j$ such that the followings hold:

\

\begin{itemize}
\item[(a)] $B(x_{a,j} ,\delta)$ converges to $B(x_a, \delta) $ in the Gromov-Hausdorff distance for all $\delta>0$.\\

\item[(b)] $\lim_{j\rightarrow \infty} r_j =0$.\\

\item[(c)] $(( M_j , r_{j}^{-2}g_j),x_{a,j})$ converge to $(( M_a , h_a),x_{a,\infty})$ in the Gromov-Hausdorff distance, where $( M_a , h_a)$ is a complete, noncompact, Ricci flat, non-flat $4$-manifold which is ALE of order $4$.\\

\item[(d)] There exists an into diffeomorphism $G_j : M_a \rightarrow M_j$ such that $G^{*}_{j}(r_{j}^{-2} g_j)$ converges to $h_a$ in the $C^\infty$-topology on $M_a$. According to this, for any $R>0$, we shall call $$A_{r_j R,\delta}(x_{a,j})=B(x_{a,j},\delta)\setminus B(x_{a,j},r_j R)$$
the degenerating neck region, and $B(x_{a,j},r_j R)$ the ALE bubble domain, and $M_j\setminus B(x_{a,j},\delta)$ the base domain.\\
\end{itemize}

In general, it is possible that there are several bubble maps $\omega^{k}_{a}$ ($k=1,\cdots ,m_a$) occuring at the orbifold singularity $x_a$.
Suppose that the maps
\begin{equation*}
u_j: (M_j, (\lambda_{j}^{k})^{-2}g_j, x_{a,k,j})\rightarrow N
\end{equation*}
converge to $\omega^{k}_{a}$ as $j\rightarrow \infty$.
Here
$$\lim_{j\rightarrow \infty}d_{g_j}(x_{a,k,j},x_{a,j}) =0.$$

To handle the complicated situation considered in this subsection, we shall investigate the blow-up process in the following three cases:\\

\noindent{\bf Case A}: $\lim_{j\rightarrow\infty}\frac{d_{g_j}(x_{a,k,j},x_{a,j})}{ r_j}<\infty$ and $\lim_{j\rightarrow\infty}\frac{\lambda_j^{k}}{ r_j}=0$;\\

\noindent{\bf Case B}: $\lim_{j\rightarrow\infty}\frac{d_{g_j}(x_{a,k,j},x_{a,j})}{r_j}<\infty$ and $0<\lim_{j\rightarrow\infty}\frac{\lambda_j^{k}}{ r_j}<\infty$;\\

\noindent{\bf Case C}: others.\\

Roughly speaking, the domain of the bubble map of Case B is (almost) the ALE bubble domain, bubbles of case A are on the bubble of Case B, the bubble of Case B is on some bubble of Case C or is away from bubbles of Case C. We remark here that we can assume $x_{a,k,j}=x_{a,j}$ in Case B without loss of generality.

In the rest of the paper, bubbles of case A, Case B, Case C will be indicated by $\omega^A$, $\omega^B$, $\omega^C$ respectively. We shall denote the centers of those three types of bubbles by $x^{A}_{a,j}$, $x^{B}_{a,j}$, $x^{C}_{a,j}$ respectively, and denote those three type of bubble scales by $\lambda_{j}^{A}$, $\lambda_{j}^{B}$, $\lambda_{j}^{C}$ respectively.

In case A, it is easy to see that
\begin{equation*}
[u_j: (M_j, (\lambda_{j}^{A})^{-2}g_j, x^{A}_{a,j})\rightarrow N ] \longrightarrow [\omega^{A}:(\mathbb{R}^4, g_{\text{euc}}, 0)\rightarrow N].
\end{equation*}

In Case B, by an adjustment of the bubble scale, we may assume $r_j=\lambda_{j}^{B}$. By our assumption in the above, we have
\begin{equation*}
[u_j: (M_j, (\lambda_{j}^{B})^{-2}g_j, x^{B}_{a,j})\rightarrow N ] \longrightarrow [\omega^{B}:( M_a , h_a,x_{a,\infty})\rightarrow N]
\end{equation*}
on any compact subset of the ALE bubble manifold $M_a$.

\begin{rem}\label{bublleAandD}
By our construction, bubble $\omega^A$ lies on the bubble $\omega^B$, and by applying the
same arguments for the proof of Theorem \ref{noneckthm1}, we know that there is no energy loss in the neck region
$$A_{\lambda_j^{A} R,\delta r_j }(x^{A}_{a,j})\equiv B(x^{A}_{a,j},{\delta r_j}) \setminus B(x^{A}_{a,j},{\lambda_j^{A} R}),$$ and the image of $\omega^ A $ is connected to the image of
 $\omega^ B$.
\end{rem}

It is natural to ask the following:

\begin{que}\label{ques1}
How the maps $u_j$ behaves on the degenerate neck regions
$$A_{r_jR,\delta}(x_{a,j})=B(x_{a,j},\delta)\setminus B(x_{a,j},r_jR)$$
when $j\rightarrow\infty$?
\end{que}

This question is related to the analysis of bubble maps of Case C.
It has certain similarity with the blow-up analysis of harmonic maps from degenerating Riemann surfaces \cite{zhu2010harmonic}. Set $T_j=\log(r_jR)$ and $T(\delta)=\log(\delta)$, then we have the following bubble-neck decomposition:

\begin{prop}\label{properneckc2} Notations and assumptions as above.
\begin{itemize}

\item [(1)] “Asymptotic boundary conditions”:
\begin{equation*}
\lim_{j\rightarrow \infty}\omega(u_j,P_{T_j,T_j+L})=\lim_{\delta\rightarrow 0}\omega(u_j,P_{T(\delta)-L,T(\delta)})=0, \quad \forall L\geq 1,
\end{equation*}
\begin{equation*}
\lim_{j\rightarrow \infty}osc_{P_{T_j,T_j+L}} u_j=\lim_{\delta\rightarrow 0}osc_{P_{T(\delta)-L,T(\delta)}}u_j=0,
\end{equation*}
where $P_{T_j^{1},T_j^{2}}$ is the cylinder corresponding to $A_{e^{T_j^{1}},e^{T_j^{2}}}(x_{a,j})\subset M_j$,
\begin{equation*}
\omega(u_j,P_{T_j^{1},T_j^{2}})\equiv \sup_{t\in[T_j^{1},T_j^{2}-1]}\int_{A_{e^{t},e^{t+1}}(x_{a,j})}|\nabla^2_{g_j} u_j|^2+|\nabla_{g_j} u_j|^4 dV_{g_j}.
\end{equation*}

\item [(2)] “bubble domain and neck domain”: after selection of a subsequence, which we still denote by $u_j$, the following two alternatives hold:

\begin{itemize}
\item [(2.1)]
\begin{equation*}
\lim_{\delta\rightarrow 0}\lim_{R\rightarrow \infty}\lim_{j\rightarrow \infty}\omega(u_j,P_{T_j,T(\delta)})=0,
\end{equation*}

\item [(2.2)] there exists some number $N_2>0$ which is independent of $j$, and $2N_2$ sequences of numbers $\{a_{j}^1\},\{b_{j}^1\}, \cdots, \{a_{j}^{N_2}\},\{b_{j}^{N_2}\}$,  such that
\begin{equation*}
T_j\leq a_j^{1}\ll b_j^{1}\leq \cdots \leq a_j^{N_2}\ll b_j^{N_2}\leq T(\delta) \quad (a_j^{\alpha}\ll b_j^{\alpha} \,\text{means}\,\lim_{j\rightarrow \infty}b_j^{\alpha}- a_j^{\alpha}=\infty )
\end{equation*}
and
\begin{equation*}
 |b_j^{\alpha}- a_j^{\alpha}|\ll |T_j|,\quad   \text{i.e.} \ \lim_{j\rightarrow \infty} \frac{|b_j^{\alpha}- a_j^{\alpha}|}{|T_j|}=0.
\end{equation*}
\end{itemize}

\end{itemize}

Denote
 $$J_j^{\alpha}\equiv P_{a_j^{\alpha}, b_j^{\alpha}}, \quad \alpha=1,\cdots, N_2,$$
 $$I_{j}^0=P_{T_j, a_{j}^1},I_j^{N_2}=P_{b_j^{N_2},T(\delta)}, I_j^{\alpha}=P_{b_j^{\alpha},a_j^{\alpha+1}}, \quad \alpha=1,\cdots, N_2-1. $$
 Then

\begin{itemize}
  \item [({\romannumeral1})] $ \alpha=0,1,\cdots, N_2,$ $\lim_{\delta\rightarrow 0}\lim_{R\rightarrow \infty}\lim_{j\rightarrow \infty}\omega(u_j,I_j^{\alpha})=0$.
  \item [({\romannumeral2})] $\forall \alpha=1,\cdots, N_2,$ there is a bubble tree which consists of at most finitely many bubble maps, i.e., finite energy biharmonic maps from $\mathbb{R}^4/\Gamma$ to $N$. Here, for simplicity of notations, we assume there is only one such bubble map, namely, there is a biharmonic map
      $$\omega^{C,\alpha}: \mathbb{R}^4/\Gamma\longrightarrow N,$$
      such that
 \begin{equation*}
 \lim_{\delta\rightarrow 0}\lim_{R\rightarrow \infty}\lim_{j\rightarrow \infty}E_{bi}(u_j,\bar{J}_j^{\alpha})=E_{bi}(\omega^{C,\alpha}),
\end{equation*}
 where $\bar{J}_j^{\alpha}$ is the region in $M_j$ corresponding to $J_j^{\alpha}$, and
  \begin{equation*}
E_{bi}(u_j,\bar{J}_j^{\alpha})\equiv\int_{\bar{J}_j^{\alpha}}|\nabla^2_{g_j}u_j|^2+|\nabla_{g_j} u_j|^4 dV_{g_j}.
\end{equation*}
  \item [({\romannumeral3})] Each bubble $\omega^{C,\alpha}$ given above is a continuous map, and
\begin{equation*}
\lim_{R\rightarrow \infty}osc_{ (\mathbb{R}^4/\Gamma)\setminus B_R}\ \omega^{C,\alpha}=0.
\end{equation*}
\end{itemize}

\end{prop}

\begin{proof}
(1) follows easily from our assumption. (2) can be shown in the same sprit as in \cite{zhu2010harmonic}.

For ({\romannumeral1}), it can be argued as in the proof for Proposition 3.1 of \cite{zhu2010harmonic}.

For ({\romannumeral2}), we only need to verify that $\omega^{C,\alpha}$ is map from $\mathbb{R}^4/\Gamma$ to $ N.$
Let $\bar{J}_j^{\alpha}$ be the annulus $B(x_{a,j},e^{b_j^{\alpha}})\setminus B(x_{a,j},e^{a_j^{\alpha}})\subset M_j $ corresponding to ${J}_j^{\alpha}$.
Proposition \ref{neckprop} says that the neck region between $\omega^{B}$ and $u_\infty$ in $M_j$ looks like a potion of a flat cone $\mathbb{R}^4/\Gamma$ for large $j$, moreover,
\begin{equation*}\label{neckcurvturestm}
r^2|R_{g_j}|\leq C \max \left\{(\frac{r_j}{r})^{\varepsilon_5},(\frac{r}{r_\infty})^{\varepsilon_5}\right\}
\end{equation*}
on $\bar{J}_j^{\alpha}$, where $\varepsilon_5$ and $r_\infty$ are small positive constants.
So by the relation $a_j^{\alpha}\ll b_j^{\alpha}$, after scaling of the scale $\lambda_j^{C,\alpha}\equiv e^{\frac{a_j^{\alpha}-b_j^{\alpha}}{2}}$, $\left(\bar{J}_j^{\alpha}, (\lambda_j^{C,\alpha})^{-2}g_j \right)$ converges to $\mathbb{R}^4/\Gamma$ as $j\rightarrow \infty$, and $u_j: \left(\bar{J}_j^{\alpha},(\lambda_j^{C,\alpha})^{-2}g_j \right)\rightarrow N$ converges to a bubble $\omega^{C,\alpha}$.

For ({\romannumeral3}), the first part can be argued as in Lemma \ref{limmapcontinus} and the second part follows from the result of Lemma \ref{remsinguinfty} which is stated and proved separately afterwards.
\end{proof}

\begin{lem}\label{remsinguinfty}
The bubble maps $\omega^B:M_a\rightarrow N$ and $\omega^{C,\alpha}:\mathbb{R}^4/\Gamma\rightarrow N$ above is uniformly continuous at the infinity in the sense that, for any $\delta > 0$, there
is $R > 0$ such that
$$osc_{M \setminus B_ R}  \  \omega <\delta,$$
where $M=M_a$ and $\omega=\omega^B$,  or $M=\mathbb{R}^4/\Gamma$ and $\omega=\omega^{C,\alpha}$.
\end{lem}

Since biharmonic maps are in general not conformally invariant in dimension 4, we
can not apply the removable singularity theorem to see that $\lim_{|x|\rightarrow \infty} \omega(x)$ is well defined, see e.g. \cite{wang2004remarks, liu2015finite}.

\begin{proof}
We shall only prove the lemma for $\omega^B$.
Since the problem is scaling invariant in dimension 4, we may assume that for some fixed small $\varepsilon_1>0$ to be determined by the arguments later,
\begin{equation*}
\int_{ M_a\setminus B_ 1}|\nabla^2_{h_a}\omega^B|^2+|\nabla_{h_a}\omega^B|^4 dV_{h_a}<\varepsilon_1.
\end{equation*}
Then the lemma is an immediate consequence of the following Lemma \ref{lemrovsiguifty}.
\end{proof}

\begin{lem}\label{lemrovsiguifty}
There is a constant $\varepsilon_1>0$ depending on $N$ such that, if $u : M_a \rightarrow N$ is a
biharmonic map from a Ricci flat ALE manifold (orbifold) of order 4 to $N$ and satisfying
\begin{equation*}
\int_{M_a \setminus B_ 1}|\nabla^2_{h_a}u|^2+|\nabla_{h_a}u|^4 dV_{h_a}<\varepsilon_1.
\end{equation*}
then $u$ is uniformly continuous at the infinity in the sense that for any $\delta> 0$, there
is $R > 0$ independent of u such that
$$osc_{M_a \setminus B_ R} \  u< \delta.$$
\end{lem}

The proof of Lemma \ref{lemrovsiguifty} will be given in Appendix \ref{prflem1}.

\

\subsection{Blow-up at the singularity \uppercase\expandafter{\romannumeral2}: multiple ALE bubbles case }\label{treemultiale}
In this subsection, we shall discuss how to construct the whole bubble tree for the convergence of biharmonic maps in the case that there are several ALE bubble manifolds (orbifolds) emerging (at the same orbifold singularity) from the G-H convergence of Einstein manifolds.

From the analysis in the previous subsections, we know that three types of bubble maps will appear at the orbifold singularity, namely,
$\omega^A: \mathbb{R}^4\rightarrow N,$ $\omega^B: M_{ALE}\rightarrow N,$ and $\omega^C: \mathbb{R}^4/\Gamma\rightarrow N,$ where $ M_{ALE}$ represent the ALE bubble manifolds (orbifolds) in the tree $\mathbf{Tr_{ALE}}$ (see Remark \ref{aletree}).

Bubble maps of type $\omega^B $ occur in the place where the Riemannian curvatures concentrate, and the number of bubble maps of kind $\omega^B$ is equal to the number of ALE bubble manifolds (orbifolds) in the tree $\mathbf{Tr_{ALE}}$ (if one of them is defined on a ALE orbifold, then we can argue as in Lemma \ref{limmapcontinus} to show that it is continuous at the orbifold singularities).  Bubble maps of type $\omega^C$ appear in degenerating neck regions such as those discussed in Proposition \ref{neckprop}. Bubble maps of type $\omega^A$ appear over the regions  where there is no degeneration of the metrics at its blow-up scale (they are on bubble maps of type $\omega^B$). With these in mind, one can easily construct the whole bubble map tree by induction. Note that the whole process will be terminated in finite steps, since the number of nontrivial ALE bubble manifolds (orbifolds) is finite and the number of nontrivial bubble maps of type $\omega^A$ and $\omega^C$ must be finite by the energy gap Theorem \ref{energygap11}) and the finite total energy assumption.

\begin{que}\label{akeyapplirem}
Is there an energy gap phenomenon for biharmonic maps from the ALE spaces in the tree $\mathbf{Tr_{ALE}}$?
\end{que}

Now we state the main theorem in this paper to end this section.
\begin{thm}\label{degeymenergy} Let $(M_j,g_j)$ be a sequence of Einstein manifolds as in Theorem \ref{mainconvgethm} and let $ u_j$ be a sequence of biharmonic maps from $(M_j,g_j)$ to $N$, satisfying
 \begin{equation*}
\int_{M_j}|\nabla^2_{g_j}u_j|^2+|\nabla_{g_j} u_j|^4 dV_{g_j}\leq \Lambda
\end{equation*}
for some $\Lambda>0$.
Without loss of generality, we assume that there is only one ALE bubble manifold and there is at most one bubble in each case in the above blow up analysis, then (up to a subsequence) we have
\begin{eqnarray*}
\int_{M_j}|\nabla^2_{g_j}u_j|^2+|\nabla_{g_j} u_j|^4 dV_{g_j}=\int_{M_\infty}|\nabla^2_{g_\infty}u_\infty|^2+|\nabla_{g_\infty} u_\infty|^4 dV_{g_\infty}+E_{\text{bi}}(\omega^b)+\sum_{a=A,B,C}E_{\text{bi}}(\omega^a),
\end{eqnarray*}
where $E_{\text{bi}}(\cdot)$ is the biharmonic energy of the corresponding bubble map. Moreover the images of all the bubbles and the image of $ u_\infty$ are all connected.
\end{thm}
By the analysis discussed before, to prove the above main theorem, we need only to show that  the biharmonic energy of the maps $u_j$ over the neck regions $I_j^{\alpha}$ is converging to 0 and the image of $u_j$ over $I_j^{\alpha}$ is converging to a point.
More precisely, we will prove the following theorem for the extrinsic case in Section \ref{degeneneckanalys} (the intrinsic case will handled in Section \ref{secintrinsic}).
\begin{thm}\label{degenockenergy}
Let $\bar{I}_j^{\alpha}= A_{e^{b_j^\alpha},e^{a_j^{\alpha+1}}}(x_{a,j}) $ be the annulus region in $M_j$ corresponding to $I_j^{\alpha}$ in Proposition \ref{properneckc2}, then
\begin{equation*}
\lim_{\delta\rightarrow 0} \lim_{R\rightarrow \infty}\lim_{j \rightarrow \infty}\int_{\bar{I}_j^{\alpha}}|\nabla^2_{{g_j}}u_j|^2+|\nabla_{g_j} u_j|^4 dV_{g_j}=0
\end{equation*}
and
\begin{equation*}
\lim_{\delta\rightarrow 0} \lim_{R\rightarrow \infty}\lim_{j \rightarrow \infty}osc_{\bar{I}_j^{\alpha}} \  u_j=0.
\end{equation*}
\end{thm}

The proof of the above theorem is the most difficult part in this paper.

\

\section{Geometry of the degenerating neck regions}\label{constrcoordinates}

To prove Theorem \ref{degenockenergy}, we need to investigate the asymptotic bebavior of biharmonic maps $u_j$ over the neck regions $I_j^{\alpha}$, which are sub-annuli of the degenerating neck regions $A_{r_jR,\delta_1}(x_{a,j})\subset M_j$. Now the main difficult issue is that we do not have satisfactory coordinates on those degenerating neck regions $A_{r_jR,\delta_1}(x_{a,j})$ at this stage. Recall that $A_{\delta_2,\delta_1}(\cdot)$ represents $B(\cdot,\delta_1)\setminus B(\cdot,\delta_2)$. In this section, we shall explore the refined geometric picture of the degenerating neck region $A_{r_jR,\delta_1}(x_{a,j})$ and construct two types of \emph{good global coordinates} on this region: $(x)$ and $(y)$. In the coordinates $(x)$, the metrics $g_j$ are $C^0$ close to the flat metric. While, in the coordinates $(y)$, the metrics $g_j$ are $C^4$ close to the flat metric in some weighted function space.

\begin{rem}
The arguments in this section are valid in higher dimensions in the domain and hence the scheme developed in this paper can be applied to geometric systems defined over non-collapsed degenerating Einstein manifolds of higher dimensions.
\end{rem}

 Firstly, we recall the definition of weighted H\"{o}lder space. The weighted H\"{o}lder space $C^{k,\alpha}_{\eta}(A_{r_1,r_2})$ ($0<\alpha<1$) with weight $\eta$ (a function of $|x|=r$) are the space of functions in $C^{k,\alpha}(A_{r_1,r_2}(0))$ ($0\in\mathbb{R}^n$) with bounded weighted norm $||\cdot||_{C^{k,\alpha}_{\eta}}$, which is defined as follows:
\begin{equation*}
||f||_{C^{k,\alpha}_{\eta}}=\sum_{j=0}^{k}\sup_{A_{r_1,r_2}(0)}\eta^{-1}|x|^{j}|D^{j}f|+\sup_{x\neq y}
\min \ (\eta^{-1}|x|^{k+\alpha},\eta^{-1}|y|^{k+\alpha})\frac{|D^{k}f(x)-D^{k}f(y)|}{|x-y|^{\alpha}}.
\end{equation*}
 Analogous weighted norms for spaces of functions on $\mathbb{R}^n\setminus B_R$ were defined in \cite{Mey} (page 252-253). By similar arguments as in the proof of Theorem 1 in \cite{Mey}, we have the following Schauder's estimates for linear elliptic equations:

\begin{thm}\label{schauderestneck}
Set \begin{equation*}
L(u)=\sum_{k,l}a_{kl}(x)\partial_k\partial_l u+\sum_{k}a_{k}(x)\partial_k u+a(x)u,
\end{equation*}
where $ a_{k,l}\in C^{p,\alpha}_{r^{0}}(A_{r_1,r_2})$, $ a_{k}\in C^{p,\alpha}_{r^{-1}}(A_{r_1,r_2})$, $ a\in C^{p,\alpha}_{r^{-2}}(A_{r_1,r_2})$ with norms less than $K_1$.
Furthermore, we assume
\begin{equation*}
\sum_{k,l}a_{kl}(x)\xi_k\xi_l\geq K_2\sum_{k}\xi_k^{2} \quad (K_2>0).
\end{equation*}
Let $u$ be a solution of $L(u)=f$, then we have
\begin{equation*}
||u ||_{C^{p+2,\alpha}_{\eta}}\leq \hat{C}\hat{K}(||u ||_{C^{0}_{\eta}}+||f ||_{C^{p,\alpha}_{r^{-2}\eta}}),
\end{equation*}
where $\hat{C}>0$ is defined by $\eta(s)\leq \hat{C} \eta(r) $ for $1/2r<s<3/2r$, and $\hat{K}>0$ depends only on $K_1$, $K_2$, $p$, $i$, $\alpha$, and $n$.
\end{thm}
We put the proof in Appendix \ref{schauderest}.

The main result in this section is the following:

\begin{thm}\label{neckmeytricflat}
For  $\delta_1>0$ small enough and for $R>0$ and $j>0$ sufficiently large, there exist two coordinates $(x)$ and $(y)$ on $A_{r_jR,\delta_1}(x_{a,j})\subset M_j$ such that the following properties hold:
in coordinates $(x)$, there hold
\begin{eqnarray*}\label{almostflat1dege1}
& &\|g_{j,kl}( x)-\delta_{kl}\|_{C^{0}}<\eta_j(|x|),\\
& &|x(\cdot)|\equiv r(\cdot)=d_{g_j}(\cdot,x_{a,j})\quad \text{and} \quad \nabla_{\partial_r}\partial_r=0,
\end{eqnarray*}
and in coordinates $(y)$, there hold
\begin{equation*}\label{almostflat1dege1}
\|g_{j,kl}( y)-\delta_{kl}\|_{C^{4,\alpha}_{\eta_j(|y|)}}<C,
\end{equation*}
where $\eta_j(t)=O(r^{-\varepsilon_5}_\infty t^{\varepsilon_5}+ (r_jR_0)^{\varepsilon_5}t^{-\varepsilon_5}),$ $r_\infty>0$, $\varepsilon_5 >0$, $r_j>0$ are the same as in the curvature estimate in Proposition \ref{neckprop}, and $R_0>0$ is some fixed large number.
Moreover,
\begin{equation*}
\|y-x\|_{C^{1}_{|y|\eta_j(|y|)}}<C,
\end{equation*}
in particular, $ |x-y|<\eta_j(|x|)|x|$.
\end{thm}

\begin{rem}
Here we say coordinates on the degenerating neck region, we mean a map
$$\varphi=\mathcal{L}^{-1}\circ proj:\mathbb{R}^4 \rightarrow  A_{r_jR,\delta_1}(x_{a,j}),$$
where $\mathcal{L}$ is a diffeomorphism from $ A_{r_jR,\delta_1}(x_{a,j})$ to $\mathbb{R}^4/\Gamma$, and $proj$ is the natural projection from $\mathbb{R}^4$ to $\mathbb{R}^4/\Gamma$.
\end{rem}

\begin{rem}
The above theorem says that, in the $(y)$ coordinates, the metric on the degenerating neck region $A_{r_jR,\delta_1}(x_{a,j})\subset M_j$ is close in $C^4$ topology (in fact $C^k$ for any $k \in \mathbb{N}$ by the argument given in the proof) to the flat metric in some weighted sense. It is easy to see that $\eta_j$ satisfies the requirement in Theorem \ref{schauderestneck}. We remark that the notation $\eta_j$ often appears in the rest of the paper, it means a function of order $\eta_j$ or it serves as a positive upper bound.

\end{rem}

To prove Theorem \ref{neckmeytricflat}, we shall apply some arguments in \cite{BKN} with some modifications. The proof will be completed in two steps. Firstly, we outline the construction of $C^{1}$-coordinates on the degenerating neck regions.

\begin{lem}\label{neckmeytricflatlem}
One can construct on $A_{r_jR,\delta_1}(x_{a,j})\subset M_j$ coordinates $(x)$ for small enough $\delta_1>0$ and sufficiently large $R>0$ and $j>0$, such that
\begin{equation*}\label{almostflat1dege1}
\|g_{j,kl}( x)-\delta_{kl}\|_{C^{0}}<\eta_j(|x|).
\end{equation*}
\end{lem}

\begin{proof}
    Recall that $ (M_j, r_j^{-2}g_j, x_{a,j})$ converges in pointed G-H topology (in $C^\infty$ topology outside some compact set) to $(M_a,h_a,p_a)$ (see Proposition \ref{covergtoorbifold} or \cite{anderson1989ricci}). The fact that $(M_a,h_a)$ is a Ricci flat ALE manifold (orbifold) gives the existence of the following diffeomorphism (see Section 2 of \cite{kasue1988compactification} or Lemma (3.1) of \cite{BKN}):
    $$\Pi_{R}:A_{\frac{R}{2}, 2R}(p_a)\subset M_a\rightarrow \mathcal{C}(S^3/\Gamma).$$
   Moreover, the limit metric
   \begin{equation*}
   \lim_{R\rightarrow\infty}R^{-2}\Pi_{R*}h_a \end{equation*}
   is the standard flat metric on the flat cone $\mathcal{C}(S^3/\Gamma)$. By taking the universal covering of $A_{R, 2R}(p_a)\subset (M_a,h_a)$, we assume $\Gamma={e}$. Let us temporarily ignore the fundamental group.

   For large $R>0$, we set $\tilde{S}_{R}:=\Pi_{R}^{-1}(\{1\}\times S^3)$. Hence
 \begin{eqnarray}\label{shape11}
 ||R\tilde{A}_{R}+Id||\leq a(R) ,
 \end{eqnarray}
where $\tilde{A}_{R}$ stands for the shape operator of $\tilde{S}_{R}$, and $a(R)\rightarrow 0$, as ${R\rightarrow\infty}$.

Notice that $r_j^{-1}A_{\frac{r_jR}{2}, 2r_jR}(x_{a,j})$ converge smoothly to $A_{\frac{R}{2}, 2R}(p_a)\subset M_a$, it is easy to see that $\partial B(x_{a,j},r_jR)$ is smooth, and $(R)^{-1}\tilde{S}_{R} $ and $(r_jR)^{-1}\partial B(x_{a,j},r_jR)$ are closed in $C^\infty$ topology.
Hence $\partial B(x_{a,j},r_jR)\equiv S_{r_jR}$ is diffeomorphic to $S^3$,
and
\begin{equation*}
\lim_{R\rightarrow\infty}\lim_{j\rightarrow\infty}(r_jR)^{-1}S_{r_jR}=(S^3,g_{S^3})
\end{equation*}
in $C^k$ topology, where $g_{S^3}$ is the standard round sphere metric. And
  \begin{equation}\label{shape111}
||( r_jR)A_{r_jR}+Id||< a(j,R),  \quad  a(j,R)\rightarrow0\text{ as } R,j\rightarrow\infty
 \end{equation}
 by (\ref{shape11}) and the arguments at the beginning of the proof.
Denote the diffeomorphism by
\begin{equation}\label{keydifff11}
\Upsilon_{r_jR}:\tilde{\theta}\in S_{r_jR}\rightarrow \theta\in S^3.
 \end{equation}

Motivated by the construction of coordinates around isolated orbifold singularity in Page 342 of \cite{BKN}, we consider a map
\begin{eqnarray*}
\Phi_{r_jR}:  S^3\times [r_jR, \delta_1]   &\longrightarrow&  M_j \\
(\theta, t) &\longmapsto& \exp\{(t-r_jR)\nu_{r_jR}(\tilde{\theta})\}=\exp\{(t-r_jR)\nu_{r_jR}(\Upsilon_{r_jR}^{-1}(\theta))\}
\end{eqnarray*}
where $ \nu_{r_jR}(\tilde{\theta})$ is outward unit normal vector.

\begin{prop}\label{intodiffeo}
$\Phi_{r_jR}$ is an into-diffeomorphism for some fixed constant $\delta_1
>0$ independent of $j$ and $R$.
\end{prop}
The above proposition will be proved soon afterward, now we turn back to the construction of coordinates.

Taking a point $\tilde{\theta}\in S_{r_jR}$, we consider parallel orthonormal vector fields $E_1(t),\cdots,$ $E_4(t)$ along $\sigma_{r_jR}(t)=\Phi_{r_jR} (\theta, t)$. Let $ X_{r_jR}(t)$ be a Jacobi field (see Section 2 of \cite{BKN}) along $\sigma_{r_jR}(t)$ with
\begin{eqnarray}\label{intialcondt11}
& &  X_{r_jR}(r_jR)  =  r_jRE_k(r_jR),  \nonumber\\
& &  \nabla_{\partial t} X_{r_jR}^{\bot}(r_jR)   =   -r_jRA_{r_jR}(E_k^{\bot}(r_jR)), \\
& &  \nabla_{\partial t}(X_{r_jR}- X_{r_jR}^{\bot})(r_jR)   =  E_k(r_jR)-E_k^{\bot}(r_jR)  \nonumber
\end{eqnarray}
for some $k=1,\cdots,4$. Here $A_{r_jR}$ stands for the shape operator of $S_{r_jR}$, $\nu_{r_jR} \perp  X_{r_jR}^{\bot}$. Then $X^{i}_{r_jR}(t)=(X_{r_jR}(t),E_i(t)) $ satisfies (see Section 5 of \cite{smith1992removing} or Section 3 of \cite{BKN})
\begin{equation*}
X^{i}_{r_jR}(t)=X^{i}_{r_jR}(r_jR)+(\partial_tX^{i}_{r_jR})(r_jR)(t-r_jR)
+\int_{r_jR}^{t}\int_{r_jR}^{s}\sum X^{l}_{r_jR}(u)K_l^{i}(u)duds,
\end{equation*}
where $K_l^{i}(u)=(Riem(E_l(u),(\partial_t\sigma_{r_jR})(u))(\partial_t\sigma_{r_jR})(u),E_i(u)).$

Thanks to the curvature estimate
\begin{equation}\label{neckcurvturestmprop}
r^2|R_{g_j}|\leq C_7 \max \left\{(\frac{r_j}{r})^{\varepsilon_5},(\frac{r}{r_\infty})^{\varepsilon_5}\right\}.
\end{equation}
 in Proposition \ref{neckprop},
  we find by direct calculations using (\ref{shape111}) and (\ref{intialcondt11}) that
  \begin{eqnarray*}
  & &| X^{i}_{r_jR}(t)|<Ct,\\
  & &(\partial_tX^{i}_{r_jR})(r_jR)t=\delta_k^{i}(1+a(j,R))t,\\
  & &\big{|}\int_{r_jR}^{t}\int_{r_jR}^{s}\sum X^{l}_{r_jR}(u)K_l^{i}(u)duds\big{|}\\
  &<& C\max\left\{\left[\varepsilon_5r_jR^{1-\varepsilon_5}+tR^{-\varepsilon_5}-t(\frac{r_j}{t})^{\varepsilon_5}\right],
  \left[\varepsilon_5^{-1} \cdot\frac{t^{1+\varepsilon_5}}{r_\infty^{\varepsilon_5}}\right]\right\}.
  \end{eqnarray*}
  for $t\in[r_jR, \delta_1]$. Noticing that
  \begin{equation*}
 a(j,R)\rightarrow0, \text{ as } R,j\rightarrow\infty
 \end{equation*}
 by (\ref{shape111}), we know that the terms $a(j,R)t$ and $tR^{-\varepsilon_5}$ have to cancel each other out. Indeed
 \begin{eqnarray*}
 (\partial_tX^{i}_{r_jR})(Nr_jR)&=&(\partial_tX^{i}_{r_jR})(r_jR)+\int_{r_jR}^{Nr_jR}\sum X^{l}_{r_jR}(u)K_l^{i}(u)du\\
 &=&(\partial_tX^{i}_{r_jR})(r_jR)+O(R^{-\varepsilon_5}-(NR)^{-\varepsilon_5}),
 \end{eqnarray*}
 where $N>0$ is some large number.
  Therefore we obtain the following key estimate
\begin{equation}\label{jacobiest1}
|t^{-1}X^{i}_{r_jR}(t)-\delta_{ik}|\leq C(r_jRt^{-1})+ C(r_jR^{1-\varepsilon_5}t^{-1})+C\max(r^{-\varepsilon_5}_\infty t^{\varepsilon_5}, (r^{\varepsilon_5}_j)t^{-\varepsilon_5}).
\end{equation}

Now fix $R=R_0$ in (\ref{jacobiest1}) for some sufficiently large $R_0$.
Then for $t\in[r_j R, \delta_1]$, we have
\begin{equation*}
|t^{-1}X^{i}_{r_jR_0}(t)-\delta_{ik}|\leq \eta_j(t),
\end{equation*}
where $\eta_j(t)=O(r^{-\varepsilon_5}_\infty t^{\varepsilon_5}+ (r_jR_0)^{\varepsilon_5}t^{-\varepsilon_5}).$

 The above estimate means that in the coordinates
\begin{equation}\label{coordinatesonneck11}
\Phi_{r_jR_0}:S^3\times [r_jR,\delta_1]=\left\{x\in\mathbb{R}^4:r_jR\leq|x|\leq\delta_1 \right\}\rightarrow A_{r_jR,\delta_1}(x_{a,j})\subset M_j,
\end{equation}
the metric $g_j$ satisfies
\begin{equation*}
|(\Phi_{r_jR_0}^{*}g_j)_{kl}-\delta_{kl}|\leq \eta_j(|x|).
\end{equation*}
We just outline it, for details one can see Step 2 in the proof of Theorem (1.1) in \cite{BKN}. In fact
from (\ref{jacobiest1}) we have that the shape operator $\bar{A}_t,t\in[r_jR,\delta_1]$ of $ S_t=\partial B(x_{a,j},t)$ satisfies
\begin{equation}\label{shapecompa}
|t\bar{A}_t+Id|\leq \eta_j(t).
\end{equation}
The standard Euclidean metric on $S^3\times [r_jR,\delta_1]$ can be written as
\begin{equation*}
dr^2+r^2d\theta^2.
\end{equation*}
So from (\ref{shapecompa}) the metric $g_j$ in $(x)$ coordinates takes the form
\begin{equation*}
dr^2+r^2(1+\eta_j(r))^2d\theta^2.
\end{equation*}
Hence
\begin{equation*}
|g_{j,kl}(x)-\delta_{kl}|\leq \eta_j(|x|).
\end{equation*}

 \end{proof}

\begin{rem}\label{coordequal2}
From the above construction, we know that
 $|x(\cdot)|=r(\cdot)=d_{g_j}(\cdot,x_{a,j})$ and $ \nabla_{\partial_r}\partial_r=0$.
\end{rem}

The final step for the proof of the Theorem \ref{neckmeytricflat} is to improve the regularity of the metric on the basis of Lemma \ref{neckmeytricflatlem}.

\begin{lem}\label{neckmeytricflatlem2} We can construct coordinates $(y)$ on $A_{r_jR,\delta_1}(x_{a,j})$ such that
\begin{equation*}
\|g_{j,kl}( y)-\delta_{kl}\|_{C^{4,\alpha}_{\eta_j}}<C.
\end{equation*}
\end{lem}

 This estimate says that if we introduce the coordinates $\tilde{y}=\rho^{-1} y$ on $A_{\rho,2\rho}(x_{a,j})\subset M_j$, then
\begin{equation*}\label{almostflat1dege1}
\|g_{j,kl}(\rho\tilde{y})-\delta_{kl}\|_{C^{4}(B_2\setminus B_1)}<\eta_j(\rho)
\end{equation*}
for all $\rho\in[r_j R, \delta_1]$.

\begin{proof}
As in the proof of Lemma \ref{neckmeytricflatlem}, at very
point $\tilde{\theta}\in S_{r_jR_0}$, we consider parallel orthonormal vector fields $E_1(t,\tilde{\theta}),\cdots,$ $E_4(t,\tilde{\theta})$ along $\sigma_{r_jR_0}(t)=\Phi_{r_jR_0} (\theta, t)$ which satisfy $$r_jR_0\nu_{r_jR_0}(\tilde{\theta})=x^i(r_jR_0,\tilde{\theta})E_i(r_jR_0,\tilde{\theta}).$$ Here $(x)=(r,\theta)$ are the coordinates constructed in Lemma \ref{neckmeytricflatlem}.
Then
\begin{equation*}\label{keyfpohozaev12}
r\partial_r=x^iE_i(r)
\end{equation*}
 along $\sigma_{r_jR_0}$, since both $E_i(t)$ and $\partial_r$ are parallel invariant and $|x|=r$.

Now we take geodesics $\sigma_1,\cdots, \sigma_K$ ($\sigma_k(t)=\Phi_{r_jR_0} (\theta_k, t)$, $\Upsilon_{r_jR_0}(\theta_k)=\tilde{\theta}_k$ )
such that
$$\{B(\sigma_k(t_l);\delta t_l)\}_{k=1,\cdots, K, l=1,\cdots}$$
 cover $A_{r_jR,\delta_1}(x_{a,j})$. We denote $\{B(\sigma_k(t_l);\delta t_l)\}$ by $B_{k,l}$.

 For all $k$ and $l$, we take harmonic coordinates associated to $E_1(t_l,\tilde{\theta}_k),\cdots,$ $E_4(t_l,\tilde{\theta}_k)$
 \begin{equation*}
 \mathbf{H}_{k,l}:B_{k,l}\rightarrow \mathbb{R}^4
 \end{equation*}
 as in \cite{BKN}.
 Moreover we set $ \mathbf{H}_{k,l}(\sigma_k(t_l))=(\theta_k, t_l)$.
 By means of the arguments in Step 3 in the proof of Theorem (1.1) in \cite{BKN} with some minor modifications, we have
 \begin{equation}\label{intialregul}
 ||(\mathbf{H}_{k,l*}g_j)_{mn}-\delta_{mn}||_{C^{1,\alpha}_{\eta_j(|z|)}(\mathbf{H}_{k,l}(B_{k,l}))}\leq C,
 \end{equation}
 where $|z|$ is the absolute value of $\mathbf{H}_{k,l}$ coordinates functions. Then by elliptic estimates (see Schauder's estimates for elliptic equations in Theorem \ref{schauderestneck}), we know that in the above coordinates, the metric satisfies
  \begin{equation}\label{weightsmooth}
 ||(\mathbf{H}_{k,l*}g_j)_{mn}-\delta_{mn}||_{C^{p,\alpha}_{\eta_j(|z|)}(\mathbf{H}_{k,l}(B_{k,l}))}\leq C
 \end{equation}
 for all $p\in \mathbb{N}$. Since Einstein metrics satisfy an elliptic equation in harmonic coordinates. In fact in harmonic coordinates, the equation for the metric takes the form \cite{BKN}
 \begin{equation*}
 \triangle_g g_{kl}=-2Ric_{kl}+Q_{kl}(g,\partial g).
 \end{equation*}

For the completeness of the argument, we outline the proof of (\ref{intialregul}). Let $\omega_t$ be the frame on
$B(\sigma_{r_jR_0}(t);2\delta t) $ obtained by radial parallel transport of the dual frame of $\{ E_1, \cdots, E_4\}$ defined along $\sigma_{r_jR_0}(t)$ as above. By estimates for harmonic coordinates in Fact (2.9) in \cite{BKN}, we have ($\lambda<1$)
 \begin{equation*}
 |d\mathbf{H}_t-\omega_t|\leq C \left(\sup_{B(\sigma_{r_jR_0}(t);\lambda\delta t)}|Rm| \right) (\lambda\delta t)^2\quad \text{on } B(\sigma_{r_jR_0}(t);\lambda\delta t),
 \end{equation*}
and then by the curvature estimates in Proposition \ref{neckprop} (see (\ref{neckcurvturestmprop})), we get
 \begin{equation*}
 |d\mathbf{H}_t-\omega_t|\leq \eta_j(t)\quad \text{on } B(\sigma_{r_jR_0}(t);\lambda\delta t),
 \end{equation*}
 where $\mathbf{H}_t$ stand for the harmonic coordinates associated to the frame $\{ E_1, \cdots, E_4\}$.
  Now we denote $\Phi_{r_jR_0} (\theta, t\xi)$ by $\Phi_t(\theta, \xi)$.
  The last estimate shows that
   \begin{equation*}
| d(t^{-1}\mathbf{H}_t\circ\Phi_t)-t^{-1}\Phi_t^{*}\omega_t|\leq \eta_j(t), \quad \text{on } B((\theta,1);\lambda\delta )\subset S^3\times(1/2,3/2).
 \end{equation*}
 From the proof of Lemma \ref{neckmeytricflatlem}, it is easy to see that
\begin{equation*}
 |t^{-1} \Phi_t^{*}\omega_t(\theta,\xi)-(dz_1, \cdots, dz_4)|\leq \eta_j(t) ,
 \end{equation*}
where $(dz_1, \cdots, dz_4)$ form a basis of cotangent bundle of $\mathbb{R}^4$.
Therefore
   \begin{equation*}
| d(t^{-1}\mathbf{H}_t\circ\Phi_t)-(dz_1, \cdots, dz_4)|\leq \eta_j(t).
 \end{equation*}
 Since ($\mathbf{H}_t\circ\Phi_t)(\theta,1)=0\in \mathbb{R}^4$, we get
 \begin{equation}\label{keypinchingchart}
| (t^{-1}\mathbf{H}_t\circ\Phi_t)-(z_1,\cdots,z_4)|\leq \eta_j(t).
 \end{equation}
Thanks to the estimate
  \begin{equation*}
|(\Phi_t^{*}g_j)_{kl}-\delta_{kl}|\leq \eta_j(t) \quad \text{on } B((\theta,1);\lambda\delta )\subset S^3\times(1/2,3/2),
 \end{equation*}
induced from Lemma \ref{neckmeytricflatlem}, and notice that
 \begin{equation*}
(t^{-1}\mathbf{H}_t\circ\Phi_t)_{*}(\Phi_t^{*}g_j)=(t^{-1}\mathbf{H}_t)_{*}g_j,
 \end{equation*}
 we can get
\begin{equation*}
|t^{-1}(\mathbf{H}_t)_{*}g_j-\delta_{mn}|<\eta_j(t) \quad
\text{ on }t^{-1}\mathbf{H}_t(B(\sigma_{r_jR_0}(t);2\delta t)).
 \end{equation*}
Next by elliptic estimates \cite{DavidGilbarg2003Elliptic}, we obtain
\begin{equation*}
||t^{-1}(\mathbf{H}_t)_{*}g_j-\delta_{mn}||_{C^{1,\alpha}}<\eta_j(t) \quad
\text{ on }t^{-1}\mathbf{H}_t(B(\sigma_{r_jR_0}(t);2\delta t)).
 \end{equation*}
Therefore we get
(\ref{intialregul}).

Next by similar arguments as in Step 4 in the proof of Theorem (1.1) in \cite{BKN}, if the intersection $B_{k_1,l_1}\cap B_{k_2,l_2}$ is not empty, then
 \begin{equation}\label{coordrelationharmloc}
 ||\mathbf{H}_{k_1,l_1}\circ \mathbf{H}^{-1}_{k_2,l_2}-Id ||_{C^{2,\alpha}_{|z|\eta_j(|z|)}(\mathbf{H}_{k_2,l_2}(B_{k_1,l_1}\cap B_{k_2,l_2}))}\leq C,
 \end{equation}
 where $|z|$ is the absolute value of $\mathbf{H}_{k_2,l_2}$ coordinates functions.
 Indeed by the above construction, $t_l^{-1}\mathbf{H}_{k,l}\circ\Phi_{t_l}$ are close to the restriction of the standard coordinates functions on $\mathbb{R}^4$ (we denote it by $\mathbf{H}_{\infty}$) for all $k$ and $l$, since every flat bundle on $S^3$ is trivial.
 If $B_{k_1,l_1}$ intersects $B_{k_2,l_2}$ (assume $\frac{t_{l_1}}{t_{l_2}}=\beta^{-N}<1$ for some $\beta>1$ and $N\in\mathbb{N}$), then by (\ref{keypinchingchart}), we get
  \begin{eqnarray*}
 & &|t_{l_1}^{-1}\mathbf{H}_{k_1,l_1}\circ\Phi_{r_jR_0}(\theta,\xi t_{l_1})-t_{l_1}^{-1}\mathbf{H}_{k_2,l_2}\circ\Phi_{r_jR_0}(\theta,\xi t_{l_1}) |\\
&\leq& \eta_j(t_{l_1})+|\beta^{N}t_{l_2}^{-1}\mathbf{H}_{k_2,l_2}\circ\Phi_{r_jR_0}(\theta,\beta^{-N}\xi t_{l_2})-\mathbf{H}_{\infty}(\theta,\xi t_{l_1})|\\
 &\leq&\eta_j(t_{l_1})+\eta_j(t_{l_2})\leq C \eta_j(t_{l_1}).
 \end{eqnarray*}
 Thus (\ref{coordrelationharmloc}) can be derived by apriori estimates for harmonic functions (see Theorem \ref{schauderestneck}).
Since all components of $\mathbf{H}_{k_1,l_1}\circ \mathbf{H}^{-1}_{k_2,l_2}$ are harmonic functions with respect to $g_j$. And Moreover by applying Theorem \ref{schauderestneck} again, it actually holds for all $p\in\mathbb{N}$ that
 \begin{equation*}
 ||\mathbf{H}_{k_1,l_1}\circ \mathbf{H}^{-1}_{k_2,l_2}-Id ||_{C^{p+1,\alpha}_{|z|\eta_j(|z|)}\left(\mathbf{H}_{k_2,l_2}(B_{k_1,l_1}\cap B_{k_2,l_2})\right)}\leq C.
 \end{equation*}

 Now we take a partition of unity $\{\rho_{k,l}\}$ associated to the covering $\{B_{k,l}\}$ so that
 if $B_{k_1,l_1}$ intersects with  $B_{k_2,l_2}$ , then
 \begin{equation*}
 || \rho_{k_1,l_1}\circ \mathbf{H}^{-1}_{k_2,l_2}||_{C^{p+1,\alpha}_{|z|\eta_j(|z|)}\left(\mathbf{H}_{k_2,l_2}(B_{k_1,l_1}\cap B_{k_2,l_2})\right)}\leq C.
 \end{equation*}
 For such a construction, we refer to Page 325 of \cite{BKN}.

 Finally, we define a smooth map $ \Psi: A_{r_jR,\delta_1}(x_{a,j})\rightarrow \mathbb{R}^4$
 by
 \begin{equation*}
 \Psi(\cdot)=\sum_{k,l}\rho_{k,l}(\cdot)\mathbf{H}_{k,l}(\cdot).
 \end{equation*}
 Then we have
 \begin{eqnarray*}
& & || \Psi\circ \mathbf{H}^{-1}_{k_1,l_1}-Id||_{C^{p+1,\alpha}_{|z|\eta_j(|z|)}(\mathbf{H}_{k_1,l_1}(B_{k_1,l_1}))}\\
&=& || \sum_{k,l}\rho_{k,l}\circ \mathbf{H}^{-1}_{k_1,l_1}(\mathbf{H}_{k,l}\circ\mathbf{H}^{-1}_{k_1,l_1}-Id)||
_{C^{p+1,\alpha}_{|z|\eta_j(|z|)}(\mathbf{H}_{k_1,l_1}(B_{k_1,l_1}))}\leq C.
 \end{eqnarray*}
 Therefore $\Psi$ defines coordinates $(y)$ such that
 \begin{equation*}\label{almostflat1dege1}
\|g_{j,kl}(y)-\delta_{kl}\|_{C^{p,\alpha}_{\eta_j(|y|)}}<C.
\end{equation*}
The proof is finished.
\end{proof}

\begin{rem}\label{coordequal}\label{distequaval}
By the above construction, $\mathbf{H}_{\infty}$ in the proof is the restriction of the identity map, so
 \begin{equation*}
\|y-x\|_{C^{1}_{|y|\eta_j(|y|)}}<C,
\end{equation*}
in particular $ |x-y|<\eta_j(|x|)|x|$, where
 $|x(\cdot)|=r=d_{g_j}(\cdot,x_{a,j})$. In fact, the map $\Phi_{r_jR_0} (\theta, t)$ define the coordinates $(x)$,  $\Phi_{r_jR_0} (\theta, t\xi)=\Phi_t(\theta, \xi)$, $t_l^{-1}\mathbf{H}_{k,l}\circ\Phi_{t_l}\rightarrow \mathbf{H}_{\infty}$; the map $\Psi$ define the coordinate $(y)$, and $\Psi\circ\mathbf{H}^{-1}_{k,l}$ is closed to $Id$.
\end{rem}

 Now we give the proof of Proposition \ref{intodiffeo}. We follow the arguments in Step 1 of the proof for Theorem (1.1) in \cite{BKN}.
  \begin{proof}
 Let $J$ be a nontrivial Jacobi field along the geodesic $\sigma_{r_jR}(t)=\Phi_{r_jR} ( \theta, t)$ with $J(r_jR)\in T_{\tilde{\theta}}S_{r_jR}$ and $\frac{\partial}{\partial t}J(r_jR)=A_{r_jR}(J(r_jR)).$  The first key observation is that for large R, and for some small enough $\delta_1$, $J(t)$ never vanishes as long as $d_{g_j}(\sigma_{r_jR}(t),S_{r_jR})=t-r_jR$.
 Indeed, let $y(t)$ be the solution of the ordinary differential equation:
 \begin{eqnarray*}
& &\frac{d^2}{dt^2}y(t)+C\max\left(\frac{r_j^{\varepsilon_5}}{t^{\varepsilon_5+2}},
\frac{t^{-2+\varepsilon_5}}
 {r_\infty^{\varepsilon_5}}\right)y(t)=0, \quad \forall \,t\in[r_jR, \delta_1),\\
 & & y(r_jR)=1, \frac{dy}{dt}(r_jR)=(r_jR)^{-1}(1-a(j,R)),
 \end{eqnarray*}
 where a constant $C$ is taken so that the sectional curvature of $M_j$ at $\sigma_{r_jR}(t)$ is bounded from above by $$C\max \left(\frac{r_j^{\varepsilon_5}}{t^{2+\varepsilon_5}},\frac{t^{\varepsilon_5-2}}
 {r_\infty^{\varepsilon}}\right).$$
 See the curvature estimate in (\ref{neckcurvturestmprop}) and the meaning of $a(j,R)$ in (\ref{shape111}).

 If $y(t)$ is positive on $[r_jR,\delta_1]$, then $\frac{d}{dt}y(t)\leq(r_jR)^{-1}(1-a(j,R))$, since $\frac{d^2}{dt^2}y(t)\leq0$. Thus
 \begin{eqnarray*}
 & &\frac{d}{dt}y(t)=(r_jR)^{-1}(1-a(j,R))-C\int_{r_jR}^{t}
 \max\left(\frac{r_j^{\varepsilon_5}}{s^{2+\varepsilon_5}},\frac{s^{\varepsilon_5-2}}
 {r_\infty^{\varepsilon_5}}\right)y(s)ds\\
 &\leq&(r_jR)^{-1}(1-a(j,R))
 -C\int_{r_jR}^{t}\max \left(\frac{r_j^{\varepsilon_5}}{s^{2+\varepsilon_5}},\frac{s^{\varepsilon_5-2}}
 {r_\infty^{\varepsilon_5}}\right)\left(1+\frac{s}{r_jR}(1-a(j,R))\right)ds.
 \end{eqnarray*}

 Therefore direct computations show that if we choose $R$ large and $\delta_1$ small, $\frac{d}{dt}y(t)$ is positive on $[r_jR,\delta_1]$, and hence so is $y(t)$. Then applying the comparison theorem on Jacobi fields to our situation, we have
 \begin{equation*}
 y(t)\leq |J(t)|/|J(r_jR)|, \quad \text{as long as}\,\,d_{g_j}(\sigma_{r_jR}(t),S_{r_jR})=t-r_jR.
 \end{equation*}
 This shows that $S_{r_jR}$ has no focal points as long as $d_{g_j}(\sigma_{r_jR}(t),S_{r_jR})=t-r_jR$.

 Suppose $S_{r_jR}$ has the cut locus $\mathcal{C}_{r_jR}^{+}$ outside $B(x_{a,j},r_jR)$. Let $q$ be a point of $\mathcal{C}_{r_jR}^{+}$ which is closest to $S_{r_jR}$, i.e., $d_{g_j}(q,S_{r_jR})=d_{g_j}(\mathcal{C}_{r_jR}^{+},S_{r_jR})$, and $\sigma_{r_jR}: [r_jR,\delta_1]\rightarrow M_j $ a geodesic such that
 $d_{g_j}(\sigma_{r_jR}(t),S_{r_jR})=t-r_jR$ on $[r_jR,r_jR+t_0]$, $t_0=d_{g_j}(q,S_{r_jR})$ and $\sigma_{r_jR}(r_jR+t_0)=q.$

 Then since $S_{r_jR}$
 has no focal points along $\sigma_{r_jR}|_{[r_jR,r_jR+t_0]}$, it turns out from the standard theory of Riemannian geometry (Lemma 2 in P.226 of \cite{gromoll2006riemannsche}, Theorem (1.5) in Chap.4 of \cite{nakagawa1977taiiki}) that $\sigma_{r_jR}$ satisfies:
 \begin{equation*}
 d_{g_j}(\sigma_{r_jR}(t),S_{r_jR})=r_jR+2t_0-t, \quad \text{on} \,\,[r_jR+t_0,r_jR+2t_0].
 \end{equation*}

  Now we show that $\mathcal{C}_{r_jR}^{+}\bigcap A_{r_jR, \delta_1}(x_{a,j}) $ must be empty for large $R$ and small $\delta_1$. In fact if there are a divergence sequences $\{R_k\}$ and a sequence $\{\delta_k\}$ converging to 0 such that $\mathcal{C}_{r_jR_k}^{+}\bigcap A_{r_jR_k, \delta_k}(x_{a,j}) $ is not empty for any $k$, then we can take $t_k\leq \delta_k$ and
  $$ \sigma_{r_jR_k}:[r_jR_k,r_jR_k+2t_k]\rightarrow M_j$$
  with $t_k=d_{g_j}(\mathcal{C}_{r_jR_k}^{+},S_{r_jR_k})$ as above.

  We claim that
  $$\lim_{k\rightarrow\infty}\lim_{j\rightarrow\infty}\frac{t_k}{r_jR_k}=\infty.$$
  Then the metric space
  $$\left(\{x\in M_j:1/2\leq (t_k+r_jR_k)^{-1}d_{g_j}(x;x_{a,j})\leq 2\}, (t_k+r_jR_k)^{-1}d_{g_j}\right)$$
  converges to $\left(\{x\in\mathbb{R}^4: 1/2\leq |x|\leq 2\},  d_{std}\right)$ in $C^{1,\alpha}$-topology, as $j\rightarrow \infty$ firstly and then $k\rightarrow\infty$ (note that the curvature estimate in Proposition \ref{neckprop} is enough to ensure this convergence by Gromov compactness theorem, see for example Theorem 2.2 in \cite{anderson1989ricci}). And the geodesic $\sigma_{r_jR_k}$ converges to a line in $\mathbb{R}^4$ from $x_0$ with $|x_0|=1/2$ to $\acute{x}_0$ with $|\acute{x}_0|=1/2$ passing through a point $x$ with $|x|=1$. This is a contradiction.  So the proof of Proposition \ref{intodiffeo} is completed.

Finally, we prove the above claim. Since $ (M_j, r_j^{-2}g_j, x_{a,j})$ converges in pointed G-H topology to the Ricci flat ALE manifold $(M_a,h_a,p_a)$, and moreover $ (M_j\setminus B(x_{a,j};r_jR_k), r_j^{-2}g_j, x_{a,j})$ converges in $C^{\infty}$ topology to $(M_a\setminus B(p_a,R_k), h_a),$ $ \sigma_{r_jR_k}$ converges to a geodesic segment $\sigma_{R_k}:[R_k,R_k+2t_k]\rightarrow M_a$ satisfying
 \begin{equation*}
 d_{h_a}(\sigma_{R_k}(t),S_{R_k})=R_k+2t_k-t, \quad \text{on} \,\,[R_k+t_k,R_k+2t_k].
 \end{equation*}
By the arguments in subsection 2.2 of \cite{kasue1989convergence}, we have $t_k\geq CR_k^{1+\varepsilon}$ for some $\varepsilon>0$, hence the claim is valid.
 \end{proof}

\

\section{Neck analysis on degenerating neck regions}\label{degeneneckanalys}

In this section, we shall analyze the behavior of biharmonic maps $u_j$ defined on the degenerating neck regions $I_j^{\alpha}$ and then prove Theorem \ref{degenockenergy}. As in Section \ref{constrcoordinates}, for simplicity of notations, we shall ignore the effect of the fundamental group of $A_{r_jR, \delta_1}(x_{a,j})$ when we carry out the neck analysis. We remark again that the notation $\eta_j$ often appears in the rest of the paper, it means a function of order $\eta_j$ or it serves as a positive upper bound, it's precise expression may differ line by line.

To begin with, we shall firstly recall the following three key results given in the previous sections. In Proposition \ref{properneckc2}, we denote the set ${I}_j^{\alpha}$ to be the cylinder corresponding to $\bar{I}_j^{\alpha}= A_{e^{b_j^\alpha},e^{a_j^{\alpha+1}}}(x_{a,j})\subset M_j. $
\begin{itemize}
\item[($\mathfrak{K}_1$)] According to the bubble-neck decomposition in Proposition \ref{properneckc2}, $\bar{I}_j^{\alpha}\subset A_{r_jR, \delta_1}(x_{a,j})$, and
\begin{equation*}
\lim_{\delta\rightarrow 0}\lim_{R\rightarrow \infty}\lim_{j\rightarrow \infty}\omega(u_j,{I}_j^{\alpha})=0,
\end{equation*}
where
\begin{equation*}
\omega(u_j,{I}_j^{\alpha})\equiv \sup_{t\in \left[b_j^\alpha,a_j^{\alpha+1}-1 \right]}\int_{A_{e^{t},e^{t+1}}(x_{a,j})}|\nabla^2_{g_j} u_j|^2+|\nabla_{g_j} u_j|^4 dV_{g_j}.
\end{equation*}

\item[($\mathfrak{K}_2$)] By Theorem \ref{neckmeytricflat}, for $\delta>0$ small enough and for $R>0$ and $j>0$ sufficiently large, there exist coordinates $(x)$ and $(y)$ on $A_{r_jR,\delta}(x_{a,j})\subset M_j$ such that
\begin{eqnarray*}\label{almostflat1dege1}
& &\|g_{j,kl}( x)-\delta_{kl}\|_{C^{0}}<\eta_j(|x|),\\
& &|x(\cdot)|\equiv r(\cdot)=d_{g_j}(\cdot,x_{a,j})\quad \text{and} \quad \nabla_{\partial_r}\partial_r=0,
\end{eqnarray*}
and
\begin{equation*}\label{almostflat1dege1}
\|g_{j,kl}( y)-\delta_{kl}\|_{C^{4,\alpha}_{\eta_j(|y|)}}<C,
\end{equation*}
where $\eta_j(t)=O(r^{-\varepsilon_5}_\infty t^{\varepsilon_5}+ (r_jR_0)^{\varepsilon_5}t^{-\varepsilon_5}),$ $r_\infty$, $\varepsilon_5 $, $r_j$ are the same as in the curvature estimate in Proposition \ref{neckprop}, and $R_0>0$ is some fixed large number.
Moreover,
\begin{equation*}
\|y-x\|_{C^{1}_{|y|\eta_j(|y|)}}<C,
\end{equation*}
in particular $ |x-y|<\eta_j(|x|)|x|$.

\item[($\mathfrak{K}_3$)] The proof of the main Theorem \ref{degeymenergyintr} can be reduced to be the proof of Theorem \ref{degenockenergy} by the arguments in Section \ref{treesigleale}.

\end{itemize}

\subsection{Decay estimates for tangential part energy}\label{tangdeacydege}
The main task in this subsection is to show that similar decay estimates as in Lemma \ref{lemtangdecay} for biharmonic maps on those degenerating neck regions still hold. We remark that the decay along degenerating neck regions may not be as fast as in the case of a fixed domain shown in Lemma \ref{lemtangdecay}, however, it is already good enough for our purpose as we will see soon.

Suppose $b_j^{\alpha}=-l^{\alpha}_{j,b}L$, $a_j^{\alpha+1}=-l^{\alpha}_{j,a}L$ for some universal constant $L>0$.
Set
$$ A_l=A_{j;e^{-lL},e^{-(l-1)L}}=B(x_{a,j},e^{-(l-1)L})\setminus B(x_{a,j},e^{-lL}), \quad  l^{\alpha}_{j,a}\leq l<l^{\alpha}_{j,b}.$$
By ($\mathfrak{K}_2$) in the above, for some $\tilde{L}$,
$$A_l\subset \tilde{A}_{l-1}\cup \tilde{A}_l\cup \tilde{A}_{l+1},$$
where $\tilde{A}_l =\left\{q\in M_j; e^{-l\tilde{L}}\leq|y(q)| \leq e^{-(l-1)\tilde{L}}\right\}$.
So if we can prove the energy decay on $\tilde{A}_l$, then the energy decay on $A_l $ follows easily.

Set $s=|y|$,
\begin{equation*}
\tilde{F}_ l (v_j) =\int_{\tilde{A}_l}\frac{v_{j}^2}{|y|^4 } dy,
\end{equation*}
for any smooth functions $v_j$ on $\bar{I}_j^{\alpha}\subset M_j$, and
\begin{equation*}
(u_j)^*(s)=\frac{1}{|\partial B_s|}\int_{\partial B_s}u_j d\sigma,
\end{equation*}
where $(y)$ are the coordinates constructed in Theorem \ref{neckmeytricflat}, $dy$ is the volume form of the flat metric, and $d\sigma$ is volume form of the round sphere of radius $s$. Here we use the $(y)$ coordinates instead of the $(x)$ coordinates for the reason that the metric coefficients of $g_j$ in $(x)$ coordinates are not regular enough for our neck analysis carried out later.

Our aim is to prove the following decay estimates.
\begin{thm}\label{thrcirdegneck}
Let $v_j=u_j-(u_j)^*,$ and $s=|y|=e^t$,
then when $j$ and $R$ are large enough and $\delta$ is small enough, for $l^{\alpha}_{j,a}-1\leq l<l^{\alpha}_{j,b}+1$, we have that for some $\vartheta>0$,
\begin{eqnarray*}\label{tangentialdecaydege}
\tilde{F}_ l \leq C\varepsilon^2 \left(e^{-\vartheta(l-l^{\alpha}_{j,a})\tilde{L}}+e^{-\vartheta(l^{\alpha}_{j,b}-l)\tilde{L}}\right),
\end{eqnarray*}
and
\begin{eqnarray*}\label{tangentialdecaykappa}
& &\int_{(-l\tilde{L},-(l-1)\tilde{L})\times S^3}\left(|\triangle_{S^3}u_j|^2+|\nabla_{S^3}u_j|^4+|\partial_t \nabla_{S^3}u_j|^2+|\partial_t \triangle_{S^3}u_j|^2 \right)dt d\theta\nonumber\\
&\leq& C\varepsilon^2 \left(e^{-\vartheta(l-l^{\alpha}_{j,a})\tilde{L}}+e^{-\vartheta(l^{\alpha}_{j,b}-l)\tilde{L}}\right).
\end{eqnarray*}

\end{thm}

The fact ($\mathfrak{K}_2$) indicates that the $(y)$ coordinates and the $(x)$ coordinates in Theorem \ref{neckmeytricflat} are close to each other in $C^1$ norm, so by the following simple fact
 \begin{eqnarray*}
 & &\int_{A_l}
\frac{1}{r^4}\left((\triangle_{S^3}u_j)^2+|\nabla_{S^3}u_j|^4\right)
+\frac{1}{r^2}|\partial_r\nabla_{S^3}u_j|^2+\frac{1}{r^2}|\partial_r\triangle_{S^3}u_j|^2dx\\
&\leq&C\int_{\tilde{A}_{l-1}\cup \tilde{A}_l\cup \tilde{A}_{l+1}}
\frac{1}{s^4}\left((\triangle_{S^3}u_j)^2+|\nabla_{S^3}u_j|^4 \right)
+\frac{1}{s^2}|\partial_s\nabla_{S^3}u_j|^2+\frac{1}{s^2}|\partial_s\triangle_{S^3}u_j|^2dy,
\end{eqnarray*}
we have
\begin{thm} \label{tangdecaydege}
Set $r=|x|$, then for $r\in \left[e^{b_j^{\alpha}},e^{a_j^{\alpha+1}} \right]$,
\begin{eqnarray*}
& &\int_{A_{j;r,e^L r}}
\frac{1}{r^4}\left((\triangle_{S^3}u_j)^2+|\nabla_{S^3}u_j|^4\right)+\frac{1}{r^2}|\partial_r\nabla_{S^3}u_j|^2+\frac{1}{r^2}|\partial_r\triangle_{S^3}u_j|^2dx\\
&\leq&C\varepsilon^2 \left[\left(\frac{r}{e^{a_j^{\alpha+1}}}\right)^{\vartheta}+\left(\frac{e^{b_j^{\alpha}}}{r}\right)^{\vartheta} \right],
\end{eqnarray*}
where $A_{j;r,e^L r}:=\{r\leq |x|\leq e^{L}r\}\subset M_j$.
\end{thm}

The proof of the above theorem is parallel to the one for Lemma \ref{lemtangdecay}, we just outline it.

Notice that the fact ($\mathfrak{K}_1$) in the above ensure that
\begin{equation*}
\int_{\tilde{A}_l}|\nabla^2_{g_j}u_j|^2+|\nabla_{g_j} u_j|^4 dV_{g_j}<\varepsilon^2
\end{equation*}
for any given $\varepsilon>0$, as long as $j$ is large enough. And from the fact ($\mathfrak{K}_2$), we know that for any $\eta_0>0$, there exists $j_0>0$, $R_0>0$ and $\delta_0>0$ such that
 for all $\rho\in \left[e^{b_j^{\alpha}}, e^{a_j^{\alpha+1}} \right]$, the metric $g_j$  satisfies
 \begin{equation}\label{almostflat1degeta0}
\|g_{j,kl}(\rho \tilde{y})-\delta_{kl}\|_{C^{4}(B_2\setminus B_1)}<\eta_0
\end{equation}
 when $j>j_0$, $R>R_0$ and $\delta<\delta_0$, where $\tilde{y}=\rho^{-1}y$.
So that we can apply the $\varepsilon-$regularity Theorem \ref{smallenergythm} on the region $\tilde{A}_l$. In the rest of paper, we shall mainly focus on the region $I_j^{\alpha}$, we may not point this out when we have used the $\varepsilon-$regularity theorem.

\

Rewrite the biharmonic map equation as in the proof of Lemma \ref{lemtangdecay} (see (\ref{similar33})) to put the affect of the metric into the tension term (or error term), we get the following key lemma:

\begin{lem}
$v_j(y) = u_j -( u_j)^*$
is an $\eta_0$-approximate biharmonic function (w.r.t. the flat metric) in
the sense of Definition \ref{defappbiharm}, where $\eta_0>0$ is the constant in Theorem \ref{thrcircledegeneck}. That is
\begin{eqnarray*}
\triangle^2v_j(s,\theta)&=&a_1\nabla\triangle v_j+a_2\nabla v_j+a_3\nabla v_j+a_4 v_j \nonumber\\
&+&\frac{1}{|\partial B_s|}\int_{\partial B_s}b_1\nabla\triangle v_j+b_2\nabla v_j+b_3\nabla v_j+b_4 v_j d\sigma +h_j(y),
\end{eqnarray*}
where the coefficients satisfy the conditions in Definition \ref{defappbiharm}. Moreover the corresponding $h_j$ satisfies
\begin{equation*}\label{tensioncondit}
|||y|^{4(1-1/p)}h_j||^2_{L^p(\tilde{A}_l) }\leq \varepsilon^2\eta_{j}^2(|y|).
\end{equation*}
\end{lem}

Finally, argue as in Corollary \ref{threecirtotangdecay}, we can prove the decay estimates in Theorem \ref{thrcirdegneck}.

\begin{rem}\label{tangdecayorder}
From Theorem \ref{neckmeytricflat}, $\eta_j(t)=O(r^{-\varepsilon_5}_\infty t^{\varepsilon_5}+ (r_jR_0)^{\varepsilon_5}t^{-\varepsilon_5}),$ so we can take $\vartheta= \min(1, 2\varepsilon_5)$ in Theorem \ref{thrcirdegneck} by the arguments in Corollary \ref{threecirtotangdecay}, where $\varepsilon_5>0$ is the constant in the curvature estimate in Proposition \ref{neckprop}.
\end{rem}

\subsection{A new Poho\v{z}aev type identity}\label{pohodege}

In this subsection, we shall develop new Poho\v{z}aev type arguments that can be applied to degenerating neck regions. The advantage of these new arguments is that they depend much less on the coordinate functions than those for a fixed domain in Subsection \ref{Pohogenermetric}. The arguments here are much more complicated, due to the fact that we are not able not do analysis on the whole geodesic ball $B(x_{a,j},r)$ and we need to take care of the degeneration of the metrics.

In the sequel, for simplicity of notation, we shall write $u$ for $u_j$ if there is no confusion. Set $ A_{j;r,cr}\equiv B(x_{a,j},cr)\setminus B(x_{a,j},r).$ Let $(x)$ be the coordinates in Theorem \ref{neckmeytricflat}.
Then $|x|=r=d_{g_j}(\cdot,x_{a,j})$ and $ \nabla_{\partial_r}\partial_r=0$.
Multiplying the extrinsic biharmonic map equation by $r\partial_r u$, by using integration by parts as in Subsection \ref{Pohogenermetric}, we get that
\begin{eqnarray}\label{intrgratebypatrdege}
0&=&\int_{A_{j;r,cr}}(r\partial_r u)(\triangle^2_{g_j} u)dV_{g_j}\nonumber\\
&=&\int_{\partial A_{j;r,cr}}(r\partial_r u)(\partial_r\triangle_{g_j} u)d\sigma_{g_j}
-\int_{\partial A_{j;r,cr}}\partial_r(r\partial_r u)\triangle _{g_j} u d\sigma_{g_j}\\
& &+\int_{A_{j;r,cr}}\triangle_{g_j}(r\partial_r u)\triangle_{g_j} u dV_{g_j}.\nonumber
\end{eqnarray}

Direct calculations show that
\begin{eqnarray}\label{divformdege}
\int_{\partial A_{j;r,cr}}r\frac{|\triangle_{g_j} u|^2}{2}d\sigma_{g_j}&=&\int_{A_{j;r,cr}}
\text{div}_{g_j}\left(\frac{|\triangle_{g_j} u|^2}{2}r\partial_r\right)dV_{g_j}\\
&=&\int_{A_{j;r,cr}}r\partial_r(\triangle_{g_j} u) \triangle_{g_j} u dV_{g_j}+\int_{A_{j;r,cr}}\text{div}_{g_j}(r\partial_r)\frac{|\triangle_{g_j} u|^2}{2} dV_{g_j}.\nonumber
\end{eqnarray}

Combining (\ref{intrgratebypatrdege}) and (\ref{divformdege}) gives that
\begin{eqnarray*}
& &\int_{\partial A_{j;r,cr}}(r\partial_ru)(\partial_r\triangle_{g_j} u)+r\frac{|\triangle_{g_j} u|^2}{2} -\partial_r(r\partial_ru)\triangle _{g_j}u d\sigma_{g_j}\\
&=&\int_{A_{j;r,cr}}
\text{div}_{g_j}\left(\frac{|\triangle_{g_j} u|^2}{2}r\partial_r\right)dV_{g_j}-
\int_{A_{j;r,cr}}\triangle_{g_j}(r\partial_r u)\triangle_{g_j} u dV_{g_j}.
\end{eqnarray*}
Set
\begin{eqnarray*}
\int_{A_{j;r,cr}}
\text{div}_{g_j}\left(\frac{|\triangle_{g_j} u|^2}{2}r\partial_r\right)dV_{g_j}-
\int_{A_{j;r,cr}}\triangle_{g_j}(r\partial_r u)\triangle_{g_j} u dV_{g_j}=P(g_j, r),
\end{eqnarray*}
then it follows that
\begin{eqnarray}\label{pohoidentgenrl2}
\int_{\partial A_{j;r,cr}}(r\partial_ru)(\partial_r\triangle_{g_j} u)+r\frac{|\triangle_{g_j} u|^2}{2} -\partial_r(r\partial_ru)\triangle _{g_j}u d\sigma_{g_j}
=P(g_j,r).
\end{eqnarray}
Rewrite the term in the left hand side of (\ref{pohoidentgenrl2}) as in Subsection \ref{Pohogenermetric}
$$\square\equiv(r\partial_ru)(\partial_r\triangle_{g_j} u)+r\frac{|\triangle_{g_j} u|^2}{2} -\partial_r(r\partial_ru)\triangle _{g_j} u,$$
the identity (\ref{pohoidentgenrl2}) is equivalent to
\begin{equation}\label{pohozaevmodege}
\int_{\partial A_{j;r,cr}} r(\square_1+\square_2 )d\sigma_{g_j}=P({g_j},r).
\end{equation}
Here the forms of $\square_1$ and $\square_2$ are expressed in the same way as in (\ref{leftterm11}) and (\ref{leftterm22}) (replacing the metric $g$ with  $g_j$).

\begin{lem}\label{estimate-g}
Under the assumptions of ($\mathfrak{K}_1$) and ($\mathfrak{K}_2$), the error term satisfies
\begin{eqnarray}\label{pohoidentgenrl3}
| P(g_j,r)|
\leq\int_{A_{j;r,cr}}\frac{\eta_j(r)}{r}|\nabla_{g_j} u||\triangle_{g_j} u|+\eta_j(r)|\triangle_{g_j} u|^2dV_{g_j}.
\end{eqnarray}
\end{lem}

\begin{proof}
By (\ref{divformdege}),
\begin{eqnarray}\label{divformdege112}
& &P(g_j,r)=\int_{A_{j;r,cr}}r\partial_r(\triangle_{g_j} u) \triangle_{g_j} u +\text{div}_{g_j}(r\partial_r)\frac{|\triangle_{g_j} u|^2}{2} -\triangle_{g_j}(r\partial_r u)\triangle_{g_j} u dV_{g_j}.\nonumber
\end{eqnarray}

The idea is to show that
$$\triangle_{g_j}(r\partial_r u)\triangle_{g_j} u - r\partial_r(\triangle_{g_j} u) \triangle_{g_j} u $$
is almost equal to $\text{div}_{g_j}(r\partial_r)\frac{|\triangle_{g_j} u|^2}{2}$, which is almost equal to $2|\triangle_{g_j} u|^2$ by applying the following key lemma.

Set $ \widetilde{\partial B_r(x_{a,j})} $ be the geodesic sphere $\partial B_r(x_{a,j})$ with the metric $\tilde{g_r}\equiv\frac{g_r}{r^2}$, where $g_r$ is the metric of $\partial B_r(x_{a,j})$ induced from $(M_j,g_j)$.
Let $\tilde{m}$ be the mean curvature, $\tilde{\nabla}$ be the gradient operator and $\tilde{\triangle}$ be the associated Laplace operator of $ \widetilde{\partial B_r(x_{a,j})} $.
So
\begin{eqnarray*}
g_j=dr^2+g_r=r^2\tilde{g_r}
\end{eqnarray*}
and
\begin{eqnarray}\label{laplacepolardege}
\triangle_{g_j}u=\partial_r^{2}u+\frac{\tilde{m}}{r}\partial_ru+\frac{\tilde{\triangle}u}{r^2}.
\end{eqnarray}
To complete the proof of Lemma \ref{estimate-g}, we need the following lemma, which will be proved in Appendix \ref{pftwolem}.

\begin{lem}\label{almostflatconseqdege}
Let $\tilde{\nabla}$, $\tilde{\triangle}$ and $\tilde{m}$ be as above, then on $ B_{2\rho}(x_{a,j})\setminus B_{\rho/2}(x_{a,j})$, the followings hold:
\begin{itemize}
  \item [1)] \begin{eqnarray*}
& &\triangle_{g_j}r=(A(r,\theta))^{-1}\frac{\partial A(r,\theta)}{\partial r}=\frac{3+\eta_j(r)}{r},\\
& & {\rm div}_{g_j}(r\partial_r)-4=\eta_j,
\end{eqnarray*}
where $ A(r,\theta)drd\theta $ is the volume form;

  \item [2)] \begin{eqnarray*}
r\partial_r\triangle_{g_j}r=-\frac{3+\eta_j(r)}{r};
\end{eqnarray*}

  \item [3)] for any smooth function $f$, on $ \widetilde{\partial B_\rho(x_{a,j})} $,
\begin{eqnarray*}
  |\nabla_{S^3}f-\tilde{\nabla}f|& =&C\eta_j|\nabla_{S^3}f|,\\
  |\triangle_{S^3}f-\tilde{\triangle}f| &=&C\eta_j(|\triangle_{S^3}f|+|\nabla_{S^3}f|),
\end{eqnarray*}
where the norm is taken with respect to the standard metric on $S^3$ and it is equivalent to the norm defined by $\frac{g_r}{r^2}$,
and
$$|\partial_r\triangle_{S^3}f-\partial_r\tilde{\triangle}f| =C\eta_j(|\partial_r\triangle_{S^3}f|+|\partial_r\nabla_{S^3}f|)+C\frac{\eta_j}{\rho}(|\triangle_{S^3}f|+|\nabla_{S^3}f|);$$

  \item [4)] on $ \widetilde{\partial B_\rho(x_{a,j})} $, $$ |\tilde{m}-3|=\eta_j;$$

  \item [5)] on $ \widetilde{\partial B_\rho(x_{a,j})} $,
$$ |\partial_r\tilde{m}|=\frac{\eta_j}{\rho}.$$

\end{itemize}

\end{lem}

By direct computations, we have
\begin{eqnarray}\label{formaumid11}
&&\triangle_{g_j}(r\partial_r u)\triangle_{g_j} u   \\
&=& \left[(\triangle_{g_j}r)\partial_r u+2\partial_r^{2} u+r\triangle_{g_j}\partial_r u \right]\triangle_{g_j} u .\nonumber
\end{eqnarray}
By (\ref{laplacepolardege}), we have
\begin{eqnarray*}
r\triangle_{g_j}\partial_r u &=& r \left(\partial_r^{2}+\frac{\tilde{m}}{r}\partial_r+\frac{\tilde{\triangle}}{r^2}\right)\partial_r u\\
&=&r \left(\partial_r^{3}+\frac{\tilde{m}}{r}\partial_r^{2}+\frac{\tilde{\triangle}}{r^2}\partial_r \right)u,
\end{eqnarray*}
and
\begin{eqnarray*}
r\partial_r \triangle_{g_j}u &=& r\partial_r \left(\partial_r^{2}+\frac{\tilde{m}}{r}\partial_r
+\frac{\tilde{\triangle}}{r^2}\right) u\\
&=&r\left(\partial_r^{3}+\frac{\tilde{m}}{r}\partial_r^{2}+\frac{1}{r^2}\partial_r\tilde{\triangle}\right)u
+r\partial_r \left(\frac{\tilde{m}}{r}\right)\partial_ru-\frac{2}{r^2}\tilde{\triangle}u.
\end{eqnarray*}
Hence
\begin{eqnarray*}
& &r\triangle_{g_j}\partial_r u-r\partial_r \triangle_{g_j}u\\
&=&r\left(\frac{\tilde{\triangle}}{r^2}\partial_r-\frac{1}{r^2}\partial_r\tilde{\triangle}\right)u
-r\partial_r \left(\frac{\tilde{m}}{r}\right)\partial_ru+\frac{2}{r^2}\tilde{\triangle}u.
\end{eqnarray*}
So that from (\ref{formaumid11}), we get
\begin{eqnarray}\label{formaumid112}
& &\triangle_{g_j}(r\partial_r u)\triangle_{g_j} u \nonumber\\
&=&r\partial_r \triangle_{g_j}u\triangle_{g_j} u+\left[(\triangle_{g_j}r)\partial_r u+2\partial_r^{2} u-r\partial_r\left(\frac{\tilde{m}}{r}\right)\partial_ru+\frac{2}{r^2}\tilde{\triangle}u\right]\triangle_{g_j} u \\
& &+r\left(\frac{\tilde{\triangle}}{r^2}\partial_r-\frac{1}{r^2}\partial_r\tilde{\triangle}\right) u\triangle_{g_j} u.\nonumber
\end{eqnarray}

On $A_{j;1/2r,2r}$ with the scaled metric $\tilde{g_j}=\frac{g_j}{r^2}$ (we denote it by $\tilde{A}_{j;1/2,2}$), $\tilde{\partial_r}=r\partial_r$ has length one. At any point $q\in \widetilde{\partial B_r(x_{a,j})}$, we can choose $\tilde{\vartheta}_i,i=1,2,3$, such that $(\tilde{\partial_r},\tilde{\vartheta}_1,\tilde{\vartheta}_2,\tilde{\vartheta}_3)$ is a normal frame at $q\in \tilde{A}_{j;1/2,2}$. From Proposition \ref{neckprop}, on $A_{j;1/2r,2r}$ with the metric $g_j$,
\begin{equation*}
r^2|Rm(g_j)|\leq \eta_j(r),
\end{equation*}
so on $\tilde{A}_{j;1/2,2}$
\begin{equation*}
|Rm(\tilde{g_j})|\leq \eta_j(r).
\end{equation*}
Therefore, by using the Ricci identity, we have
\begin{equation*}
\left|\left[\tilde{\triangle}\tilde{\partial_r}-\tilde{\partial_r}\tilde{\triangle}\right]u \right|\leq \eta_j(r)|\tilde{\nabla}u|.
\end{equation*}
Hence,
\begin{equation}\label{ricciidentitydege}
\left|r \left[\frac{\tilde{\triangle}}{r^2}\partial_r-\frac{1}{r^2}\partial_r\tilde{\triangle}\right]u\right|\leq \frac{\eta_j(r)}{r}|\nabla u|.
\end{equation}

With the help of Lemma \ref{almostflatconseqdege}, from
(\ref{ricciidentitydege}) and (\ref{formaumid112}), we can obtain that
\begin{eqnarray}\label{formaumid1123}
& &|\triangle_{g_j}(r\partial_r u)\triangle_{g_j} u-r\partial_r \triangle_{g_j}u\triangle_{g_j} u-2|\triangle_{g_j} u|^2| \\
&\leq&
\frac{\eta_j(r)}{r}|\nabla_{g_j} u||\triangle_{g_j} u|.\nonumber
\end{eqnarray}
Indeed by Lemma \ref{almostflatconseqdege},
\begin{eqnarray*}
\left[(\triangle_{g_j}r)\partial_r u+2\partial_r^{2} u-r\partial_r\left(\frac{\tilde{m}}{r}\right)\partial_ru
+\frac{2}{r^2}\tilde{\triangle}u\right]=2\triangle_{g_j}u+\frac{\eta_j(r)}{r}\partial_ru.
\end{eqnarray*}

Again by using 1) of Lemma \ref{almostflatconseqdege}, (\ref{formaumid1123}) and (\ref{divformdege112}),
we have
\begin{eqnarray*}
 | P({g_j},r)|
\leq\big{|}\int_{A_{j;r,cr}}\frac{\eta_j(r)}{r}|\nabla_{g_j} u||\triangle_{g_j} u|+\eta_j(r)|\triangle_{g_j} u|^2dV_{g_j}\big{|}.
\end{eqnarray*}
The proof of Lemma \ref{estimate-g} is completed.
\end{proof}

By Proposition \ref{properneckc2} and Lemma \ref{almostflatconseqdege}, we can easily derive by using the same arguments as in Subsection \ref{Pohogenermetric} the following error estimate:
\begin{lem}\label{errorringhtj11}
On $A_{j;r,cr}\subset\bar{I}_j^{\alpha}= A_{e^{b_j^\alpha},e^{a_j^{\alpha+1}}}(x_{a,j})\subset A_{j;r_jR,\delta},$ for any $\varepsilon>0$, when $j$ and $R>0$ are large enough and $\delta>0$ is sufficiently small, we have
\begin{equation}\label{rigttermestdege}
 | P({g_j},r)| = \varepsilon^2\eta_j(r).
\end{equation}
\end{lem}

\begin{proof}
By applying (\romannumeral1) of Proposition \ref{properneckc2} (see also ($\mathfrak{K}_1$) in the beginning of this Section) which ensure that the energy  of $u_j$ on $A_{j;r,cr}$ is less than $\varepsilon^2$, and the H\"{o}lder inequalities,
$$\big{|}\int_{A_{j;r,cr}}\frac{\eta_j(r)}{r}|\nabla_{g_j} u||\triangle_{g_j} u|+\eta_j(r)|\triangle_{g_j} u|^2dV_{g_j}\big{|}  $$
is bounded by $\varepsilon^2\eta_j(r)$.
See also the proof of Lemma \ref{rigttermestlem11}.
\end{proof}

Now we are going to analyse the relation between the tangential part energy and the radial part energy.

Using integration by part as in (\ref{coaera444}), we get
\begin{eqnarray*}
& &\int_{A_{j;r,cr}}(\partial_r u)(\partial^{3}_r u)dV_{g_j}\\
&=&\int_{\partial({A_{j;r,cr}})}(\partial_r u)(\partial^{2}_r u)d\sigma_{g_j}
-\int_{A_{j;r,cr}}(\partial^{2}_r u)^2+\left(\frac{3}{r}+\frac{\eta_j}{r}\right)(\partial_r u)(\partial^{2}_r u) dV_{g_j}.
\end{eqnarray*}
We have used $\triangle_{g_j} r=\frac{3}{r}+\frac{\eta_j}{r} $ in the last formula, see 1) of Lemma \ref{almostflatconseqdege}. Recall that
\begin{eqnarray*}\label{leftterm2j}
\square_2
&=&(\partial_r u)(\partial^{3}_r u)+\frac{\tilde{m}-1}{r}(\partial_r u)(\partial^{2}_r u)+\left(\frac{\partial_r\tilde{m}}{r}-\frac{2\tilde{m}}{r^2}\right)(\partial_r u)^2\\
&+&\frac{1}{r^2}(\partial_r u)(\partial_r\tilde{\triangle} u)-\frac{\tilde{m}}{r^3}(\partial_r u)(\tilde{\triangle} u).\nonumber
\end{eqnarray*}
We get
\begin{eqnarray*}
& &\int_{A_{j;r,cr}}\square_2dV_{g_j}\\
&=&\int_{\partial({A_{j;r,cr}})}(\partial_r u)(\partial^{2}_r u)d\sigma_{g_j}
-\int_{A_{j;r,cr}}(\partial^{2}_r u)^2+\left(\frac{3}{r}+\frac{\eta_j}{r}\right)(\partial_r u)(\partial^{2}_r u) dV_{g_j}\\
& &+\int_{A_{j;r,cr}}\frac{\tilde{m}-1}{r}(\partial_r u)(\partial^{2}_r u)+\left(\frac{\partial_r\tilde{m}}{r}-\frac{2\tilde{m}}{r^2}\right)(\partial_r u)^2dV_{g_j}\\
& &+\int_{A_{j;r,cr}}\frac{1}{r^2}(\partial_r u)(\partial_r\tilde{\triangle} u)-\frac{\tilde{m}}{r^3}(\partial_r u)(\tilde{\triangle} u)dV_{g_j}.
\end{eqnarray*}

Next, by direct calculations using 4) and 5) of Lemma \ref{almostflatconseqdege}, we have
\begin{eqnarray*}
& &\int_{A_{j;r,cr}}\square_2dV_{g_j}\\
&=&\int_{\partial({A_{j;r,cr}})}(\partial_r u)(\partial^{2}_r u)d\sigma_{g_j}
+\int_{A_{j;r,cr}}\left(O(\frac{\eta_j}{r^2})-\frac{6+\eta_j}{r^2}\right)(\partial_r u)^2-(\partial^{2}_r u)^2dV_{g_j}\\
& &+\int_{A_{j;r,cr}}\frac{2+\eta_j}{r}(\partial_r u)(\partial^{2}_r u)-\left(\frac{3}{r}+\frac{\eta_j}{r}\right)(\partial_r u)(\partial^{2}_r u)dV_{g_j}\\
& &+\int_{A_{j;r,cr}}\frac{1}{r^2}(\partial_r u)(\partial_r\tilde{\triangle} u)-\frac{\tilde{m}}{r^3}(\partial_r u)(\tilde{\triangle} u)dV_{g_j}.
\end{eqnarray*}
Hence
\begin{eqnarray*}
& &\int_{A_{j;r,cr}}\square_2dV_{g_j}\\
&=&\int_{\partial(A_{j;r,cr})}(\partial_r u)(\partial^{2}_r u)d\sigma_{g_j}
+\int_{A_{j;r,cr}}\left(O(\frac{\eta_j}{r^2})-\frac{5}{r^2}\right)(\partial_r u)^2-(\partial^{2}_r u)^2dV_{g_j}\\
& &+\int_{A_{j;r,cr}}\left(-\frac{1}{2r}+\frac{\eta_j}{r}\right)\partial_r (\partial_r u)^2-\frac{1}{r^2}(\partial_r u)^2dV_{g_j}\\
& &+\int_{A_{j;r,cr}}\frac{1}{r^2}(\partial_r u)(\partial_r\tilde{\triangle} u)-\frac{\tilde{m}}{r^3}(\partial_r u)(\tilde{\triangle} u)dV_{g_j}.
\end{eqnarray*}

By the use of the following formula
\begin{eqnarray}\label{coarea22}
\int_{\partial({A_{j;r,cr}})}\frac{1}{2r}(\partial_r u)^2d\sigma_{g_j}
=\int_{A_{j;r,cr}}\frac{1}{2r}\partial_r (\partial_r u)^2+\frac{1}{r^2}(\partial_r u)^2+O(\frac{\eta_j}{r^2})(\partial_r u)^2dV_{g_j},
\end{eqnarray}
we have
\begin{eqnarray*}
& &\int_{A_{j;r,cr}}\square_2dV_{g_j}\nonumber\\
&=&\int_{\partial (A_{j;r,cr})}(\partial_r u)(\partial^{2}_r u)d\sigma_{g_j}
+\int_{A_{j;r,cr}}\left(O(\frac{\eta_j}{r^2})-\frac{5}{r^2}\right)(\partial_r u)^2-(\partial^{2}_r u)^2dV_{g_j}\\
& &-\int_{\partial({A_{j;r,cr}})}\frac{1}{2r}(\partial_r u)^2d\sigma_{g_j}+\int_{A_{j;r,cr}}\frac{\eta_j}{r}\partial_r (\partial_r u)^2dV_{g_j}\nonumber\\
& &+\int_{A_{j;r,cr}}\frac{1}{r^2}(\partial_r u)(\partial_r\tilde{\triangle} u)-\frac{\tilde{m}}{r^3}(\partial_r u)(\tilde{\triangle} u)dV_{g_j}.\nonumber
\end{eqnarray*}
The formula (\ref{coarea22}) can be deduced as in (\ref{coarea1}), except that
here we use $\triangle_{g_j}r=\frac{3}{r}+\frac{\eta_j}{r}$.

Putting the boundary terms together, we have
\begin{eqnarray}\label{leftterm3dege}
& &\int_{A_{j;r,cr}}\square_2dV_{g_j}\nonumber\\
&=&\int_{\partial({A_{j;r,cr}})}(\partial_r u)(\partial^{2}_r u)-\frac{1}{2r}(\partial_r u)^2d\sigma_{g_j}\nonumber\\
& &+\int_{A_{j;r,cr}}\left(O(\frac{\eta_j}{r^2})-\frac{5}{r^2}\right)(\partial_r u)^2-(\partial^{2}_r u)^2+\frac{\eta_j}{r}\partial_r (\partial_r u)^2dV_{g_j}\\
& &+\int_{A_{j;r,cr}}\frac{1}{r^2}(\partial_r u)(\partial_r\tilde{\triangle} u)-\frac{\tilde{m}}{r^3}(\partial_r u)(\tilde{\triangle} u)dV_{g_j}.\nonumber
\end{eqnarray}

Recall that
\begin{eqnarray*}\label{leftterm1j}
\square_1=\frac{-1}{2}(\partial^{2}_ru)^2+(\frac{\tilde{m}^2}{2r^2})(\partial_r u)^2
+\frac{1}{2r^4}(\tilde{\triangle}u)^2+(\frac{\tilde{m}}{r^3})(\partial_r u)(\tilde{\triangle}u)\nonumber.
\end{eqnarray*}

Set $$\mathcal{R}(u_j,r)=\frac{3}{2}(\partial^{2}_r u_j)^2+\left(\frac{10-\tilde{m}^2}{2r^2}+O(\frac{\eta_j}{r^2})\right)(\partial_r u_j)^2 $$
and
 $$ \mathcal{T}(u_j,r)=\frac{1}{2r^4}(\tilde{\triangle}u_j)^2
+\frac{1}{r^2}\left((\partial_r u_j)(\partial_r\tilde{\triangle}u_j)\right).$$
Note that by 4) of Lemma \ref{almostflatconseqdege}, $$\mathcal{R}(u_j,r)=\frac{3}{2}(\partial^{2}_r u_j)^2+\left(\frac{1}{2r^2}+O(\frac{\eta_j}{r^2})\right)(\partial_r u_j)^2 .$$

Hence we have
\begin{eqnarray}\label{pohoradictangendege}
& &\int_{A_{j;r,cr}}\square_2+\square_1dV_{g_j}\nonumber\\
& =&\int_{A_{j;r,cr}}\mathcal{T}(u_j,r)-\mathcal{R}(u_j,r)+\frac{\eta_j}{r}\partial_r (\partial_r u_j)^2 dV_{g_j}\\
& &+\int_{\partial({A_{j;r,cr}})}(\partial_r u_j)(\partial^{2}_r u_j)-\frac{1}{2r}(\partial_r u_j)^2d\sigma_{g_j}.\nonumber
\end{eqnarray}

From (\ref{pohozaevmodege}) we obtain that
\begin{eqnarray}\label{pohozaevmodege2}
c\int_{A_{j;cr,c^2r}}(\square_1+\square_2)dV_{g_j}-
\int_{A_{j;r,cr}}(\square_1+\square_2)dV_{g_j}=\int_{r}^{cr}\frac{P(g_j,s)}{s}ds.
\end{eqnarray}

By similar argument as in the proof of Lemma \ref{errorringhtj11}, we get
\begin{equation*}
\int_{A_{j;r,c^2r}}|\partial_r (\partial_r u_j)^2|dV_{g_j}\leq C\varepsilon^2 r,
\end{equation*}
it follows that
\begin{equation*}
\int_{A_{j;r,c^2r}}\frac{\eta_j}{r}|\partial_r (\partial_r u_j)^2|dV_{g_j}\leq\varepsilon^2\eta_j(r).
\end{equation*}
Therefore, by combining (\ref{pohozaevmodege2}), (\ref{pohoradictangendege}), (\ref{rigttermestdege}) (i.e., $P({g_j},r)=\varepsilon^2\eta_j(r)$) and the last inequality together, we have
\begin{eqnarray}\label{pohoradictangen3dege}
& &c\int_{A_{j;cr,c^2r}}\mathcal{R}(u_j,r) dV_{g_j}
-\int_{A_{j;r,cr}}\mathcal{R}(u_j,r) dV_{g_j}\nonumber\\
&=&c\int_{A_{j;cr,c^2r}}\mathcal{T}(u_j,r)dV_{g_j}
-\int_{A_{j;r,cr}}\mathcal{T}(u_j,r)dV_{g_j}\\
& &+c\int_{\partial({A_{j;cr,c^2r}})}((\partial_r u_j)(\partial^{2}_r u_j))-\frac{1}{2r}(\partial_r u_j)^2d\sigma_{g_j}\nonumber\\
& &-\int_{\partial({A_{j;r,cr}})}((\partial_r u_j)(\partial^{2}_r u_j))-\frac{1}{2r}(\partial_r u_j)^2d\sigma_{g_j}+\varepsilon^2\eta_j(r).\nonumber
\end{eqnarray}

Set $r_2(s)=e^{t_0+s}$ and $r_1(s)=e^{t_0-s}$, then similarly to (\ref{pohoradictangen4}), we get the following formula:
\begin{eqnarray}\label{pohoradictangen4dege}
& &c\int_{A_{j;cr_1,cr_2}}\mathcal{R}(u_j,r) dV_{g_j}
-\int_{A_{j;r_1,r_2}}\mathcal{R}(u_j,r) dV_{g_j}\nonumber\\
&=&c\int_{A_{j;cr_1,cr_2}}\mathcal{T}(u_j,r)dV_{g_j}
-\int_{A_{j;r_1,r_2}}\mathcal{T}(u_j,r)dV_{g_j}\\
& &+c\int_{\partial({A_{j;cr_1,cr_2}})}((\partial_r u_j)(\partial^{2}_r u_j))-\frac{1}{2r}(\partial_r u_j)^2d\sigma_{g_j}\nonumber\\
& &-\int_{\partial({A_{j;r_1,r_2}})}((\partial_r u_j)(\partial^{2}_r u_j))-\frac{1}{2r}(\partial_r u_j)^2d\sigma_{g_j}+\int_{r_1}^{r_2}\varepsilon^2\frac{\eta_j(r)}{r}dr.\nonumber
\end{eqnarray}
Since
$$\eta_j(t)=O(r^{-\varepsilon_5}_\infty t^{\varepsilon_5}+ (r_jR_0)^{\varepsilon_5}t^{-\varepsilon_5}),$$
the last term in the right hand side of (\ref{pohoradictangen4dege}) can be controlled by $$\varepsilon^2\max(\eta_j(r_1),\eta_j(r_2))\leq C\varepsilon^2(e^{\varepsilon_5(t_0-\log r_\infty)}
+e^{\varepsilon_5(\log(r_jR_0)-t_0)})e^{\varepsilon_5 s}.$$

\

\subsection{The energy identity and no neck property}
Now we are in a  position to prove the following:

\begin{thm}
Let $\bar{I}_j^{\alpha} $ be the annulus region in $M_j$ corresponding to $I_j^{\alpha}$ in Proposition \ref{properneckc2}, then (for the extrinsic case)
\begin{equation*}\label{energyidentityformula}
\lim_{\delta\rightarrow 0} \lim_{R\rightarrow \infty}\lim_{j \rightarrow \infty}\int_{\bar{I}_j^{\alpha}}|\nabla^2_{{g_j}}u_j|^2+|\nabla_{g_j} u_j|^4 dV_{g_j}=0
\end{equation*}
and
\begin{equation*}
\lim_{\delta\rightarrow 0} \lim_{R\rightarrow \infty}\lim_{j \rightarrow \infty}osc_{\bar{I}_j^{\alpha}}u_j=0.
\end{equation*}
\end{thm}

\begin{proof} Recall $\bar{I}_j^{\alpha}  = A_{b(j,\alpha) ,a(j,\alpha)}(x_{a,j}) $, where $e^{b_j^\alpha}\equiv b(j,\alpha) $, $a(j,\alpha)\equiv e^{a_j^{\alpha+1}} $.
If we set $c=e^L$, then from the decay estimates for the tangential part energy
\begin{eqnarray*}
& &\int_{A_{j;r,e^L r}}
\frac{1}{r^4}(|\triangle_{S^3}u_j|^2+|\nabla_{S^3}u_j|^4)+\frac{1}{r^2}|\partial_r\nabla_{S^3}u_j|^2
+\frac{1}{r^2}|\partial_r\triangle_{S^3}u_j|^2dx\\
&\leq&C\varepsilon^2 \left[\left(\frac{r}{a(j,\alpha)}\right)^{\vartheta}+\left(\frac{b(j,\alpha)}{r}\right)^{\vartheta}\right],
\end{eqnarray*}
in Theorem \ref{tangdecaydege}, we get by the Cauchy-Schwarz inequality and 3) of Lemma \ref{almostflatconseqdege} that for $r\in[b(j,\alpha) ,a(j,\alpha)]\subset [r_jR ,\delta]$,
\begin{eqnarray}\label{equaldecaydege123}
& &\int_{A_{j;r,e^Lr}}|\mathcal{T}(u_j,r)|dV_{g_{j}}\leq\tau\int_{A_{j;r,cr}}\frac{1}{r^2}(\partial_r u_j)^2 dV_{g_{j}} \nonumber\\
& &+C(\tau)\int_{A_{j;r,e^L r}}
\frac{1}{r^4}(|\tilde{\triangle}u_j|^2+|\tilde{\nabla}u_j|^4)+\frac{1}{r^2}|\partial_r\tilde{\nabla}u_j|^2
+\frac{1}{r^2}|\partial_r\tilde{\triangle}u_j|^2dV_{g_{j}}\\
&\leq&C(\tau)\varepsilon^2 \left[\left(\frac{r}{a(j,\alpha)}\right)^{\vartheta}
+\left(\frac{b(j,\alpha)}{r}\right)^{\vartheta}\right]+\tau\int_{A_{j;r,cr}}\frac{1}{r^2}(\partial_r u_j)^2 dV_{g_{j}}\nonumber,
\end{eqnarray}
where $\tau>0$ is a small number, and $C(\tau)$ is a bounded number depending on $\tau$.
Due to Remark \ref{tangdecayorder}, we can take $\vartheta=2\varepsilon_5>0$ (w.l.o.g. we assume that $2\varepsilon_5\leq1$). It is easy to see that the term $\tau\frac{1}{r^2}(\partial_r u_j)^2 $ in (\ref{equaldecaydege123}) can be absorbed into $\mathcal{R}(u_j,r)$ if we choose $\tau>0$ to be small enough.

Recall that $A_l=A_{j;e^{-lL},e^{-(l-1)L}}$, $b_j^{\alpha}=-l^{\alpha}_{j,b}L$, $a_j^{\alpha+1}=-l^{\alpha}_{j,a}L$ and $$\eta_j=O\left((\frac{r}{\delta_1})^{\varepsilon_5}+(\frac{r_jR}{r})^{\varepsilon_5}\right).$$
Set
$$\mathcal{I}_j^{\alpha}=(\bar{I}_j^{\alpha}\cup A_{l^{\alpha}_{j,a}-1})\setminus A_{l^{\alpha}_{j,b}+1}.$$
Due to (\romannumeral1) of Proposition \ref{properneckc2} (see also ($\mathfrak{K}_1$) in the beginning of this section), we can assume that
$$ \int_{A_{l^{\alpha}_{j,a}-1}\cup A_{l^{\alpha}_{j,b}+1}}|\nabla^2_{g_j} u_j|^2+|\nabla_{g_j} u_j|^4dV_{g_j}\leq \varepsilon^2<\varepsilon_0.$$
Recall that
$$\mathcal{R}(u_j,r)=\frac{3}{2}(\partial^{2}_r u_j)^2+\left(\frac{1}{2r^2}+O(\frac{\eta_j}{r^2})\right)(\partial_r u_j)^2,$$
and notice that the boundary terms can be cancelled out by each other, from (\ref{pohoradictangen3dege}) we get
\begin{eqnarray*}\label{pohoradictangen2}
& &(c-1)\int_{\bar{I}_j^{\alpha} }\frac{3}{2}(\partial^{2}_r u_j)^2+ \left(\frac{1}{2r^2}+O(\frac{\eta_j}{r^2})\right)(\partial_r u_j)^2 dV_{g_j}\\
&\leq &c\int_{\mathcal{I}_j^{\alpha} }\mathcal{R}(u_j,r) dV_{g_j}
-\int_{\bar{I}_j^{\alpha} }\mathcal{R}(u_j,r) dV_{g_j}+C\varepsilon^2\nonumber\\
&\leq&(c+1)\sum_{l^{\alpha}_{j,a}-1}^{l^{\alpha}_{j,b}+1}\int_{A_l}|\mathcal{T}(u_j,r)|dV_{g_j}
+C\varepsilon^2\sum_{l^{\alpha}_{j,a}}^{l^{\alpha}_{j,b}}\eta_j(e^{-lL})\\
& &+c\int_{\hat{\partial}(\mathcal{I}_j^{\alpha} )\cup\hat{\partial}(\bar{I}_j^{\alpha} )}|(\partial_r u_j)(\partial^{2}_r u_j)|+|\frac{1}{2r}(\partial_r u_j)^2|d\sigma_{g_j}+C\varepsilon^2.
\nonumber\\
\end{eqnarray*}
And then by (\ref{equaldecaydege123}), we have
\begin{eqnarray*}
& &\int_{\bar{I}_j^{\alpha} }\frac{3}{2}(\partial^{2}_r u_j)^2+(\frac{1}{2r^2})(\partial_r u_j)^2 dV_{g_j}\\
&\leq&C(c)C(\tau)\sum_{l^{\alpha}_{j,a}-1}^{l^{\alpha}_{j,b}+1}\varepsilon^2
\left[\left(\frac{e^{-lL}}{a(j,\alpha)}\right)^{\vartheta}
+\left(\frac{b(j,\alpha)}{e^{-lL}}\right)^{\vartheta}\right]+
\int_{\bar{I}_j^{\alpha}  }C(c)\left(\frac{\tau}{r^2}\right)(\partial_r u_j)^2 dV_{g_{j}}\\
& &+C(c)\int_{\hat{\partial}(\mathcal{I}_j^{\alpha} )\cup\hat{\partial}(\bar{I}_j^{\alpha} )}|(\partial_r u_j)(\partial^{2}_r u_j)|+|\frac{1}{2r}(\partial_r u_j)^2|d\sigma_{g_j}
+C(c)\varepsilon^2\sum_{l^{\alpha}_{j,a}}^{l^{\alpha}_{j,b}}\eta_j(e^{-lL})+C(c)\varepsilon^2.
\end{eqnarray*}
In the above estimate, $C(c)>0$ is some constant dependent on $c$.
Next, by taking $\tau$ small properly, we have
\begin{eqnarray*}\label{pohoradictangen2}
& &\int_{\bar{I}_j^{\alpha} }\frac{3}{2}(\partial^{2}_r u_j)^2+\frac{1}{2r^2}(\partial_r u_j)^2 dV_{g_j}\\
&\leq&
C\int_{\hat{\partial}(\mathcal{I}_j^{\alpha} )\cup\hat{\partial}(\bar{I}_j^{\alpha} )}|(\partial_r u_j)(\partial^{2}_r u_j)|+|\frac{1}{2r}(\partial_r u_j)^2|d\sigma_{g_j}
+C\varepsilon^2.
\end{eqnarray*}
Here we have used the notation $\hat{\partial}(B_R\setminus B_r)=\partial B_R \cup \partial B_r $.

Next by using the $\varepsilon$-regularity Theorem \ref{smallenergythm} and Sobolev embedding theorem, we get
\begin{equation*}
\int_{\hat{\partial}(\bar{I}_j^{\alpha})\cup\hat{\partial}(\bar{I}_j^{\alpha} )}|(\partial_r u_j)(\partial^{2}_r u_j)|+|\frac{1}{2r}(\partial_r u_j)^2|d\sigma_{g_j}\leq C\varepsilon^2.
\end{equation*}
Therefore, we have that for any $\varepsilon>0$, there are constants $R_0>0$, $\delta_0>0$ and $j_0>0$ depending only on $\varepsilon$, such that, when $j>j_0$, $R>R_0$ and $\delta<\delta_0$, if we take some $c\geq2$ in the above arguments, then
\begin{equation}\label{radialenergyestimatedege}
\int_{\bar{I}_j^{\alpha} }(\partial^{2}_r u_j)^2+\frac{1}{r^2}(\partial_r u_j)^2 dV_{g_j}\leq C\varepsilon^2.
\end{equation}
Combining the last inequality with (\ref{equaldecaydege123}), the energy identity can be proved.

Recall that $r_2(s)=e^{t_0+s}$,  $r_1(s)=e^{t_0-s}$, and $$\mathcal{R}(u_j,r)=\frac{3}{2}(\partial^{2}_r u_j)^2+\left(\frac{1+\eta_j(r)}{2r^2}\right)(\partial_r u_j)^2.$$
To prove the no neck property, we define
\begin{equation*}
F(s)=\int_{A_{j;r_1,r_2}}\mathcal{R}(u_j,r) dV_{g_j},
\end{equation*}
and
\begin{equation*}
G(s)=\int_{A_{j;e^Lr_1,e^Lr_2}}\mathcal{R}(u_j,r)dV_{g_j}.
\end{equation*}
It is easy to see that $F$ and $G$ are non-decreasing function of $s$, and they are uniformly bounded by $C\varepsilon^2$ (see (\ref{radialenergyestimatedege})). We remark that $b_j^{\alpha}< t_0<a_j^{\alpha+1}<0$, $s\in [0, \min(t_0-b_j^{\alpha},a_j^{\alpha+1}-t_0)]. $

First, we rewrite (\ref{pohoradictangen4dege}) to be
\begin{eqnarray}\label{pohoradictangen4dege11}
& &
\int_{\partial({A_{j;r_1,r_2}})} \left((\partial_r u_j)(\partial^{2}_r u_j)\right)-\frac{1}{2r}(\partial_r u_j)^2d\sigma_{g_j}-\int_{A_{j;r_1,r_2}}\mathcal{R}(u_j,r) dV_{g_j}\nonumber\\
&=&c\int_{A_{j;cr_1,cr_2}}\mathcal{T}(u_j,r)dV_{g_j}
-\int_{A_{j;r_1,r_2}}\mathcal{T}(u_j,r)dV_{g_j}\\
& &+c\int_{\partial({A_{j;cr_1,cr_2}})}((\partial_r u_j)(\partial^{2}_r u_j))-\frac{1}{2r}(\partial_r u_j)^2d\sigma_{g_j}\nonumber\\
& &-c\int_{A_{j;cr_1,cr_2}}\mathcal{R}(u_j,r) dV_{g_j}+\int_{r_1}^{r_2}\varepsilon^2\frac{\eta_j(r)}{r}dr.\nonumber
\end{eqnarray}
Then by noticing that $\frac{1}{\sqrt{3}}<\frac{2}{3}$,
\begin{equation*}
\partial_sF(s)=r_1(s)\int_{\partial B_{r_1}(x_{a,j})}\mathcal{R}(u_j,r) dV_{g_j}+r_2(s)\int_{\partial B_{r_2}(x_{a,j})}\mathcal{R}(u_j,r) dV_{g_j},
\end{equation*}
and
\begin{eqnarray*}
& &|\int_{\partial(B_{r_2}\setminus B_{r_1})}(\partial_r u)(\partial^{2}_r u)d\sigma |\\
&\leq &\frac{1}{\sqrt{3}}r_2\int_{\partial B_{r_2} }\frac{3}{2}(\partial^{2}_r u)^2
+(\frac{1}{2r^2})(\partial_r u)^2  d\sigma\\
& &+ \frac{1}{\sqrt{3}}r_1\int_{\partial B_{r_1} }\frac{3}{2}(\partial^{2}_r u)^2+(\frac{1}{2r^2})(\partial_r u)^2 d\sigma,
\end{eqnarray*}
we get from (\ref{pohoradictangen4dege11}) and (\ref{equaldecaydege123}) that (set $c=e^L$)
\begin{eqnarray*}
-\frac{2}{3}\partial_s F(s)-F(s)
&\leq& e^L(\frac{2}{3}\partial_s G(s)-G(s))\\
& &+C\varepsilon^2(e^{\varepsilon_5(t_0-\log r_\infty)}
+e^{\varepsilon_5(\log(r_jR_0)-t_0)})e^{\varepsilon_5 s}
+C\varepsilon^2(e^{\vartheta(t_0-a_j^{\alpha+1})}
+e^{\vartheta(b_j^{\alpha}-t_0)})e^{\vartheta s}.
\end{eqnarray*}

We remark here that the term involving $\tau$ in (\ref{equaldecaydege123}) and the error terms involving $\eta_j(r)$ in $G(s)$ and $F(s)$ do not have essential effects to the validity of the last inequality, as long as we set $\tau$ properly small.

Recall that $0<\vartheta=2\varepsilon_5\leq 1$. Multiplying $e ^{-\frac{3}{2}s}$ to both sides of the above inequality, and noticing that
$[b_j^{\alpha}, a_j^{\alpha+1}]\subset[\log(r_jR),\log\delta]$, we have
\begin{equation}\label{ineqalgeb1}
e^L\partial_s(e^{-\frac{3}{2}s}G)(s)+\partial_s(e^{-\frac{3}{2}s}F)(s)\geq  -3e^{-\frac{3}{2}s}F(s) -C\varepsilon^2(e^{\varepsilon_5(t_0-a_j^{\alpha+1})}+e^{\varepsilon_5(b_j^{\alpha}-t_0)})e^{(\varepsilon_5-\frac{3}{2})s}.
\end{equation}
Integrating
the above inequality (\ref{ineqalgeb1}) from $s = L $ to $s=2L$,
we have
\begin{eqnarray}\label{ineqalgeb12}
& &e^L(e^{-3L}G(2L)-e^{-\frac{3}{2}L}G(L)) +(e^{-3L}F(2L)-e^{-\frac{3}{2}L}F(L))\\
&\geq &-3\int_{L}^{2L}e^{-\frac{3}{2}s}F(s)ds-Ce^{(\varepsilon_5-\frac{3}{2})L}
\varepsilon^2(e^{\varepsilon_5(t_0-a_j^{\alpha+1})}+e^{\varepsilon_5(b_j^{\alpha}-t_0)}).\nonumber
\end{eqnarray}

Notice that on $[L,2L]$, $F(s)\leq F(2L)$, we have from (\ref{ineqalgeb12}) that
\begin{eqnarray*}
& &e^LG(L)+F(L)\\
&\leq & e^{-\frac{3}{2}L}(F(2L)+e^{L}G(2L))+3e^{\frac{3}{2}L}\int_{L}^{2L}e^{-\frac{3}{2}s}F(s)ds+Ce^{\varepsilon_5L}
\varepsilon^2(e^{\varepsilon_5(t_0-a_j^{\alpha+1})}
+e^{\varepsilon_5(b_j^{\alpha}-t_0)})\\
&\leq& e^{-\frac{3}{2}L}(F(2L)+e^{L}G(2L))+3LF(2L)+Ce^{\varepsilon_5L}
\varepsilon^2(e^{\varepsilon_5(t_0-a_j^{\alpha+1})}
+e^{\varepsilon_5(b_j^{\alpha}-t_0)}).
\end{eqnarray*}
So
\begin{eqnarray*}
e^LG(L)\leq e^{-\frac{1}{2}L} G(2L)+ (3L+1)F(2L)+Ce^{\varepsilon_5L}
\varepsilon^2(e^{\varepsilon_5(t_0-a_j^{\alpha+1})}
+e^{\varepsilon_5(b_j^{\alpha}-t_0)}).
\end{eqnarray*}

Thanks to the simple fact that $ G(2L), F(2L)\leq G(3L)$, if we choose $L>0$ such that $$ {\rm max} \left\{(3L+1)e^{-L},e^{-\frac{3}{2}L}\right \} \leq \frac{1}{2}e^{\frac{-L}{2}},$$
then
\begin{eqnarray*}
G(L)\leq e^{\frac{-L}{2}} G(3L)+Ce^{(\varepsilon_5-1)L}
\varepsilon^2(e^{\varepsilon_5(t_0-a_j^{\alpha+1})}
+e^{\varepsilon_5(b_j^{\alpha}-t_0)}).
\end{eqnarray*}

By applying an iteration argument, we get
\begin{eqnarray*}
G(L)\leq C
\varepsilon^2(e^{\frac{1}{6}(t_0-a_j^{\alpha+1})}
+e^{\frac{1}{6}(b_j^{\alpha}-t_0)})+C
\varepsilon^2(e^{\varepsilon_5(t_0-a_j^{\alpha+1})}
+e^{\varepsilon_5(b_j^{\alpha}-t_0)}),
\end{eqnarray*}
since $G(s)$ are uniformly bounded by $C\varepsilon^2$.

Finally, similarly as in Subsection \ref{Pohogenermetric}, the above radial energy decay together with the decay of the tangential part energy in (\ref{equaldecaydege123}) is enough to ensure the validity of the no neck property. The proof of Theorem \ref{energyidentityformula} is now completed. \end{proof}

\

\section{Intrinsic biharmonic maps}\label{secintrinsic}

In this section, we shall show that the energy identity and no neck property as stated in
Theorem \ref{noneckthmintr} and Theorem \ref{degeymenergyintr} hold also for intrinsic biharmonic maps.

Recall that the equations of extrinsic and intrinsic biharmonic maps differ only by the following term (see Section 7 of \cite{liu2016neck}):
\begin{eqnarray*}
\mathbf{I}(u)\equiv P(u)(B(u)(\nabla u,\nabla u) \nabla_uB(u)(\nabla u,\nabla u))+2B(u)(\nabla u,\nabla u)B(u)(\nabla u,\nabla P( u)).
\end{eqnarray*}
That is to say, the intrinsic biharmonic map equations take the form
 \begin{equation}\label{intrequ}
 \triangle^{2}_{g}u=\triangle_g u(B(u)(\nabla _g u,\nabla_g u))+2\nabla_g\cdot\langle\triangle_g u,\nabla_g(P(u))\rangle
 -\langle\triangle_g (P(u)),\triangle_g u\rangle+\mathbf{I}(u).
 \end{equation}
Here $B $ is the second fundamental
form of $N\subset \mathbb{R}^K$ and $P(u)$ is the projection to the tangent space $T _u N$.

In view of the equation \eqref{intrequ}, it is easy to see that the three circle type method developed for deriving the decay estimate of the tangential part energy for extrinsic biharmonic maps in both the non-degenerating case (see Subsection \ref{tangdeacynodege}) and the degenerating case (see Subsection \ref{tangdeacydege}) can be applied to intrinsic biharmonic maps by slightly modifying the arguments. Therefore, to handle the case of intrinsic biharmonic maps, it is sufficient to develop the necessary Poho\v{z}aev type arguments.

\subsection{Non-degenerating case}
We deal with the non-degenerating case. As in Subsection \ref{Pohogenermetric}, we multiply the term $\mathbf{I}(u)$ by $x^ke_k u$, notice that $x^ke_k u=r\partial_ru$ is a tangent vector, then
\begin{eqnarray*}
& & r\partial_r u\cdot\mathbf{I}(u)\\
&=& B(u)(\nabla u,\nabla u) \  \nabla_{r\partial_ru}B(u)(\nabla u,\nabla u)+2B(u)(\nabla u,\nabla u) \ B(u)(\nabla u,\nabla (r\partial_ru))\\
&=&r\partial_r \left (\frac{|B(u)(\nabla u,\nabla u)|^2}{2} \right )+2|B(u)(\nabla u,\nabla u)|^2\\
&=&\frac{1}{r^3}\partial_r \left[\frac{r^4}{2}|B(u)(\nabla u,\nabla u)|^2 \right].
\end{eqnarray*}

By applying similar argument as is done for (\ref{coaera444}), we get
\begin{eqnarray*}
& &\int_{ B(o,r)}\frac{1}{r^3}\frac{\partial} {\partial r}\left(\frac{r^4}{2}f\right) dV_g\\
&=&-\int_{ B(o,r)}\frac{\partial_r(\sqrt{g})}{r^3\sqrt{g}}\left(\frac{r^4}{2}f\right)dV_g+\int_{\partial B(o,r)}\frac{1}{r^3}\left(\frac{r^4}{2}f\right)d\sigma_g.
\end{eqnarray*}
Notice that
$$\frac{\partial_r \sqrt{g}}{\sqrt{g}}=\triangle_g r-\frac{3}{r}=O(r),$$
we have
\begin{eqnarray*}
\int_{B_r(p)}\frac{1}{r^3}\partial_r \left[\frac{r^4}{2}|B(u)(\nabla u,\nabla u)|^2 \right]dV_g=\int_{\partial B_r(p)}\frac{r}{2}|B(u)(\nabla u,\nabla u)|^2d\sigma_g+O(r^2),
\end{eqnarray*}
since
\begin{equation*}
\int_{B_r(p)}|B(u)(\nabla u,\nabla u)|^2dV_g\leq\int_{B_r(p)}C|\nabla u|^4dV_g\leq C.
\end{equation*}
It follows that
\begin{eqnarray*}
\int_{B_r(p)}x^ke_k u\cdot\mathbf{I}(u)dV_g=\int_{\partial B_r(p)}\frac{r}{2}|B(u)(\nabla u,\nabla u)|^2d\sigma_g+O(r^2).
\end{eqnarray*}

Therefore, the corresponding Poho\v{z}aev identity for intrinsic biharmonic maps over Riemannian manifold is
\begin{eqnarray*}\label{pohoidentgenrintr}
& &\int_{\partial B_r(p)}(r\partial_ru)(\partial_r\triangle_{g} u)+r\frac{|\triangle_g u|^2}{2} -\partial_r(r\partial_ru)\triangle _g u -\frac{r}{2}|B(u)(\nabla u,\nabla u)|^2 d\sigma_g\nonumber\\
&=&-\int_{B_r(p)}(\triangle_g x^k)e_ku (\triangle _g u )-2\int_{B_r(p)}(\nabla_gx^k)\nabla_g(e_k u)\triangle_g udV_g\\
& &+\int_{B_r(p)}(e_kx^k)\frac{|\triangle_g u|^2}{2} dV_g+\int_{B_r(p)}(x^k\text{div}(e_k))\frac{|\triangle_g u|^2}{2} dV_g\nonumber\\
& &-\int_{B_r(p)}Ric_g(\nabla_g u, x^ke_k)\triangle_g udV_g+O(r^2).\nonumber
\end{eqnarray*}
Or equivalently,
\begin{equation*}\label{pohozaevmodintr}
\int_{\partial B_r(p)} r(\square_1+\square_2 )-\frac{r}{2}|B(u)(\nabla u,\nabla u)|^2d\sigma_g=P_{I}(g,r),
\end{equation*}
where $ P_{I}(g,r)=P(g,r)+O(r^2).$ $\square_1,\,\square_2 $ and $P(g,r)$ are defined as in Subsection \ref{Pohogenermetric}.
Next, by using integration by parts as before, we have (recall that $P(g,r)=O(r^2) $)
\begin{eqnarray}\label{pohoradictangenintr}
& &\int_{B_{cr}\setminus B_r}\frac{3}{2}(\partial^{2}_r u)^2+\left(\frac{1}{2r^2}+O(1)\right)(\partial_r u)^2 dV_g\nonumber\\
&\leq&\int_{B_{cr}\setminus B_r}\frac{1}{2r^4}(\tilde{\triangle}u)^2+\frac{1}{r^2}(\partial_r u)(\partial_r\tilde{\triangle}u)dV_g\\
& &+\int_{\partial(B_{cr}\setminus B_r)}(\partial_r u)(\partial^{2}_r u)-\frac{1}{2r}(\partial_r u)^2d\sigma_g\nonumber\\
& &+\int_{B_{cr}\setminus B_r}\frac{1}{2}|B(u)(\nabla u,\nabla u)|^2dV_g+Cr.\nonumber
\end{eqnarray}

Now we decompose $B(u)(\nabla u,\nabla u)$ into the tangential part $\frac{1}{r^2}B(u)(\tilde{\nabla} u,\tilde{\nabla} u)$ and the radial part $B(u)(\partial_r u,\partial_r u)$, where $\tilde{\nabla}$ is the gradient operator of $ \widetilde{\partial B_r(p)} $ (see the definition in Subsection \ref{Pohogenermetric}).
By the decay estimate for the tangential part energy (using three circle type method), we know that the term $\frac{1}{r^4}|B(u)(\tilde{\nabla} u,\tilde{\nabla} u)|^2$ has no essential effect for the proof of the energy identity and no neck property.
On the other hand,  by $\varepsilon$-regularity Theorem \ref{smallenergythm} and Sobolev embedding theorem, we have
\begin{equation*}\label{pohozaevmod}
|B(u)(\partial_r u,\partial_r u)|^2\leq C\varepsilon^2\frac{1}{r^2}(\partial_r u)^2.
\end{equation*}
Thus, if we choose $\varepsilon>0$ small enough, this term can be absorbed into the left hand of (\ref{pohoradictangenintr}).

The rest of the proof can be completed by applying the same arguments as in the case of extrinsic biharmonic maps.

\subsection{Degenerating case}
 We denote the term $\mathbf{I}(u_j)$ by $\mathbf{I_j}(u)$. Here, for simplicity, we denote $\nabla_{g_j}$ by $\nabla_j$, and write $u$ for $u_j$.
As in Subsection \ref{pohodege}, we multiply the term $\mathbf{I_j}(u)$ by $r\partial_r u$, then we obtain that on $A_{j;r,cr}$,
\begin{eqnarray*}
& &r\partial_r u\cdot\mathbf{I_j}(u)\\
&=& B(u)(\nabla_j u,\nabla_j u) \nabla_{r\partial_r u}B(u)(\nabla_j u,\nabla_j u)+2B(u)(\nabla_j u,\nabla_j u)B(u)(\nabla_j u,\nabla_j (r\partial_r u))\\
&=&r\partial_r \left(\frac{|B(u)(\nabla_j u,\nabla_j u)|^2}{2}\right)+2|B(u)(\nabla_j u,\nabla_j u)|^2\\
&=&\frac{1}{r^3}\partial_r\left[\frac{r^4}{2}|B(u)(\nabla_j u,\nabla_j u)|^2\right].
\end{eqnarray*}

Recall that
$$ \frac{\partial_r\sqrt{g_j}}{\sqrt{g_j}}=\triangle_{g_j}r-\frac{3}{r}=\frac{\eta_j(r)}{r},$$ where $\sqrt{g_j}$ is the square root of determinant of the metric $g_j$ in the $(x)$ coordinates in Theorem \ref{neckmeytricflat}. Then again by using integration by parts as before, we have
\begin{eqnarray*}
\int_{A_{j;r,cr}}\frac{1}{r^3}\partial_r \left[\frac{r^4}{2}|B(u)(\nabla_j u,\nabla_j u)|^2 \right]dV_{g_j}=\int_{\partial A_{j;r,cr}}\frac{r}{2}|B(u)(\nabla_j u,\nabla_j u)|^2d\sigma_{g_j}+\varepsilon^2\eta_j(r).
\end{eqnarray*}
In the above, we have used (\romannumeral1) of Proposition \ref{properneckc2} which ensures that the energy  of $u_j$ on $A_{j;r,cr}$ is less than $\varepsilon^2$.

 Then by applying similar arguments as in the non-degenerating case, we have
\begin{equation*}\label{pohozaevmodintr}
\int_{\partial A_{j;r,cr}} r(\square_1+\square_2 )-\frac{r}{2}|B(u)(\nabla_j u,\nabla_j u)|^2d\sigma_{g_j}=P_{I}(g_j,r),
\end{equation*}
where $ P_{I}(g_j,r)=P(g_j,r)+\varepsilon^2\eta_j(r).$ The three terms $\square_1,\,\square_2 $ and $P(g_j,r)$ (recall that $P(g_j,r)\leq \varepsilon^2\eta_j(r) $) are defined as before.
So, we have
\begin{eqnarray*}\label{pohozaevmodegeintr}
& &c\int_{A_{j;cr,c^2r}}(\square_1+\square_2)-\frac{1}{2}|B(u)(\nabla_j u,\nabla_j u)|^2dV_{g_j}\\
& &-
\int_{A_{j;r,cr}}(\square_1+\square_2)-\frac{1}{2}|B(u)(\nabla_j u,\nabla_j u)|^2dV_{g_j}\nonumber\\
&=&\int_r^{cr}\frac{1}{s}P_{I}(g_j,s)ds.
\end{eqnarray*}

Finally, by applying the same arguments as in Subsection \ref{pohodege} for extrinsic biharmonic maps, it is easy to see that the additional term $|B(u)(\nabla _j u,\nabla _j u)|^2$ can be easily handled and has no essential influence to the rest of the proof.

\begin{rem}
In fact all results derived in this paper which are valid for extrinsic biharmonic maps are also true for intrinsic biharmonic maps.
\end{rem}

\

\appendix
\section{Proof of Lemma \ref{almostflatconseq} and Lemma \ref{almostflatconseqdege}}\label{pftwolem}

We prove Lemma \ref{almostflatconseqdege} in detail, and outline the proof for Lemma \ref{almostflatconseq}. Recall that
\begin{equation}\label{almostflat1dege1prf}
\|g_{j,kl}( x)-\delta_{kl}\|_{C^{0}}<\eta_j(|x|),
\end{equation}
\begin{equation}\label{almostflat2prf}
\|g_{j,kl}(\rho \tilde{y})-\delta_{kl}\|_{C^{4}(B_2\setminus B_1)}<\eta_j(\rho),
\end{equation}
where $\tilde{y}=\rho^{-1}y$, $(x)$ and $(y)$ are the two coordinates given in Theorem \ref{neckmeytricflat}.

\begin{proof}
Firstly, we prove 1) of Lemma \ref{almostflatconseq}. Since $e_k$ is parallel invariant along the radial geodesics, if we write $e_k=f^{k;l}\partial_{x^l}$, then $f^{k;l},l=1,\cdots,4,$ satisfy the following equations:
\begin{equation*}
\frac{df^{k;l}}{dr}+\Gamma_{mn}^{l}\frac{dx^m}{dr}f^{k;n}=0,\,\, f^{k;l}(0)=\delta_{kl}.
\end{equation*}
In normal coordinates on $B_r$,
\begin{equation*}
|g_{kl}-\delta_{kl}|<Cr^2,
\end{equation*}
so $|\Gamma_{mn}^{l}|<Cr.$ Therefore
\begin{equation*}
e_k=\partial_{x^k}+h^{kl}\partial_{x^l},\,\,h^{kl}=O(r^2).
\end{equation*}
Hence 1) of Lemma \ref{almostflatconseq} follows. And 2) of Lemma \ref{almostflatconseq} follows from direct calculations in local coordinates.

To prove 3) of Lemma \ref{almostflatconseq}, recall that $ \widetilde{\partial B_r} $ is the sphere $\partial B_r$ with metric $\tilde{g}_r\equiv\frac{g_r}{r^2}$, where $g_r$ is the metric of $\partial B_r$ induced from $(M,g)$. By the fact that the standard metric $g_{S^3}$ on $S^3$ and the metric $\tilde{g}_r$ satisfy
\begin{equation*}
||g_{S^3,kl}-\tilde{g}_{r,kl} ||_{C^4}\leq Cr^2,
\end{equation*}
we get by direct calculations the estimates in 3).

Next by standard theory in Riemannian geometry (see for example \cite{chavel2006riemannian}), for a fix metric $g$, in the geodesic normal coordinates $\tilde{m}=r\triangle_{g}r=3+O(r^2)$, so 4), 5) and 6) of Lemma \ref{almostflatconseq} follow immediately.

Now we prove 1), 2), 4), 5) of Lemma \ref{almostflatconseqdege}.
As in the proof of Lemma \ref{neckmeytricflatlem}, at a point $\tilde{\theta}\in S_{r_jR_0}=\partial B(x_{a,j},r_jR_0)$, we consider nontrivial Jacobi field $J$ along the geodesic $\sigma_{r_jR_0}(t) = \Phi_{r_jR_0} (\theta, t)$, with initial conditions $J(r_jR_0)\in T_{\tilde{\theta}}S_{r_jR_0}$, $\frac{\partial}{\partial t}J(r_jR_0)=A_{r_jR_0}(J(r_jR_0)).$
Let $\mathcal{A}(t,\theta)$ be the matrix solution of the differential equation in $T_{\tilde{\theta}}S_{r_jR_0}$,
\begin{equation*}
\frac{\partial^2}{\partial t^2}\mathcal{A}(t,\theta)+\mathcal{R}(t,\theta)\mathcal{A}(t,\theta)=0,
\end{equation*}
satisfying $\mathcal{A}(r_jR_0,\theta)=r_j R_0(u_1, u_2,u_3)$,
$\frac{\partial}{\partial t}\mathcal{A}(r_jR_0,\theta)=A_{r_jR_0}(\mathcal{A}(r_j R_0,\theta)).$ Here $\{u_1, u_2,u_3\}$ form an orthonormal basis for $T_{\tilde{\theta}}S_{r_jR_0}$, and $\mathcal{R}(t,\theta)$ is the composition of the curvature operator at $\sigma_{r_jR_0}(t)$ with the parallel translation along the geodesic.

As before due to the curvature estimate in Proposition \ref{neckprop}, we have
\begin{equation*}
|\mathcal{A}(t,\theta)-tId|=O(t\eta_j(t)),\,\, |\frac{\partial}{\partial t}\mathcal{A}(t,\theta)-Id|=\eta_j(t),\,\,\text{and}\,\, |\frac{\partial^2}{\partial t^2}\mathcal{A}(t,\theta)|=\frac{\eta_j(t)}{t}.
\end{equation*}
It is easy to see that
\begin{equation*}
(J_1(r),J_2(r),J_3(r))=\mathcal{A}(t,\theta)(u_1(r), u_2(r),u_3(r)),
\end{equation*}
where $u_l(r), l=1,2,3$ are the parallel invariant vector fields along the geodesics.
So \begin{equation*}
\det {\mathcal{A}(t,\theta)}=\frac{J_1(r)\wedge J_2(r)\wedge J_3(r)}{u_1(r)\wedge u_2(r) \wedge u_3(r)}.
\end{equation*}

By properties of Jacobi fields, we know that along $\sigma_{r_jR_0}(t)$,
$A(t,\theta)=v(\theta)\det {\mathcal{A}(t,\theta)}$
 for some constant $v(\theta)$. Here $A(t,\theta)$ is the square root of the determinant of the metric in the polar coordinates $(t,\theta)$. Or equivalently $dvol_{g_j}=A(t,\theta)drd\theta.$
Therefore
\begin{eqnarray*}
& &\triangle_{g_j}r=m(r)\\
&=&\sum_{l=1}^{3}\langle \nabla _{J_l} \partial_r,J_l(r)\rangle=\sum_{l=1}^{3}\langle \nabla _{\partial_r}J_l(r) ,J_l(r)\rangle=\text{tr}(\partial_r\mathcal{A}\mathcal{A}^{-1})\\
&=&\partial_r(\log(\det{\mathcal{A}(r,\theta)}))=(A(r,\theta))^{-1}\frac{\partial A(r,\theta)}{\partial r}\\
&=&\frac{3+\eta_j(r)}{r}
\end{eqnarray*}
and
\begin{eqnarray*}
r\partial_r\triangle_{g_j}r=-\frac{3+\eta_j(r)}{r},
\end{eqnarray*}
where $m(r)$ is the mean curvature of the geodesic sphere $\partial B(x_{a,j},r)$ with the induce metric from $(M_j, g_j)$. See \cite{wei2009comparison} and Lemma 2.3 in \cite{Zhu1997comparison} for similar computations.
By now we have proved 1), 2) of Lemma \ref{almostflatconseqdege}.
Next by $\tilde{m}=rm(r)=r\triangle_{g_j}r$, we obtain the desired estimates in 4) and 5) of Lemma \ref{almostflatconseqdege}.

Finally, we prove 3) of Lemma \ref{almostflatconseqdege}.
Notice that (\ref{almostflat1dege1prf}) and (\ref{almostflat2prf}) indicate that the metric $g_j$ is close to the flat metric, and that $\triangle_{g_j}u$, $\partial_ru$ and $\tilde{\triangle}u$ are independent on coordinate functions. It follows from direct computations that
\begin{eqnarray*}
|\triangle_{g_j}u-\triangle_{g_{euc}}u|\leq \eta_j(r)|\triangle_{g_{euc}}u|+\frac{\eta_j(r)}{r}|\nabla_{g_{euc}}u|_{g_{euc}},
\end{eqnarray*}
where $g_{euc}$ is the standard flat metric.
So from 1), 2), 4), 5) of Lemma \ref{almostflatconseqdege} and the following two formulas
\begin{eqnarray*}
\triangle_{g_{euc}}u=\partial_r^{2}u+\frac{3}{r}\partial_ru+\frac{\triangle_{S^3}u}{r^2}
\end{eqnarray*}
and
\begin{eqnarray*}
\triangle_{g_j}u=\partial_r^{2}u+\frac{\tilde{m}}{r}\partial_ru+\frac{\tilde{\triangle}u}{r^2},
\end{eqnarray*}
we can get the desired estimates.
\end{proof}

\

\section{Proof of Theorem \ref{schauderestneck}}\label{schauderest}
Firstly, we recall the classical Schauder's interior estimates for linear elliptic equations (\cite{douglis1955interior}, see also \cite{Mey}). Let $B$ be an open ball in $\mathbb{R}^n$. Let $d_x=d(x, \partial B)$ and $d_{xy}=\min(d_x,d_y)$. For functions $f$ in $C^p(B)$ define
\begin{equation*}
||f||_{C^{p,\alpha}_{\sigma}}=\sum_{j=0}^{p}\sup_{B}d_x^{j-\sigma}|D^{j}f|+\sup_{x\neq y}
d_{xy}^{p+\alpha-\sigma}\frac{|D^{p}f(x)-D^{p}f(y)|}{|x-y|^{\alpha}}.
\end{equation*}

\begin{thm}\label{schauderestbound}
Set \begin{equation*}
L(u)=\sum_{k,l}a_{kl}(x)\partial_k\partial_l u+\sum_{k}a_{k}(x)\partial_k u+a(x)u,
\end{equation*}
where $ a_{k,l}\in C^{p,\alpha}_{0}(A_{r_1,r_2})$, $ a_{k}\in C^{p,\alpha}_{-1}(A_{r_1,r_2})$, $ a\in C^{p,\alpha}_{-2}(A_{r_1,r_2})$ with norms less than $K_1>0$, and
\begin{equation*}
\sum_{k,l}a_{kl}(x)\xi_k\xi_l\geq K_2\sum_{k}\xi_k^{2} \quad (K_2>0).
\end{equation*}
Let $u$ be a solution of $L(u)=f$,
then
\begin{equation*}
||u ||_{C^{p+2,\alpha}_0}\leq {K}(||u ||_{C^{0}_0}+||f ||_{C^{p,\alpha}_{-2}}),
\end{equation*}
where ${K>0}$ depends only on $K_1$, $K_2$, $p$, $i$, $\alpha$, and $n$.
\end{thm}
We remark here that $K$ in the above does not depend on the radius of the ball.

Now we list some properties of the $ C^{p,\alpha}_{\eta}(A_{r_1,r_2})$ norm (see also Lemma 1 in \cite{Mey}):
\begin{lem} \label{normrelation} Let $B_x$ be the open ball $\{y;|x-y|<|x|/2 \}$ in $ A_{r_1,r_2}$, and $\hat{C}>0$ be as in Theorem \ref{schauderestneck}, then the two norms
$C^{p,\alpha}_{\eta}(A_{r_1,r_2})$ and $C^{p,\alpha}_{\sigma}(B_x)$ have the following relations:
\begin{equation*}
\eta^{-1}(|x|)||f ||_{C^{p,\alpha}_{\sigma}(B_x)}\leq 9 \hat{C}||f ||_{C^{p,\alpha}_{r^{\sigma}\eta}(A_{\frac{|x|}{2},\frac{3|x|}{2}})} \quad (\sigma\leq0)
\end{equation*}
and
\begin{equation*}
4^{-(p+\alpha-\sigma)}||f ||_{C^{p,\alpha}_{ r^{\sigma}\eta}(A_{r_1,r_2})}\leq
\sup_{B_x\subset A_{r_1,r_2}}\eta^{-1}(|x|)||f ||_{C^{p,\alpha}_{\sigma}(B_x)} \quad (\sigma\leq0).
\end{equation*}
\end{lem}
\begin{proof}
For $j\leq p$ and $\sigma\leq0$, let $y$ and $z$ be points in $B_x$, and $d_y=d(y ,\partial B_x)$,
\begin{equation*}
d_y^{-\sigma+j}|D^{j}f(y) |\leq \hat{C} \eta(|x|)\eta^{-1}(|y|)|y|^{-\sigma+j}|D^{j}f(y) |.
\end{equation*}

If $|y-z|<\max(1/4|y|,1/4|z|)$ ($y\neq z$), then
\begin{eqnarray*}
& & d_{yz}^{-\sigma+p+\alpha}\frac{|D^{p}f(y) -D^{p}f(z)|}{|y-z|^{\alpha}}\\
&\leq& \hat{C} \eta(|x|)\min(\eta^{-1}(y)|y|^{-\sigma+p+\alpha} ,\eta^{-1}(z)|z|^{-\sigma+p+\alpha})\frac{|D^{p}f(y) -D^{p}f(z)|}{|y-z|^{\alpha}}\\
&\leq& \hat{C} \eta(|x|)||u||_{C^{p,\alpha}_{r^{\sigma}\eta}(A_{\frac{|x|}{2},\frac{3|x|}{2}})}.
\end{eqnarray*}
If $|y-z|\geq\max(1/4|y|,1/4|z|)$, then
\begin{eqnarray*}
& & d_{yz}^{-\sigma+p+\alpha}\frac{|D^{p}f(y) -D^{p}f(z)|}{|y-z|^{\alpha}}\\
&\leq& 4^{\alpha}\hat{C} \eta(|x|)(\eta^{-1}(y)|y|^{-\sigma+p}|D^{p}f(y)| +\eta^{-1}(z)|z|^{-\sigma+p}|D^{p}f(z)|)\\
&\leq& 8\hat{C} \eta(|x|)||u||_{C^{p,\alpha}_{r^{\sigma}\eta}(A_{\frac{|x|}{2},\frac{3|x|}{2}})}.
\end{eqnarray*}

Therefore, the first inequality follows directly from the definition of the $C^{p,\alpha}_{\sigma}(B_x)$ norm.

The second inequality is a consequence of the following two inequalities:
\begin{eqnarray*}
\eta^{-1}(|x|)||f ||_{C^{p,\alpha}_{\sigma}(B_x)} &\geq & \sum_{j=0}^{p}\eta^{-1}(|x|)d_x^{j-\sigma}|D^{j}f|\\
&\geq& 2^{\sigma-p}\sum_{j=0}^{p}\eta^{-1}|x|^{j-\sigma}|D^{j}f|,
\end{eqnarray*}
 \begin{eqnarray*}
& &\eta^{-1}(|x|)||f ||_{C^{p,\alpha}_{\sigma}(B_x)} \\
&\geq&  \eta^{-1}(|x|)||f ||_{C^{p}_{\sigma}(B_x)}+4^{\sigma-p-\alpha}\eta^{-1}|x|^{j-\sigma+\alpha} \sup_{\{z;|x-z|<|x|/4 \}}\frac{|D^{p}f(x) -D^{p}f(z)|}{|x-z|^{\alpha}}.
\end{eqnarray*}

\end{proof}

Finally, we give the proof of Theorem \ref{schauderestneck}.

\begin{proof}
By applying the first inequality in Lemma \ref{normrelation} with $\eta\equiv 1$, it follows that
$$ ||a_{k,l}||_{C^{p,\alpha}_{0}(B_x)},\, || a_{k}||_{ C^{p,\alpha}_{-1}(B_x)},\,
|| a||_{ C^{p,\alpha}_{-2}(B_x)}\leq 9K_1.$$
From Theorem \ref{schauderestbound}, we have
\begin{equation*}
||u ||_{C^{p+2,\alpha}_0(B_x)}\leq {K}(||u ||_{C^{0}_0(B_x)}+||f ||_{C^{p,\alpha}_{-2}(B_x)}).
\end{equation*}
Using the first inequality in Lemma \ref{normrelation} again, we get
\begin{equation*}
\eta^{-1}(|x|)||u ||_{C^{p+2,\alpha}_0(B_x)}\leq 9\hat{C}{K}(||u ||_{C^{0}_{\eta}(A_{r_1,r_2})}+||f ||_{C^{p,\alpha}_{r^{-2}\eta}(A_{r_1,r_2})}),
\end{equation*}
for all $B_x \subset A_{r_1,r_2}$.

Hence the desired estimate is an immediate consequence of the above inequality and the second inequality in Lemma \ref{normrelation}.
\end{proof}

\

\section{Proof of Lemma \ref{lemrovsiguifty}}\label{prflem1}

In this section, we shall prove Lemma \ref{lemrovsiguifty} by modifying the arguments used in the proof of the case that the domain is $\mathbb{R}^{4}$ \cite{liu2016neck}. We remark that by  similar arguments we can prove the removable singularity Theorem \ref{remsingu}. For the definition of ALE manifolds and the regularity of the metrics of Ricci flat ALE manifolds, we refer to e.g. \cite{BKN, chenexpansion20201}. In the following, we shall only prove the case of extrinsic biharmonic maps and it is easy to see that the case of intrinsic biharmonic maps can be shown by sightly modifying the arguments as in Section \ref{secintrinsic}, see also Section 7 of \cite{liu2016neck}.

\begin{proof}
Since $(M_a,h_a)$ is an \textit{Asymptotically Locally Euclidean} (ALE) Ricci flat manifold (orbifold) of order 4 (see Theorem \ref{mainconvgethm}), by definition, there exists a compact subset $K$, a finite subgroup $ \Gamma\subset O(n)$ and a $C^{\infty} $  diffeomorphism
$\Phi :M_a\setminus K \rightarrow (\mathbb{R}^{4}\setminus B(0;R))/\Gamma$ such that $\varphi= \Phi^{-1}\circ proj $ satisfies
(where $proj$ is the natural projection of $\mathbb{R}^{4}$ to $\mathbb{R}^{4} /\Gamma_{i}$ )
$$ (\varphi^{\ast}h_a)_{ij}(z)=\delta_{ij}+O(|z|^{-4}),
\ \ \ \ \ \partial_{k}(\varphi^{\ast}h_a)_{ij}(z)=O(|z|^{-5}),$$
for $z,w\in \mathbb{R}^{n}\setminus B(0;R)$,
where $\partial_{k}$ denotes $\frac{\partial} {\partial{z_{k}}}$.
Since $M_a$ is Ricci flat, by elliptic estimates we obtain that
\begin{equation*}
\partial^{\alpha}(\varphi^{\ast}h_a)_{ij}(z)=O(|z|^{-(4+|\alpha|)})
\end{equation*}
for any $\alpha$ such that $|\alpha|>0$.

It is easy to see that we only need to prove the lemma for biharmonic maps on the domain $(\mathbb{R}^{4}\setminus B(0,R),\varphi^{\ast}h_a)$. For simplicity, we write $g$ for $\varphi^{\ast}h_a$.
By a scaling argument, we can assume that $R=1$. Set
\begin{equation*}
\mathbb{R}^4\setminus B_1=\bigcup_{l=0}^{\infty}A_l,
\end{equation*}
where $A_l=B_{e^{lL}}\setminus B_{e^{(l-1)L}}$.
Let $\rho=|z|\in[e^{(l-1)L},e^{lL}]$, then we have
\begin{equation}\label{almostflatinfty1}
\|g_{ij}(\rho y)-\delta_{ij}\|<\rho^{-4}
\end{equation}
and
\begin{equation}\label{almostflatinfty2}
\partial^{\alpha}(g_{ij}(\rho y)=O(|\rho|^{-(4+|\alpha|)}),
\end{equation}
where $z=\rho y$.
Namely, the metric after scaling to $ B _2 \setminus B_ 1$ is
close to the flat metric in $C^ k $ norm for any $k\in \mathbb{N}$.

Since
\begin{equation*}\label{energysmallnecki2}
\int_{\mathbb{R}^4\setminus B_1}|\nabla_{g}^2u|^2+|\nabla_g u|^4dV_g<\varepsilon_1,
\end{equation*}
we can assume that for any $\varepsilon>0$, there exists $l_0$ such that
\begin{equation*}\label{energysmallnecki2}
\int_{A_l}|\nabla_{g}^2u|^2+|\nabla_g u|^4<\varepsilon^2<\varepsilon_0
\end{equation*}
holds for $l>l_0.$

By direct computations as in the proof for Lemma \ref{lemtangdecay} (with the use of (\ref{almostflatinfty1}) and (\ref{almostflatinfty2})), the biharmonic map equation (\ref{extrequ3}) (for $u(\rho y)$) in local coordinates on $ B _2 \setminus B_ 1$) is equivalent to
\begin{eqnarray*}
\triangle^2 u=&&\alpha_1(u)\nabla \triangle  u \sharp \nabla u+\alpha_2(u)\nabla^{2} u \sharp \nabla^{2} u\nonumber\\
&+&\alpha_3(u)\nabla^{2} u \sharp \nabla u\sharp \nabla u+\alpha_4(u)\nabla u \sharp \nabla u\sharp \nabla u\sharp \nabla u+\tilde{\tau}(u).
\end{eqnarray*}
Here $\tilde{\tau}(u)$ satisfies
\begin{eqnarray}\label{tauestminfty}
||\tilde{\tau}(u)||_{L^p(B _2 \setminus B_ 1)}&\leq& C\varepsilon\eta
\end{eqnarray}
for any $p>1$, where $\eta=O(\rho^{-4})$.
After scaling back to the scale of $A_l$,
the equation (\ref{extrequ3}) for $u(z)$  in local coordinates becomes
\begin{eqnarray*}
\triangle^2 u&=&\alpha_1(u)\nabla \triangle  u \sharp \nabla u+\alpha_2(u)\nabla^{2} u \sharp \nabla^{2} u\nonumber\\
&+&\alpha_3(u)\nabla^{2} u \sharp \nabla u\sharp \nabla u+\alpha_4(u)\nabla u \sharp \nabla u\sharp \nabla u\sharp \nabla u+{\tau}(u).
\end{eqnarray*}
In the above ${\tau(u(z))}=\rho^{-4}\tilde{\tau}(u(\rho y)). $

As in the proof of Lemma \ref{lemtangdecay}, $v=u-u^*(r)$ ($|z|=r$) is an approximate biharmonic function in the sense of Definition \ref{defappbiharm}, and
\begin{equation}\label{twonormdecay}
F_ l (v)\leq C\varepsilon^2 e^{-lL}
\end{equation}
for all $l>l_0$.
In the above
 \begin{equation*}
u^*(r)=\frac{1}{|\partial B_r|}\int_{\partial B_r}u(r,\theta)d\sigma.
\end{equation*}

Then by (\ref{twonormdecay}), Lemma \ref{appbiharmest} and the Sobolev embedding theorem,
we have for any  $p+q\leq 3$,  $q\geq 1$ and $t\in [0,\infty)$,
\begin{equation}\label{tangentialestify}
|\partial^p_{t}\nabla^q_{S^3} u|(t,\theta)\leq C\varepsilon e^{-t/2}.
\end{equation}

The next step is to control the radial part energy by a Poho\v{z}aev type argument.
Recall that $u$ is an approximate biharmonic map if and only if $\triangle ^2 u-\tau(u)$ is normal to the tangent space $T_ u N$. On the other hand, $\partial_t u $ is a tangent vector at $u(\cdot) \in N$. Therefore,
\begin{equation*}
\int_{S^3}(\triangle ^2 u-\tau(u))\partial_t u d\theta=0
\end{equation*}
for all $t$, where $d\theta$ is the volume element of $S^3$. Then by the equation
\begin{equation*}
\triangle^2u=e^{-4t}\left((\partial_t+\triangle_{s^3})^2-4\partial^2_{t}\right)u
\end{equation*}
and some computations as in Section 5 of \cite{liu2016neck}, we can obtain
\begin{equation}\label{neckequ1}
\partial_t G(t)=\int_{S^3}e^{4t}\tau(u)\partial_t u d\theta,
\end{equation}
where
\begin{equation*}
G(t)=\int_{S^3}2\partial_{t}u\partial^3_{t}u-|\partial^2_{t}u|^2+|\triangle_{s^3}u|^2-2|\partial_t\nabla_{S^3}u|^2
-4|\partial_{t}u|^2d\theta.
\end{equation*}

Now as in Section 5 of \cite{liu2016neck}, we can prove that
\begin{equation*}
G(t)=\int_{S^3}2\partial_{t}u\partial^3_{t}u-|\partial^2_{t}u|^2+|\triangle_{s^3}u|^2-2|\partial_t\nabla_{S^3}u|^2
-4|\partial_{t}u|^2d\theta
\end{equation*}
is zero when $t\rightarrow \infty$. Indeed by the $\varepsilon $-regularity theorem (Theorem \ref{smallenergythm}) and the Sobolev embedding theorem, we have for $1\leq k\leq 3$,
\begin{equation*}
\max_{\mathbb{R}^4\setminus B_\rho }\rho^k|\nabla^k u|=o(1)
\end{equation*}
as $\rho\rightarrow \infty$.

Therefore we have
\begin{eqnarray}\label{neckequ2ify}
& &\partial_t \left(\int_{S^3}\partial_{t}u\partial^2_{t}ud\theta \right)-
\int_{S^3}3/2|\partial^2_{t}u|^2
+2|\partial_{t}u|^2d\theta\nonumber\\
&=&\Theta(t)+\int_{t}^{\infty}\int_{S^3}e^{4t}\tau(u)\partial_t u d\theta ds,
\end{eqnarray}
where
\begin{equation*}
\Theta(t)=\int_{S^3}-1/2|\triangle_{s^3}u|^2+|\partial_t\nabla_{S^3}u|^2d\theta.
\end{equation*}
By using (\ref{tangentialestify}), $|\Theta|\leq C\varepsilon^2 e^{-t}.$

For some fixed $t_0\in [1, \infty)$, set
\begin{equation*}
F(t)=\int_{t_0-t}^{t_0+t}\int_{S^3}3/2|\partial^2_{t}u|^2
+2|\partial_{t}u|^2d\theta dt.
\end{equation*}
F is defined for $0\leq t \leq t_0$.
 Integrating (\ref{neckequ2ify}) from $t_0 -t$
to $t_0+ t$, we obtain
\begin{eqnarray*}
F(t)&\leq& \frac{1}{2\sqrt{3}}\left(\int_{\{t_0-t\}\times S^3}+\int_{\{t_0+t\}\times S^3}\right)3/2|\partial^2_{t}u|^2+2|\partial_{t}u|^2 d\theta\nonumber\\
& &+\int_{t_0-t}^{t_0+t}|\Theta(s)|ds+\int_{t_0-t}^{t_0+t}\int_{s}^{\infty}\int_{S^3}e^{4t}|\tau(u)\partial_t u |d\theta d\tilde{s} ds
\end{eqnarray*}

By direct computations (transforming back to $z$-coordinates), we have
\begin{eqnarray*}
& &\int_{t_0-t}^{t_0+t}\int_{s}^{\infty}\int_{S^3}e^{4\tilde{s}}|\tau(u)\partial_t u |d\theta d\tilde{s} ds\nonumber\\
&\leq&\int_{e^{t_0-t}}^{e^{t_0+t}}\int_{\mathbb{R}^4\setminus B_{e^s}}|\tau(u)||\nabla u|dzdr\\
&\leq&C\int_{e^{t_0-t}}^{e^{t_0+t}}||\tau(u)||_{L^{4/3}(\mathbb{R}^4\setminus B_{e^{s}})}||\nabla u||_{L^{4}(\mathbb{R}^4\setminus B_{e^{s}})}ds\nonumber\\
&\leq&C\int_{e^{t_0-t}}^{e^{t_0+t}}||\tau(u)||_{L^{4/3}(\mathbb{R}^4\setminus B_{e^{s}})}ds.
\end{eqnarray*}
In the above inequality, we have used the assumption that $\varepsilon_1>0$ is small.

Now we claim the following

\

\noindent{\bf Claim B}:
$$||\tau(u)||_{L^{4/3}(\mathbb{R}^4\setminus B_{e^{s}})}\leq C\varepsilon e^{-s}.$$

\

By using the above estimate, we have
\begin{equation*}
F(t)\leq 1/2\partial_t F+ C\varepsilon^2 e^{t-t_0}.
\end{equation*}
Multiplying $-e^{-2t}$  to both sides of the above inequality and integrating from $t=1$
to $t=t_0-1$, we get
\begin{equation*}
F(1)\leq e^{-2(t_0-1)}F(t_0-1)+C\varepsilon^2e^{-t_0-1}.
\end{equation*}

Thanks to the above inequality and (\ref{tangentialestify}), we can obtain the desired decay estimate for the biharmonic energy. Next, we can argue exactly as in Section 5 of \cite{liu2016neck} to show a decay estimate
\begin{equation*}
\max_{|z|=e^{t_0}}|z||\nabla u|\leq C\varepsilon |z|^{-1/2},
\end{equation*}
 by which we can complete the proof of the lemma.
\end{proof}

Finally, we prove {\bf Claim B}.

\begin{proof}
Set $s=(l_s-1)L$. Notice that ${\tau}=\rho^{-4}\tilde{\tau}. $
Since $ \eta(\rho)=C\rho^{-4}$ in (\ref{tauestminfty}),
we have
\begin{eqnarray*}
\int_{\mathbb{R}^4\setminus B_{e^s}}|\tau(u)|^{4/3}dz &\leq& \sum_{l=l_s}^{\infty}\int_{A_l}|\tau(u)|^{4/3}dz\\
&\leq&\sum_{l=l_s}^{\infty}Ce^{lL\times(4-4\times 4/3)}\int_{B_1\setminus B_{e^{-L}}}|\tilde{\tau}(u)|^{4/3}dy\\
&\leq&\sum_{l=l_s}^{\infty}C\varepsilon e^{lL\times(4-4\times 4/3-4\times 4/3)}\\
&=&\sum_{l=l_s}^{\infty}C\varepsilon e^{(-20/3)lL}<C\varepsilon e^{-s}.
\end{eqnarray*}
\end{proof}

\

\end{document}